\documentclass[11pt]{amsart}

\usepackage{mathtools}
\usepackage{amsmath}
\usepackage{amssymb, stmaryrd, bm}
\usepackage{amsthm}
\usepackage{graphicx}
\usepackage{tikz}
\usepackage{xcolor}
\usepackage{enumitem}
\usepackage{algorithm}

\usepackage[foot]{amsaddr}
\usepackage[pdftex,colorlinks,backref=page,citecolor=blue]{hyperref}
\mathtoolsset{showonlyrefs}

\allowdisplaybreaks

\usepackage[margin=1.2in]{geometry}
\usepackage{bookmark}

\def\E{\mathbb{E}}
\def\var{\mathbb{Var}}

\def\bin{{\rm Bin}}

\def\bl{\boldsymbol{\lambda}}
\def\bz{\boldsymbol{\zeta}}

\def\N{\mathbb{N}}
\def\C{\mathbb{C}}
\def\E{\mathbb{E}}
\def\R{\mathbb{R}}
\def\P{\mathbb{P}}
\def\Z{\mathbb{Z}}

\def\eps{\varepsilon}

\def\cC{\mathcal {C}}

\def\cT{\mathcal {T}}

\def\1{\mathbf{1}}

\def\lam {\lambda}

\def\tce{t_c + \eps}
\def\tce2{t_c + \frac{\eps}{2}}
\def\ER{Erd\H{o}s-R\'{e}nyi }

\def\dist{\mathrm{dist}}

\def\var{\textup{var}}

\newtheorem*{theorem*}{Theorem}
\newtheorem{theorem}{Theorem}
\numberwithin{theorem}{section}
\newtheorem{lemma}[theorem]{Lemma}
\newtheorem{cor}[theorem]{Corollary}
\newtheorem{defn}[theorem]{Definition}
\newtheorem*{defn*}{Definition}

\newtheorem*{prop*}{Proposition}
\newtheorem{conj}{Conjecture}
\newtheorem*{conj*}{Conjecture}
\newtheorem{claim}[theorem]{Claim}
\newtheorem{question}{Question}

\newtheorem*{fact*}{Fact}

\newtheorem{condition}{Condition}
\newtheorem{definition}{Definition}
\newtheorem{remark}[theorem]{Remark}

\numberwithin{equation}{section}

\subjclass[2020]{05C80, 05C30, 60F10}

\begin{document}
\title{Lower tails for triangles  inside the critical window}

\author{Matthew Jenssen}
\author{Will Perkins}
\author{Aditya Potukuchi}
\author{Michael Simkin}

\address{King's College London, Department of Mathematics}
\email{matthew.jenssen@kcl.ac.uk}

\address{Georgia Institute of Technology, School of Computer Science}
\email{math@willperkins.org}

\address{York University, Department of Electrical Engineering and Computer Science}
\email{apotu@yorku.ca}

\address{Massachusetts Institute of Technology, Department of Mathematics}
\email{msimkin@mit.edu}

\date{\today}

\begin{abstract}  
We study the probability that the random graph $G(n,p)$ is triangle-free.  When $p =o(n^{-1/2})$ or $p = \omega(n^{-1/2})$ the asymptotics of the logarithm of this probability are known via Janson's inequality  in the former case and via regularity or hypergraph container methods in the latter case.  We prove for the first time an asymptotic formula for the logarithm of this probability when $p = c n^{-1/2}$ for $c$ a sufficiently small constant. 

More generally, we study  lower-tail large deviations  for triangles in  random graphs: the probability that $G(n,p)$ has at most $\eta$ times its expected number of triangles,  when $p = c n^{-1/2}$ for $c$  and  $\eta \in [0,1)$ constant.  Our results apply for all $c$ if $\eta \ge .4993$ and for $c$  small enough otherwise. 

For $\eta$ small (including the case of triangle-freeness), we prove that a phase transition  occurs as $c$ varies, in the sense of a non-analyticity of the rate function, while for $\eta \ge .4993$ we prove that no phase transition occurs. 

On the other hand for the random graph $G(n,m)$, with $m = b n^{3/2}$, we show that a phase transition occurs in the lower-tail problem for triangles  as $b$ varies for \emph{every} $\eta \in [0,1)$.

Our method involves ingredients from algorithms and statistical physics including the cluster expansion and concentration inequalities for contractive Markov chains. 
\end{abstract}

\maketitle

\section{Introduction}
\label{secIntro}

How rare is it to see atypical  subgraph counts in the \ER random graph $G(n,p)$?    This question of large deviations in random graphs has a rich history in both combinatorics and probability theory.  In combinatorics the question is closely related to asymptotic enumeration of graph classes. Good bounds on large deviation probabilities are also central tools in probabilistic combinatorics and the analysis of algorithms.  In probability theory, this question  is a prototypical example of the theory of non-linear large deviations and a wealth of new techniques have been developed in the area in the last decade.

Both the known results and the techniques applied to the problem depend on several factors, including the subgraph of interest, the edge density $p$ of the random graph, and whether we aim to understand the lower tail (fewer copies of the subgraph than expected) or the upper tail (more copies).

In the dense regime, with $p$ constant, Chatterjee and Varadhan~\cite{chatterjee2011large} proved a very general large deviation principle using Szemer\'{e}di's regularity lemma, which reduces determining the first-order asymptotics of the logarithm of a large deviation probability (that is, determining the \emph{rate function}) to solving an optimization problem over the space of graphons, limiting objects of sequences of dense graphs. Further reductions to variational problems were proved for $p$ tending $0$ sufficiently slowly with $n$~\cite{chatterjee2016nonlinear,eldan2018gaussian,cook2020large,augeri2020nonlinear,kozma2023lower}, as well as extensions to hypergraphs~\cite{cook2024regularity}. Complementing these results are papers that solve the variational problem for a given graph, $p$, and desired deviation~\cite{lubetzky2017variational,zhao2017lower,bhattacharya2020upper,liu2021upper,bhattacharya2017upper}.

On the other hand, for much smaller $p$, near the threshold for the appearance of a single copy of $H$, the number of copies of a (strictly balanced) subgraph $H$ has an approximately Poisson distribution~\cite{erdos1960evolution,bollobas1981threshold}. One might expect Poisson-like large deviation behavior to persist for larger $p$ and Janson's inequality~\cite{janson1987uczak,janson1990poisson} makes this precise for  lower-tail probabilities.

We briefly survey what is known in the special case where $H$ is a triangle (or more generally a clique).  For the lower tail, Janson,  {\L}uczak, and Ruci\'nski~\cite{janson1987uczak} (using Janson's inequality) determined the rate function up to  constant factors. For the upper tail (`the infamous upper tail'~\cite{janson2002infamous}) determining the rate function up to constants took many years to resolve~\cite{kim2000concentration,janson2004upper,chatterjee2012missing,demarco2012tight}, though now  the rate function for the upper-tail has been determined for all $p$, culminating in the work of Harel, Mousset, and Samotij~\cite{harel2022upper}.

For the lower tail however, a complete understanding of  the rate function remains elusive for two reasons. First, the variational problem for lower tails in dense graphs seems much more challenging than that for upper tails~\cite{lubetzky2017variational,zhao2017lower}. Second, the range of applicability of the  main methods for estimating lower tails, namely Janson's inequality on one hand, and the variational approach~\cite{kozma2023lower} or container and regularity-based methods~\cite{luczak2000triangle,balogh2018method} on the other,  leave a gap, in what we call the `critical regime' below. The main contribution of this paper is to address  the rate function inside this critical regime.

Here we focus on the lower-tail problem for triangles.  That is, we want to estimate the probability
\begin{equation}\label{eqLowerTailDef}
    \mathbb{P}_p(X \le \eta \E X)
\end{equation}
where $X$ is the number of triangles in the random graph $G(n,p)$, the probability and expectation are with respect to $G(n,p)$, $\eta \in [0,1)$ is fixed, and $p=p(n)$.  In particular, we aim for first-order asymptotics of the logarithm of the lower-tail probability, $\log \mathbb{P}_p(X \le \eta \E X)$; that is, computing the large deviation rate function of the lower-tail event.

The special case $\eta =0$, that of triangle-freeness, is perhaps the most studied and most influential in combinatorics.
Through the work of Erd\H{o}s, Kleitman, and Rothschild~\cite{erdosasymptotic}, Janson,  {\L}uczak, and Ruci\'nski~\cite{janson1987uczak}, Pr{\"o}mel and Steger~\cite{promel1996asymptotic}, and  {\L}uczak~\cite{luczak2000triangle}, the asymptotics of $\log \P_p(X=0)$ are known when $p$ is either much smaller or much larger than $n^{-1/2}$.  When $p = o(n^{-1/2})$, $-\log \P_p(X=0) = (1+o(1))\binom{n}{3}p^3$; that is, the lower tail exhibits Poisson-like behavior.  On the other hand, when $p =\omega(n^{-1/2})$, $\log \P_p(X=0) =  (1+o(1))\frac{n^2 \log(1-p)}{4}$; that is,  the probability of triangle-freeness matches the probability $G(n,p)$ is bipartite. 
It is worth noting that the methods used to prove these two results are very different. The first is proved using Janson's Inequality. In fact this powerful tool was first introduced in~\cite{janson1987uczak} to address this problem. The second can be proved using (a sparse version of) the regularity lemma~\cite{luczak2000triangle} or using the method of hypergraph containers~\cite{balogh2018method}.

In the intermediate (or `critical') regime, when $p = \Theta(n^{-1/2})$ it is known via~\cite{janson1987uczak} that $-\log \P_p(X=0) = \Theta(n^{3/2})$ but the leading constant is not known.   In the language of large deviation theory, the problem is to determine the rate function
\begin{equation}
 \varphi(c)=   \lim_{n \to \infty} \frac{1}{n^{3/2}} \log \P_p[ X = 0]
\end{equation}
where $p = c/\sqrt{n}$. Even establishing existence of the limit is open, see Question~\ref{QLimitExist} below.  

The story is somewhat similar for the general lower-tail problem for triangles.  For $\eta \in [0,1)$ fixed, when $p = o(n^{-1/2})$, $ -\log \P_p( X \le \eta \E X)=(1+o(1)) (1- \eta +\eta \log \eta)\binom{n}{3}p^3$~\cite{janson1990poisson,janson2016lower}, again matching Poisson behavior.  For larger $p$, the answer is more complicated and not completely understood.  Following~\cite{chatterjee2011large,chatterjee2016nonlinear,eldan2018gaussian}, Kozma and Samotij~\cite{kozma2023lower} established that a variational problem determines the lower-tail rate function when $p = \omega(n^{-1/2})$.  Zhao~\cite{zhao2017lower} solved some cases of this  variational problem\footnote{The variational problem is stated in~\cite{zhao2017lower} in the language of graphons, as formulated by Chatterjee and Varadhan; this is equivalent to a variational problem on random graphs on $n$ vertices with independent edges as stated in~\cite{chatterjee2016nonlinear,kozma2023lower}}, establishing, for instance,  that if $\eta > .1012$ and $p = p(n) \to 0$, then the solution to the variational problem  is a constant (and the problem is said to exhibit `replica symmetry').

As in the triangle-free case, in the critical regime  $p = \Theta(n^{-1/2})$, it is only known that $-\log \P_p(X \le \eta \E X) = \Theta(n^{3/2})$, and both the Poisson paradigm of~\cite{janson1990poisson,janson2016lower} and the variational framework of~\cite{kozma2023lower} break down and fail to identify the leading constant. 

Here for the first time we are able to compute  first-order asymptotics for $\log \P_p ( X \le \eta \E X)$ in this critical regime.  Our method works for $p = c n^{-1/2}$ for $c \in (0, \overline c(\eta))$, where the upper bound $\overline c(\eta)$  ranges from $e^{-1/2}$ in the case $\eta=0$ to $+\infty$ for any $\eta \in [ .4993,1)$.   

The method we use is different than the methods used for the lower tail in both the sub- and super-critical regimes. Via a statistical physics reformulation of the problem, we reduce the estimation of the rate function to the  estimation of the expected edge density in the corresponding conditional distribution.  Via a conditioning argument, we show that this edge density is (approximately) the solution to a fixed point equation. 

Beyond finding the asymptotic formula, we can deduce some structural results about the conditional distribution: we show that the conditional distribution converges in normalized cut metric to an \ER random graph with edge probability $q= q(p,\eta)$, but in contrast to the super-critical regime $p=\omega(n^{-1/2})$, convergence of subgraph counts fails. 

We also justify the use of the term `critical regime' by establishing  existence and non-existence of  phase transitions in the problem, in the sense of non-analyticities of the rate function.  The critical regime is also where global structure emerges in the conditional distribution; we make this precise in Section~\ref{secConditionalIntro}.

We prove analogous results for the lower tail in the random graph $G(n,m)$. We state our results for $G(n,p)$ in Sections~\ref{subsecMainResults}-\ref{secConditionalIntro} and then state our results for $G(n,m)$ in Section~\ref{subsecGnmResults}.

\subsection{Lower tails}
\label{subsecMainResults}

 To state the main results we need a little notation.  The random variable $X$ will always denote the number of triangles in a random graph.  Probabilities and expectations with respect to $G(n,p)$ will be denoted  $\P_p$ and $\E_p$. We use $\eta$ to control the strength of the lower-tail event, i.e., we consider the event $ X \le \eta \E_p X  $.  We will often use the parameterization $p = c /\sqrt{n}$.  We use $W(\cdot)$ to denote the Lambert-W function, the inverse of the function $xe^{x}$ on the positive real axis.  

Our main result determines the asymptotics of the logarithm of the lower-tail event for a subset of $(\eta,c)$ pairs.
Let
\begin{equation}
\label{eqEtaStarDef}
\eta^\ast:= \left(\frac{W(2/e)}{2/e}\right)^{3/2}\approx 0.49928 \,,
\end{equation}
and 
define $\overline c:[0,1]\to \R\cup\{\infty\}$ as follows:
\begin{equation}
\label{eqOverlineCeta}
\overline c(\eta) =
\begin{cases} 
\frac{1}{\sqrt{e(1-\eta/\eta^*)}} & \text{if } \eta < \eta^\ast \\
+ \infty & \text{if } \eta \geq \eta^\ast\, . 
\end{cases}
\end{equation}
\begin{theorem}\label{thmLowerTailAsymp}
 Fix $c>0$ and $\eta\in [0,1)$ satisfying $c< \overline c(\eta)$.
If $p=(1+o(1))c/\sqrt{n}$ then 
\begin{equation}
\label{eqMainThmAsymptotics}
  \lim_{n\to\infty} \frac{1}{n^{3/2}}\log \P_p(X\leq \eta \E X)=\frac{1}{2} \cdot\left[  \frac{W(2 \zeta c^2)^{3/2} +3W(2\zeta c^2)^{1/2}}{3\sqrt{2\zeta}}-\frac{\log(1-\zeta) \eta c^3}{3}-c\right]
\end{equation}
where $\zeta$ is the unique solution in $[0,1]$ to the equation
 \begin{equation}
 \label{eqZetaDef}
    (1-\zeta)\left( \frac{ W(2 \zeta c^2)}{2 \zeta c^2} \right)^{3/2}= \eta\, . 
 \end{equation}
\end{theorem}

Theorem~\ref{thmLowerTailAsymp} immediately gives the following corollary.  For the special case of triangle-freeness, we obtain the log asymptotics  for $c< e^{-1/2}$.
\begin{cor}
\label{thmLogAsymptotics}
    Fix $0 < c < e^{-1/2}$. If $p =(1+o(1)) c n^{-1/2}$, then
    \begin{equation}
     \lim_{n \to \infty}   \frac{1}{n^{3/2}} \log \P_p(X=0) 
     = \frac{1}{2} \left[  \frac{W(2 c^2)^{3/2} +3W(2 c^2)^{1/2}}{3\sqrt{2}} - c\right]\, .
    \end{equation}
\end{cor}

On the other hand, for the case of more moderate lower tails, with $\eta \ge \eta^\ast$, Theorem~\ref{thmLowerTailAsymp} gives the log asymptotics for \emph{all} $c>0$.  The formula given by  Theorem~\ref{thmLowerTailAsymp} interpolates between two very different formulas, one accurate when $p =o(n^{-1/2})$ and the other when $p = \omega(n^{-1/2})$.  The first is the `Poisson bound' of $-\frac{c^3}{6}(1- \eta +\eta \log \eta)$, which treats the random variable $X$ like a Poisson random variable of the same mean. Janson~\cite{janson1990poisson} and Janson and Warnke~\cite{janson2016lower} show this bound gives asymptotics of the log when $c \to 0$.  The second is the `replica symmetric bound', which lower bounds the lower-tail event by considering a random graph $G(n,q)$ with $q = \eta^{1/3}p$ so the expected number of triangles is $\eta \E_p X$.  The results of Kozma and Samotij~\cite{kozma2023lower} along with those of Zhao~\cite{zhao2017lower} show the replica symmetric bound gives asymptotics of the log when $c \to \infty$ (and $\eta > .1012$); see more in Section~\ref{secPhaseTransition}.  See Figure~\ref{fig3compare} for a plot of the formula from Theorem~\ref{thmLowerTailAsymp} against the Poisson bound  and the replica symmetric bound when $\eta =1/2$.

\begin{figure}
    \centering
\includegraphics[width=0.6\linewidth]{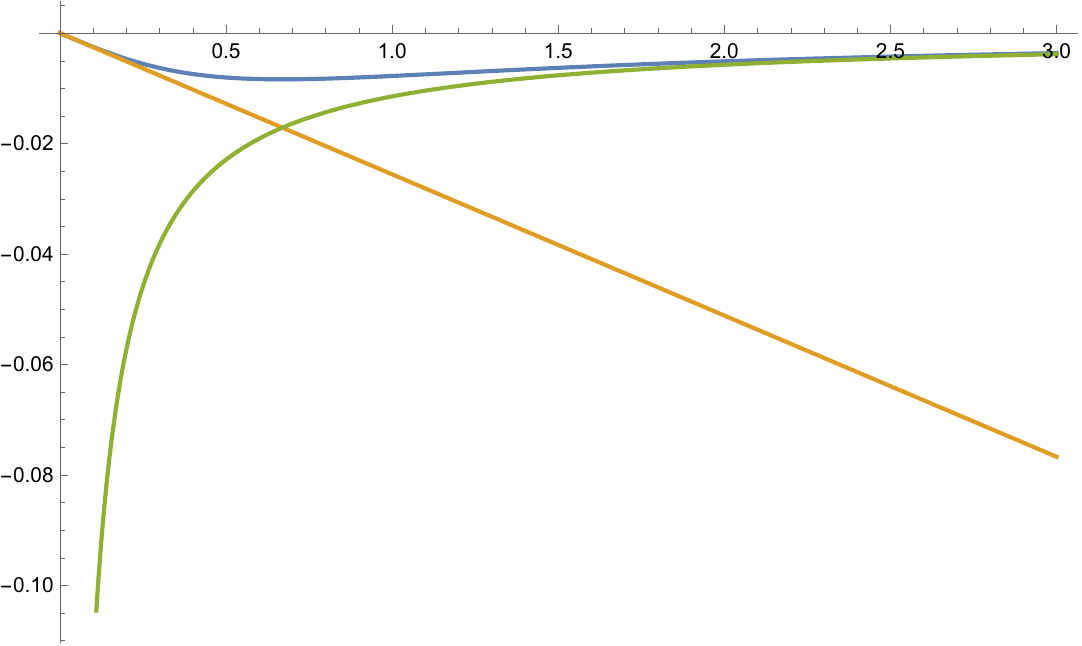}
    \caption{The RHS of Theorem~\ref{thmLowerTailAsymp} for $\eta=1/2$ (blue) plotted as a function of $c$ against the Poisson bound (orange) and the replica symmetric bound (green), with all three functions scaled by a factor $c^{-2}$. }
    \label{fig3compare}
\end{figure}

\subsection{Phase transition and absence of phase transition}
\label{subsecPhaseIntro}

We now describe what Theorem~\ref{thmLowerTailAsymp} tells us about the  existence or non-existence of a phase transition in the lower-tail problem for a given $\eta \in [0,1)$.  The notion of phase transition we consider is the statistical physics definition (the Yang--Lee definition~\cite{yang1952statistical}) involving non-analyticities of thermodynamic functions, see e.g.~\cite{ruelle1999statistical,friedli2017statistical}).

In particular, define the lower-tail rate function (in the critical regime) as
\begin{equation}
\label{eqFetaDef}
    \varphi_\eta(c) :=  \liminf_{n \to \infty} \frac{1}{n^{3/2}} \log \P_{c/\sqrt{n}} (X \le \eta \E X)\,.
\end{equation}
Then we say a \emph{phase transition} occurs at some $c^\ast > 0$ if the function $\varphi_\eta(c)$ is non-analytic at $c= c^\ast$.

\begin{cor}
\label{corPhaseTransitionGnp}
    A phase transition for triangle-freeness in  $G(n,p)$  occurs at $p=c^\ast/\sqrt{n}$ for some $c^\ast \in [1/\sqrt{e}, 4.342)$ in the sense that the function $\varphi_0(c)$ is non-analytic at $c=c^\ast$.

   More generally, there exists $\eta_\ell \approx .0091$ so that for every $\eta \in [0, \eta_\ell)$, there exists $c^\ast_\eta \in (0,\infty)$ so that a phase transition occurs in the lower-tail problem at  $c^\ast_{\eta}$.
\end{cor}
The value $\eta_\ell$ is given explicitly in~\cite{zhao2017lower} as the maximum of a univariate function on the unit interval.\footnote{In fact due to different normalization, $\eta_\ell$ is the cube of the lower threshold value $\approx .209$ from~\cite{zhao2017lower}.}

To see how Theorem~\ref{thmLowerTailAsymp} can be used to prove the existence of a phase transition, consider the triangle-free case $\eta=0$.  A lower bound for $\varphi_0(c)$ is $- \frac{c}{4}$ (by considering the probability of the event that $G$ is bipartite). The function on the RHS in Corollary~\ref{thmLogAsymptotics} becomes smaller than $-c/4$ around $c = 4.342$ (see Figure~\ref{fig:ldrate}), and thus the theorem cannot hold beyond this point. Since the RHS function is an analytic function of $c$, and analytic functions have unique analytic continuations, $\varphi(c)$ must have a non-analytic point  somewhere below $4.342$.  

\begin{figure}
\label{figLDrate}
    \centering
    \includegraphics[width=0.6\linewidth]{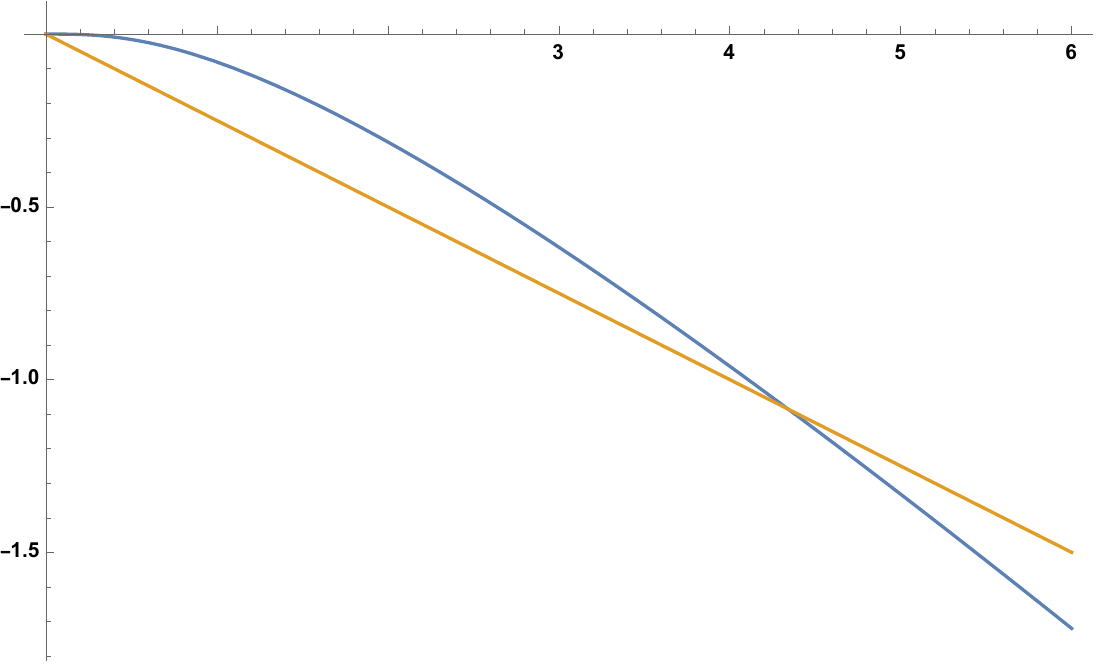}
    \caption{The RHS of Corollary~\ref{thmLogAsymptotics} (blue) plotted against the lower bound of $-c/4$ (orange) on $\varphi_0(c)$ defined in~\eqref{eqFetaDef}.
    After these functions cross, the formula in Corollary~\ref{thmLogAsymptotics} cannot hold. }
    \label{fig:ldrate}
\end{figure}

On the other hand, for moderate lower tails ($\eta \ge \eta^\ast$), the formula~\eqref{eqMainThmAsymptotics} (an analytic function of $c$) in Theorem~\ref{thmLowerTailAsymp} holds for all $c>0$ and so there is no phase transition.

\begin{cor}
\label{corNoPhaseTransitionGnp}
    If $\eta \ge \eta^\ast$, then no phase transition occurs in the lower-tail problem for triangles with $p =\Theta(n^{-1/2})$.  That is, the function $\varphi_\eta(c)$ is an analytic function of $c$ for $c \in (0, \infty)$ (and in fact we can replace the $\liminf$ in~\eqref{eqFetaDef} by a $\lim$).
\end{cor}

In the cases that a phase transition occurs, it would be interesting to prove the phase transition is unique and to locate it precisely; see more in Section~\ref{secOQuestions}.

\subsection{Conditional distributions and convergence in cut metric}
\label{secConditionalIntro}

Beyond determining rate functions, our methods give structural information about the conditional distribution: $G(n,p)$ conditioned on the lower-tail event $\{ X \le \eta \E X \}$.  

For dense graphs ($p$ constant), the result of Chatterjee and Varadhan~\cite{chatterjee2011large} states that whp $G(n,p)$ conditioned on the large deviation event in question is close in cut metric to a graphon achieving the optimum of the variational problem mentioned above; see~\cite{lovasz2012large,chatterjee2017large} for background on graphons, the cut metric, and the relation to large deviations.  For $p =o(1)$, less is known.  For the case $\eta=0$ of triangle-freeness, the result of {\L}uczak~\cite{luczak2000triangle} implies that for $p=\omega(n^{-1/2})$, $G(n,p)$ conditioned on triangle-freeness is close in cut metric to a random bipartite graph; much stronger is the result of Osthus, Pr{\"o}mel, and Taraz showing that for $p \ge (\sqrt{3}+\eps) \sqrt{\log n}\cdot n^{-1/2}$ the conditional measure is bipartite whp; see also~\cite{jenssen2023evolution} for  precise structural information at slightly smaller $p$.  For the more general lower tail, Chin \cite{chin2023structure} has recently proved (following~\cite{kozma2023lower,zhao2017lower}) that when $p =\omega(n^{-1/2})$ and $\eta > .1012$, the conditional measure is close in cut metric and in subgraph counts to $G(n,q)$ where $q= \eta^{1/3} p$.

Here we also determine the typical structure of the lower-tail conditional distribution, but the structure differs in one key respect from the results for $p = \omega(n^{-1/2})$. Our results say that in the setting of Theorem~\ref{thmLowerTailAsymp}, the conditional distribution converges in the normalized cut metric to that of $G(n,q)$ for an explicit $q = q(p,\eta)$ that is different than the replica symmetric $\eta^{1/3} p$.  Because of this, we know that the distribution does not convergence to $G(n,q)$ in terms of small subgraph counts; the typical number of triangles differs by a constant factor from the typical number in $G(n,q)$.

We now give the definitions to state the results precisely.
For a matrix $A\in \R^{n \times n}$, define the cut norm by
\[
\|A\|_\boxempty=\sup_{x,y\in\{0,1\}^n} x^T A y\, .
\]
For a graph $G$, we let $A_G$ denote its adjacency matrix and we let $J=J_n$ denote the $n \times n$ all ones matrix. We note that $\|A_G\|_\boxempty$ coincides with the usual definition of the cut norm of the graphon associated to $G$. 
\newpage

\begin{theorem}
\label{thmLowerTailStructure}
Fix $c>0$ and $\eta\in [0,1)$ satisfying $c< \overline c(\eta)$, and
 and let $\zeta$ be as in Theorem~\ref{thmLowerTailAsymp}.
If $p=(1+o(1))c/\sqrt{n}$ and $G$ is distributed as $G(n,p)$ conditioned on the event $X\leq \eta \E X$, then with probability $1$,
\begin{align}\label{eq:cutnormGnp}
\
\lim_{n\to\infty} \frac{1}{qn^2} \| A_{G} - qJ \|_\boxempty =0 \, ,
\end{align}
 where 
\begin{equation}
\label{eqQform}
    q=\sqrt{\frac{ W(2 \zeta c^2)}{2 \zeta}}  \cdot \frac{1}{\sqrt{n}}\, .
\end{equation}
\end{theorem}

\begin{figure}
    \centering
\includegraphics[width=0.56\linewidth]{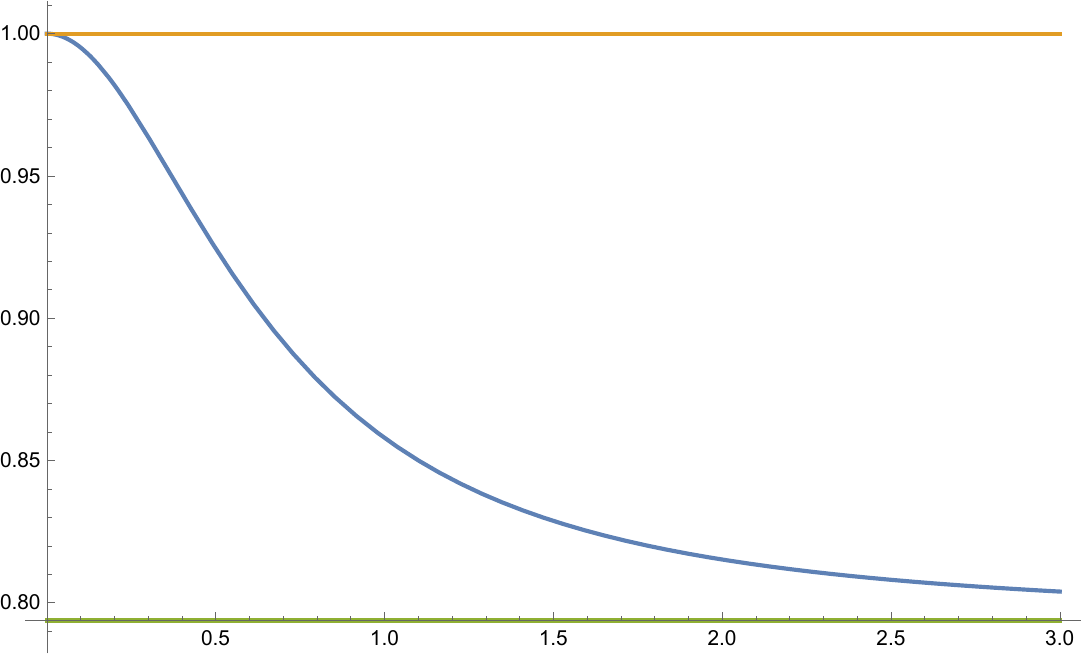}
    \caption{The edge density $q$ (scaled by $p^{-1}$) of the conditional distribution for the lower-tail event when $\eta=1/2$  as a function of $c$ (in blue); the horizontal lines mark the edge density of the unconditioned random graph (in orange) and of the random graph achieving the target number of triangles (in green). This gives another visualization of the interpolation between the Poisson and replica symmetric bounds shown in Figure~\ref{fig3compare}.}
    \label{figDensity}
\end{figure}

It is straightforward to show that the random graph $G(n,q)$ also satisfies \eqref{eq:cutnormGnp}, and so by the triangle inequality $G=G_n$ converges almost surely to $G(n,q)$ in the normalized cut metric.   In Figure~\ref{figDensity} we plot the function $q$ in~\eqref{eqQform} as a function of $c$ for the case $\eta=1/2$.

 Theorem~\ref{thmLowerTailStructure} says that whp the number of edges in any cut of $G$ with linear-sized parts is, to first order, the expected number of edges between parts of these sizes  in $G(n,q)$.   The max cut of $G$ is therefore whp  half the edges of $G$ to first order.  This suggests another way of understanding phase transitions in the lower-tail problem as $c$ varies, via typical sizes of extremal cuts, capturing large-scale structure in the conditional distribution.

To make this precise, let $\mathrm{MC}(G)$ denote the fraction of edges of $G$ in its max cut (setting $\mathrm{MC}(G) =0$ if $G$ has no edges).

\begin{cor}
\label{corMaxCut}
   Fix $c>0$ and $\eta\in [0,1)$ satisfying $c< \overline c(\eta)$. If $p=(1+o(1))c/\sqrt{n}$ and $G$ is distributed as $G(n,p)$ conditioned on the event $X\leq \eta \E X$, then whp $\mathrm{MC}(G)=1/2 + o(1)$.
\end{cor}

 In particular, if we define 
 \begin{equation}
    g_{\eta}(c) = \liminf_{n \to \infty} \E_{c/\sqrt{n}} \left[ \mathrm{MC}(G) \mid X \le \eta \E X \right] \,,
\end{equation}
 then if $\eta \ge \eta^\ast$ we have $g_\eta(c)=1/2$ for \emph{all} $c>0$.  On the other hand, for the case $\eta=0$, we know that $g_\eta$ must have a non-analytic point.
\begin{cor}
    \label{corMaxCutPhase}
     The function $g_0(c)$ witnesses a phase transition for triangle-freeness in  $G(n,p)$, in the sense that for some $c^\ast >0$,  $g_0(c)$ is non-analytic at $c=c^\ast$.
\end{cor}
That is, the max-cut fraction $g_0(c)$ is an \emph{order parameter} for a phase transition.   Corollary~\ref{corMaxCutPhase} follows from Corollary~\ref{corMaxCut} and the result of {\L}uczak~\cite{luczak2000triangle} that implies that $g_0(c) \to 1$ as $c \to \infty$.  In~\cite{jenssen2023evolution}, it was conjectured that there is exactly one non-analyticity of $g_0(c)$ for $c \in (0,\infty)$; here we prove there is at least one.   See Section~\ref{secOQuestions} and~\cite[Section 1.3]{jenssen2023evolution} for some related questions and conjectures.  For more on order parameters and phase transitions in combinatorial problems see, e.g.~\cite{borgs1990rigorous,bollobas2001scaling,borgs2001birth,biskup2007large}.

\subsection{Lower tails for $G(n,m)$}
\label{subsecGnmResults}

We now consider the lower-tail problem for the random graph $G(n,m)$, a uniformly random graph on $n$ vertices with $m$ edges.  Using results on the typical number of edges in the lower-tail conditional distribution, we can transfer the main result, Theorem~\ref{thmLowerTailAsymp}, to a result for $G(n,m)$, but perhaps  surprisingly the conclusions about phase transitions are different.

Here we will parameterize $m = (1+o(1))\frac{b}{2} n^{3/2}$, so the average degree in $G(n,m)$ is asymptotically $b\sqrt{n}$. Probabilities and expectations with respect to $G(n,m)$ will be denoted $\P_m$ and $\E_m$.

\begin{theorem}\label{thmLowerTailGnm}
Fix $b>0$ and $\eta\in [0,1)$ satisfying $b^2(1-\eta)<W(2/e)$. If $m = (1+o(1))\frac{b}{2} n^{3/2}$,  then
 \begin{equation}
 \label{eqGnmLTform}
      \lim_{n \to \infty}   \frac{1}{n^{3/2}} \log \P_m(X \leq \eta \E X) =-\frac{b^3}{6}(1-\eta+\eta \log\eta)\, . 
    \end{equation}
\end{theorem}

For the special case $\eta =0$, this result implies a formula for $T(n,m)$, the number of triangle-free graphs on $n$ vertices and $m$ edges, accurate up to a factor $\exp\{o(n^{3/2})\}$.  
\begin{cor}\label{corSmallCTnm}
    Fix  $b>0$ satisfying $ b^2 < W(2/e)$. If  $m = (1+o(1)) \frac{b}{2} n^{3/2}$, then
 \begin{equation}
 \label{eqGnmTFform}
      \lim_{n \to \infty}   \frac{1}{n^{3/2}} \log \P_m(X=0) = -\frac{ b^3}{6} \,.
    \end{equation}
  In other words, 
    \begin{equation}
        T(n,m) = \binom{\binom{n}{2}}{m} \exp \left( - \frac{4}{3}\left( \frac{m}{n} \right)^3 + o(n^{3/2})\right).
    \end{equation}
\end{cor}
Solving $b^2 = W(2/e)$ yields $b \approx .48117$ for the threshold up to which the theorem applies.

Note the particularly simple form of the RHS's of~\eqref{eqGnmLTform} and~\eqref{eqGnmTFform} compared to the more complicated formulas for the case of $G(n,p)$  in Theorem~\ref{thmLowerTailAsymp} and Corollary~\ref{thmLogAsymptotics}.  In the $G(n,m)$ case the formulas match the Poisson-type behavior of the lower tail in the sparser regimes $p = o(n^{-1/2})$, covered by~\cite{janson1987uczak,janson1990poisson,janson2016lower}.  These formulas also give an intuitive interpretation for the $G(n,p)$ results in Section~\ref{subsecMainResults}. A simple calculation shows that
\begin{equation}
\label{eqoptimization}
\log \P_p(X \le \eta \E_p X) = (1+o(1)) \max_{m} \left \{ \log \P_p( |G|=m)  + \log \P_m (X \le \eta \E_p X) \right \} \,.   
\end{equation}
The first term in the braces is a simple binomial large deviation probability, while the second, at least in the regime covered by Theorem~\ref{thmLowerTailGnm}, has Poisson-type form.  The more complicated-seeming formula in Theorem~\ref{thmLowerTailAsymp} involving  Lambert-W functions arises by solving this one variable optimization problem.   

As in Section~\ref{subsecPhaseIntro}, we can ask about the existence of  phase transitions in the lower-tail problem for $G(n,m)$ in the critical regime $m = \Theta(n^{3/2})$, in the sense of a non-analyticity of the $G(n,m)$ lower-tail rate function
\begin{equation}
    \label{eqGnmFreeEnergyDef}
    \widehat \varphi_\eta (b) := \liminf_{n \to \infty} \frac{1}{n^{3/2}} \log \P_{\lfloor b n^{3/2}\! /2\rfloor}(X \le \eta \E X) \,.
\end{equation}
   The answer is very different to the case of $G(n,p)$: in the $G(n,m)$ model, for \emph{every} $\eta \in [0,1)$, a phase transition occurs.  

\begin{cor}
    \label{corGnmPhase}
    For all $\eta \in [0,1)$, the function $\widehat \varphi_\eta(b)$ is non-analytic at $b^\ast$ for some $b^\ast \in (0,\infty)$. 
\end{cor}

Corollary~\ref{corGnmPhase} stands in particular contrast to Corollary~\ref{corNoPhaseTransitionGnp} for $\eta \ge \eta^\ast$.  To get some intuition for the difference in these results, consider strategies for realizing the lower-tail event in $G(n,p)$.  On one hand, there are `local' strategies that lower the typical number of triangles by (in a rough sense) deleting triangles at each vertex. When $p = o(n^{-1/2})$ this strategy is asymptotically optimal and Janson's inequality gives the asymptotics of the logarithm.  When $p = c/\sqrt{n}$ for $c$ small this strategy is still optimal, but its effect is more delicate as we have to consider how the typical vertex degrees change; or in other words, the optimizer of the RHS of~\eqref{eqoptimization} is a constant factor smaller than $\binom{n}{2} p$.  On the other hand, when $c$ is large, a different `global' strategy is effective: be nearly bipartite.  When $\eta \ge \eta^\ast$, the first strategy always outperforms the second.  In $G(n,m)$, however, decreasing the average degree is forbidden, and so this dramatically curtails the effectiveness of the first strategy for large $c$, and the second strategy eventually wins out for any $\eta$, forcing a phase transition.

\subsection{An overview of the method}
\label{secTechniques}

The techniques we use to prove the main  results combine ideas from statistical physics, combinatorics, and computer science.  
We first describe some of  the techniques that go into proving Theorem~\ref{thmLowerTailAsymp} in the special case $\eta=0$, leaving  the more general method to  Section~\ref{secLocalCondition}.

The first step is to reformulate the problem of estimating $\P_p ( X=0)$ to the problem of estimating a statistical physics partition function  via the following identity (here all sums are over labeled graphs on $n$ vertices):
\begin{equation}
\label{eqZlamIntrodef}
    \P_p( X =0) = (1-p)^{\binom{n}{2}} \sum_{G \text{ triangle-free}} \lam^{|G|} =: (1-p)^{\binom{n}{2}} \, Z(\lam)
\end{equation}
where $\lam = \frac{p}{1-p}$, and $|G|$ denotes the number of edges of $G$.  
The function $Z(\lam)$ is the partition function of a hard-core model on a $3$-uniform hypergraph (in which the $\binom{n}{2}$ edges of the complete graph $K_n$ are the vertices and three edges are in a hyperedge if they form a triangle; see e.g.~\cite{balogh2015independent,mousset2020probability,jenssen2023evolution} for more on this formulation).  

The partition function $Z(\lam)$ is the normalizing constant of a probability measure $\mu_\lam$ on triangle-free graphs with vertex set $[n]$ (the \emph{Gibbs measure}), defined by
\begin{equation}
\label{eqmulambdaDef}
    \mu_\lam(G)= \frac{\lam^{|G|}}{Z(\lam) } \, ;
\end{equation}
this measure is exactly the distribution of $G(n,p)$ conditioned on the event $\{X =0 \}$.  

Next we observe  that to approximate $Z(\lam)$ (and thus $\P_p(X=0)$), it suffices to approximate the expectation $\E_{\mu_\lam} |G|$ (which is simply the expected number of edges in $G(n,p)$ conditioned on triangle-freeness). Indeed we have the identity
\begin{align}\label{eqZderivative}
\E_{\mu_\lam}(|G|)= \sum_{G \text{ triangle-free}}|G| \frac{\lam^{|G|}}{Z(\lam)}=\lam (\log Z(\lam))'\, , 
\end{align}
and thus
\begin{align}\label{eqLogZidentity}
   \log Z(\lam)= \int_{0}^\lam \frac{\E_{\mu_\theta} |G|}{\theta} d\theta\, ,
\end{align}
and so approximating the expectation on the whole interval $[0,\lam]$ gives an approximation to $\log Z(\lam)$.
Such a relation between  the hard-core partition function  of a graph and its `occupancy fraction' was used in extremal combinatorics in, e.g.~\cite{davies2017independent,davies2018average}. 

The heart of the method is a \textbf{local conditioning strategy} to estimate $\E_{\mu_\lam} |G|$.  

The first step of this strategy is to show that the measure $\mu_\lam$ on triangle-free graphs admits a strong concentration inequality for Lipschitz functions.  This is proved by showing that a local-update Markov chain (the Glauber dynamics) with stationary distribution $\mu_{\lam}$ is contractive, then applying the approach of~\cite{luczak2008concentration,barbour2019longterm}  deducing concentration inequalities from contraction of dynamics.  One consequence is that under the distribution  $\mu_\lam$, the degrees of the vertices of $G$ are tightly concentrated around a deterministic (but a priori unknown) quantity $d = d(n,p)$. 

We then find a way to write $d$ as a function of itself; this function has a unique fixed point (under the conditions of Theorem~\ref{thmLowerTailAsymp}), and thus we can deduce an asymptotic formula for $d$, which immediately yields an asymptotic formula for $\E _{\mu_\lam} |G|$.  

The fixed point equation is obtained by picking a vertex $v$, and conditioning the measure $\mu_{\lam}$ on $G_v$, the graph $G$ with $v$ removed.  Conditioned on the event $G_v = H$, the distribution of the neighborhood of $v$ in $G$ is exactly that of the \emph{hard-core model} on the graph $H$; this is the probability distribution $\nu_{H,\lam}$ on independent sets in $H$ given by
\begin{equation}
    \nu_{H,\lam}(I) = \frac{\lam^{|I|}}{\Xi_H(\lam)} \quad \text{where} \quad \Xi_H(\lam)= \sum_{I \text{ independent set}}\lam^{|I|}\, .
\end{equation}
Note the similarity of the measure $\nu_{H,\lam}$ to that of $\mu_{\lam}$ defined in~\eqref{eqmulambdaDef}, but crucially while $\mu_{\lam}$ is supported on hypergraph independent sets, $\nu_{H,\lam}$ is supported on graph independent sets.  This (along with the condition on $\lam$ in the theorem) allows us to write an approximation for  the expected size of an independent set from $\nu_{H,\lam}$ as an infinite series with terms depending on $\lam$ and the average degree of $H$.  This gives the desired fixed point equation.  

For the  general case $\eta \in [0,1)$, we proceed in a similar way but consider instead the partition function
\begin{equation}
    Z(\lam,\zeta) := \sum_G \lam ^{|G|} (1-\zeta)^{ X(G)} \, ,
\end{equation}
where as above $X(G)$ denotes the number of triangles of $G$, and $\zeta \in (0,1]$.  This partition function penalizes but does not forbid triangles (unless $\zeta =1$ in which case $Z(\lam,\zeta) = Z(\lam)$).  We will show that in the setting of Theorem~\ref{thmLowerTailAsymp}, up to  a constant shift, $\log Z(\lam,\zeta)$ approximates $\log \P_p(X \le \eta \E X)$ once $\zeta$ is chosen so that the expected number of triangles in the associated Gibbs measure $\mu_{\lam,\zeta}$ is close to the target; that is,  $\E_{\mu_{\lam,\zeta}} X = (1+o(1)) \eta \E_p X$.  In the general case, after local conditioning the distribution of the neighborhood of $v$ is an antiferromagnetic Ising model $\nu_{H,\lam,\zeta}$ on graphs, instead of the hard-core model $\nu_{H,\lam}$ in the $\eta =0 $ case.

\subsection{Future directions}
\label{secOQuestions}

Once the existence of a phase transition has been established (as in the case  $\eta < \eta_{\ell}$ here), there are a host of more detailed questions one can ask to fully understand the phenomenon. Is the phase transition unique? Is it a first- or second-order phase transition? What is the behavior at the critical point? Based on analogies to order--disorder phase transitions in lattices systems, and in particular to the ferromagnetic Potts model on $\Z^d$ (see the survey~\cite{duminil2017lectures}), we make some predictions about the lower-tail phase transition.

\begin{conj}
\label{conjPhaseTransitionBehavior}
    For each $\eta$ sufficiently small (including the interval $[0,\eta_\ell)$), the following hold:
    \begin{itemize}
        \item There is a unique non-analytic point $c^\ast= c^\ast(\eta)$ of the rate function $\varphi_{\eta}(c)$ on the positive real axis.
        \item The phase transition is a first-order phase transition: the first derivative of $\varphi_{\eta}(c)$ is discontinuous at $c^\ast$.
        \item The max-cut fraction is an order parameter for the phase transition and jumps at the critical point: $g_{\eta}(c) =1/2$ for $c<c^\ast$, while $g_\eta(c) > 1/2$ for each $c \ge c^\ast$.
        \item At the critical point, the `ordered' state dominates: there exists $\eps>0$ so that with $p=c^\ast/\sqrt{n}$, whp the max-cut fraction of $G(n,p)$ conditioned on the lower-tail event is at least $1/2+\eps$. 
    \end{itemize}
\end{conj}

A starting point towards proving Conjecture~\ref{conjPhaseTransitionBehavior} would be to show that the lower-tail rate function exists in general; that is, that the liminf defining $\varphi_\eta(c)$ in~\eqref{eqFetaDef} can be replaced by a limit.  
\begin{question}
\label{QLimitExist}
    Does the  limit 
    \begin{equation}
    \label{eqLim1}
     \lim_{n \to \infty} \frac{1}{n^{3/2}} \log \P_{c/\sqrt{n}}(X \le \eta \E X)
    \end{equation}
    exist for all $c>0$ and $\eta \in [0,1)$?
\end{question}
This is not known even in the special case $\eta =0$.  
In the study of spin glasses and spin models on sparse random graphs, the existence of such a limit can often be proved without knowing its value, e.g., in~\cite{guerra2002thermodynamic,bayati2013combinatorial}, but we do not know how to apply those approaches here.

A very natural  question would be to study the non-existence probability and lower tails for larger cliques.  For avoiding a copy of $K_{r+1}$, the critical regime is $p = \Theta \left( n^{- \frac{2}{r+2}}  \right)$: for much smaller $p$ the asymptotics of the logarithm are given by Janson's Inequality~\cite{janson1987uczak,janson1990poisson,janson2016lower} while for  much larger $p$, the logarithm of the non-existence probability is asymptotic to the logarithm of the probability of being $r$-partite (which has been proved using hypergraph containers~\cite{luczak2000triangle,balogh2015independent,saxton2015hypergraph}).  Moving beyond triangles would require at least one significant new idea: the local conditioning strategy described above yields a pairwise interacting model (the hard-core model on a graph) when applied to avoiding triangles; when applied to avoiding larger cliques, the resulting model would be a hypergraph hard-core model for which the necessary convergence results are lacking or even false~\cite{galvin2024zeroes,zhang2023note}.

One could also further consider forbidden subgraphs $H$ beyond cliques. If $H$ is non-bipartite, the `critical density' is $p = \Theta(n^{-1/m_2(H)})$,   where $m_2(H) = \max_{F \subseteq H; |E(F)| \ge 2} \frac{|E(F)|-1 }{|V(F)|-2  } $ is the $2$-density of $H$.

Returning to triangle-freeness, the story does not  end at asymptotics of the log probability. 
Much sharper results, and in particular first-order asymptotics of $\P_p(X=0)$, are known when $p$ is sufficiently smaller or sufficiently larger than $n^{-1/2}$. 
 For $p \le n^{-1/2-\eps}$ with $\eps>0$ constant, asymptotics of $\P_p(X=0)$ are known via Janson's Inequality~\cite{janson1987uczak} and extensions~\cite{wormald1996perturbation,stark2018probability,mousset2020probability}. For $p \ge c  \sqrt{ \log n} \cdot n^{-1/2}$, asymptotics of $\P_p(X=0)$ are known by either showing that for sufficiently large $c$, $G(n,p)$ conditioned on triangle-freeness is bipartite whp~\cite{erdosasymptotic,promel1996asymptotic,osthus2003densities}, or by quantifying how close a typical graph drawn from the conditional distribution  is to being bipartite~\cite{jenssen2023evolution}.

 One can also study related algorithmic questions, in particular about the existence of efficient algorithms to sample from the lower-tail conditional distribution.  In~\cite{jenssen2024sampling} the current authors give efficient sampling algorithms for $G(n,p)$ conditioned on triangle-freeness when $p<c/\sqrt{n}$ and when $p> C/\sqrt{n}$, for constants $ c,C>0$.

\subsection{Organization}

Section~\ref{secPrelim} contains some preliminaries on Gibbs measures and partition functions, the cluster expansion, and tools relating Markov chains and concentration  of measure.  In Section~\ref{secLocalCondition} we prove the main theorem on lower tails in $G(n,p)$, generalizing the strategy outline in Section~\ref{secTechniques}; we postpone proofs of some lemmas to later sections.  In Section~\ref{sec:concentration} we prove concentration results for the statistical physics models we work with, using the connection to Markov chains.  In Section~\ref{secRegularity}, we apply the cluster expansion to deduce facts about the anti-ferromagnetic Ising model on  nearly regular graphs without too many short cycles.  In Section~\ref{secLowerTailPartition} we relate statistical physics partition functions to the lower-tail probability and prove our results on the lower tail for $G(n,m)$.   In Section~\ref{secCutNorm} we prove the results on the typical structure of graphs drawn from the conditional measure. In Section~\ref{secPhaseTransition} we deduce our results on phase transitions.  Finally in Appendix~\ref{secPinnedCluster} we prove a result on convergence of the cluster expansion. 

\section{Preliminaries}
\label{secPrelim}

\subsection{Notation}
We let $\mathbb{N}=\{0,1,\ldots,\}$.
For a graph $G = (V,E)$ we denote its number of edges by $|G|$, its number of triangles by $X(G)$, its maximum degree by $\Delta(G)$, its minimum degree by $\delta(G)$, and its average degree by $\overline d(G)$.  For $v\in V$ we let $d_G(v)$ denote the degree of the vertex $v$ in $G$, and   let $t_G(v)$ denote the number of triangles of $G$ incident to $v$. For $u,v\in V$ we will often write $uv$ for the pair $\{u,v\}$ and we let $d_G(uv)$ denote the codegree (number of common neighbors) of $u$ and $v$. We will often identify $G$ with its edge set and write $uv\in G$ to mean $uv\in E$. Most graphs considered in this paper will have $n$ or $n-1$ vertices unless specified otherwise. 

We use $\P_p$ and $\E_p$ to refer to probabilities and expectations with respect to the \ER random graph $G(n,p)$ (each of the $\binom{n}{2}$ edges included independently with probability $p$). 
When we work with the random graph $G(n,m)$ (a uniformly random graph on $n$ vertices with $m$ edges), we use the notation $\P_m$, $\E_m$ (slightly abusing notation, but the random graph will always be clear from context).  When we write an expectation inside a probability, the expectation will always be with respect to the outer probability measure unless otherwise noted, so $\P_p( X \le \E X)= \P_p (X \le \E_p X)$. Given a probability measure $\mu$, we write $X\sim \mu$ to denote that $X$ is a random sample from $\mu$.

All asymptotic notation is to be understood with respect to the limit $n \to \infty$. If the implicit constant in asymptotic notation $O, \Omega$ etc. depends on some parameter $\delta$, we denote this with a subscript e.g $O_\delta(\cdot), \Omega_\delta(\cdot)$.
We say a sequence of events $A_n$ holds `with high probability' (abbreviated `whp') with respect to a probability measure $\P$ if $\mathbb P(A_n) = 1-o(1)$. We say that a sequence of events $A_n$ holds `with very high probability' (abbreviated `wvhp') if $\mathbb P(A_n) = 1-n^{-\omega(1)}$ i.e. the probability that $A_n$ fails is superpolynomially small. 

For two functions $f,g:\mathbb{N}\to \mathbb{R}$ we understand $f(n)\leq g(n)$ to mean that the inequality holds for $n$ sufficiently large.

In the next section we introduce a statistical physics model with parameters $\lam \ge 0$, $\zeta \in (0, 1]$.  We will always take $\lam = p/(1-p)$.

\subsection{Partition functions and Gibbs measures}
\label{subsecPartitionFunctions}

We will work with  a couple of different statistical physics models and their associated partition functions.  

The first model is a probability distribution on the set of labeled graphs on $n$ vertices.  The partition function and corresponding Gibbs measure are
\begin{align}
\label{eqZlamBetaDef}
    Z(\lam,\zeta) &= \sum_{G} \lam^{|G|} (1-\zeta)^{ X(G)} \\
    \label{eqMulamBetaDef}
    \mu_{\lam,\zeta}(G) &= \frac{\lam^{|G|} (1-\zeta)^{ X(G)}}  {Z(\lam,\zeta)} \,.
\end{align}
Here $\lam \ge 0$ is the activity, and $\zeta \in [0, 1]$ governs how much triangles are penalized in the measure.  When $\zeta = 1$, $\mu_{\lam,\zeta}$ is supported on triangle-free graphs and coincides with the distribution $\mu_\lam$ defined in~\eqref{eqmulambdaDef}.   We will denote probabilities and expectations with respect to $\mu_{\lam,\zeta}$ by $\P_{\lam,\zeta}$ and $\E_{\lam,\zeta}$ respectively.

Our main task will be to find an asymptotic formula for $\log Z(\lam,\zeta)$ and relate this to the lower-tail probability $\P_p (X \le \eta \E X)$, where $\lam = p/(1-p)$ and $\zeta$ will be chosen so that the expected number of triangles in a graph drawn from $\mu_{\lam,\zeta}$ is close to the target $\eta \E_p X $.  

In the course of the proof, we will also consider Gibbs measures that are probability distributions on vertex subsets of some  graph $H$.  

For a graph $H=(V,E)$ and $S\subseteq V$, let $E(S)=\{xy\in E: x,y\in S\}$, the edge set of the induced subgraph $H[S]$.
For $\lam \ge 0$ and $\zeta \in [0, 1]$,  define
\begin{align}
\label{eqXilamBetaDef}
    \Xi_H(\lam,\zeta) &= \sum_{S\subseteq V} \lam^{|S|} (1-\zeta)^{ |E(S)|}\\
    \label{eqNulamBetaDef}
    \nu_{H,\lam, \zeta}(S) &=\frac{\lam^{|S|} (1-\zeta)^{ |E(S)|}}{\Xi_H(\lam, \zeta)} \,.
\end{align}
The model $\nu_{H,\lam, \zeta}$ is an  asymmetric anti-ferromagnetic Ising model (with external magnetic field)\footnote{The more standard statistical physics parametrization of the anti-ferromagnetic Ising model would be in terms of an inverse temperature $\beta \in [0, +\infty]$; here we have taken $\zeta = 1- e^{-\beta}$.}; one can also think of it  as a positive-temperature hard-core model in which edges are penalized but not forbidden (unless $\zeta=1$ when the model is exactly the hard-core model).    We will denote probabilities and expectations with respect to $\nu_{H,\lam,\zeta}$ by $\P_{H,\lam,\zeta}$ and $\E_{H,\lam,\zeta}$ respectively.

We will use the following facts about derivatives of log partition functions (generalizing~\eqref{eqZderivative}).  
\begin{lemma}
\label{lemDerivativeIdentities}
The following identities hold for any $\lam \ge0$, $\zeta \in [0,1]$, and any graph $H$.
    \begin{align}
        \E_{\lam,\zeta}|G| &= \lam  \frac{\partial}{\partial \lam}\log Z(\lam,\zeta)  \\
        \E_{H,\lam,\zeta} |S| &= \lam  \frac{\partial}{\partial \lam}\log \Xi_H(\lam,\zeta) \\
        \E_{H,\lam,\zeta} |E(S)| &=   - (1-\zeta) \frac{\partial}{\partial \zeta} \log \Xi_H(\lam,\zeta) \,. 
    \end{align}
\end{lemma}
The proof is immediate from the definitions of $Z(\lam,\zeta)$ and $\Xi_H(\lam,\zeta)$.

Another simple but very useful fact about both $\mu_{\lam,\zeta}$ and $\nu_{H,\lam,\zeta}$ is that they are both stochastically dominated by independent $p$-percolation (edge percolation in the first case, vertex percolation in the second case).

\begin{lemma}\label{lemStochDom}
For every $\lam \ge 0$, $\zeta\in [0,1]$, the distribution $\mu_{\lam,\zeta}$ is stochastically dominated by the distribution of $G(n,p)$ where $\lam = p/(1-p)$.  That is, there is a coupling of the two distributions $G \sim \mu_{\lam,\zeta}$ and $G' \sim G(n,p)$ so that with probability $1$, $G \subseteq G'$.

Likewise, for every graph $H$ and every $\lam \ge 0$, $\zeta\in [0,1]$, the distribution $\nu_{H,\lam,\zeta}$ is stochastically dominated by a random subset $S_p \subseteq V(H)$ in which each vertex is included independently with probability $p$.
\end{lemma}

\begin{proof}
Sample $G \sim\mu_{\lam,\zeta}$ by sampling one edge at a time with the correct conditional probability given the previous history. The probability of including an edge $e$ in $G$ is then
\[
\frac{\lam (1 - \zeta)^{t(e)}}{1+\lam (1 - \zeta)^{t(e)}}\leq \frac{\lam}{1+\lam}=p.
\]
where $t(e)$ denotes the number of triangles $e$ forms with edges included in previous steps.
 We can therefore couple the sampling with an edge-by-edge sampling of $G(n,p)$ so that an edge is present in the sample from $\mu_{\lam,\zeta}$ only if it is present in the sample from $G(n,p)$.  The proof for $\nu_{H,\lam,\zeta}$ is identical.
\end{proof}

We can use stochastic domination to obtain a high probability bound on the maximum degree of $G$ drawn from $\mu_{\lam,\zeta}$ and for the size of the random subset drawn from $\nu_{H,\lam,\zeta}$. The following corollary is immediate from Lemma~\ref{lemStochDom} and Chernoff's bound. 

\begin{cor}\label{cor:SDdegbound}
Let $\lambda > 0$, $\zeta \in [0,1]$, and $p = \frac{\lambda}{1 + \lambda}$. For any $\kappa > 0 $ we have
\begin{align*}
&\mathbb{P}_{\lambda,\zeta}\left(\Delta(G) > np(1 + \kappa)\right) \leq n \cdot \exp\left(-\frac{\kappa^2 np}{2 + \kappa}\right),~\text{and} \\
& \mathbb{P}_{H,\lambda,\zeta}\left(|S| > np(1 + \kappa)\right) \leq n \cdot \exp\left(-\frac{\kappa^2 np}{2 + \kappa}\right) \,.
\end{align*}
\end{cor}

We can also use stochastic domination to prove a lower bound on the expected number of edges in a sample from $\mu_{\lam,\zeta}$. 
\begin{lemma}\label{lemDegLB}
For $\lam\ge0$, $\zeta \in [0,1]$, we have
\[
\E_{\lam,\zeta}|G|\geq \binom{n}{2}  p(1-p^2)^n\, ,
\]
where $p=\lam/(1+\lam)$.
\end{lemma}

\begin{proof}
Call a pair $\{u,v\}\in \binom{[n]}{2}$ \emph{open} if $G\cup \{u,v\}$ has no triangles including the edge $\{u,v\}$. Let $Y$ denote the number of open pairs of $G$. By Lemma~\ref{lemStochDom}, 
$\E(Y)\geq \binom{n}{2}(1-p^2)^{n-2}$.
Conditioned on the event that a pair is open, the probability that it is an edge of $G$ is $p$ and so $\E|G|\geq p\E(Y)$. The result follows. 
\end{proof}

In particular, Corollary~\ref{cor:SDdegbound} and Lemma~\ref{lemDegLB} together tell us that when $\lam = \Theta(n^{-1/2})$, then $\E_{\lam,\zeta}|G| = \Theta(n^{3/2})$.

\subsection{Cluster expansion}\label{subseccluster}

To understand the asymmetric Ising model $\nu_{H,\lam,\zeta}$, we will use the cluster expansion.  The cluster expansion is a formal power series for $\log \Xi_H(\lam,\zeta)$ around $\lam=0$.  Conveniently, the terms of the cluster expansion have a nice combinatorial interpretation (see e.g.~\cite{scott2005repulsive,faris2010combinatorics}). A \textit{cluster} $\Gamma=(v_1, \ldots, v_k)$ is a tuple of vertices from $H$ such that the induced graph $H[\{v_1, \ldots, v_k\}]$ is connected. We let $\cC(H)$ denote the set of all clusters of $H$. We call $k$ the size of the cluster and denote it by $|\Gamma|$. Given a cluster $\Gamma$ of size $k$, the \emph{incompatibility graph} $H_\Gamma$, is the graph on vertex set $\{1,\ldots,k\}$ with an edge between $i, j$ if either $v_i, v_j$ are adjacent in $H$ or $i \ne j$ and $v_i=v_j$.  In particular, by the definition of a cluster, the incompatibility graph $H_\Gamma$ is connected. If $e=\{i,j\}\in E(H_\Gamma)$ we set $w_e=-\zeta$ if $v_i, v_j$ are adjacent in $H$ and we set $w_e=-1$ if $v_i=v_j$.

As a formal power series, the cluster expansion is the infinite series
\begin{align}\label{eqclusterexp}
\log \Xi_H(\lam,\zeta) = \sum_{\Gamma\in \cC(H)} \phi_{H,\zeta}(\Gamma) \lam^{|\Gamma|} \,,
\end{align}
where
\begin{align}
\label{eqUrsell}
\phi_{H,\zeta}(\Gamma) &= \frac{1}{|\Gamma|!} \sum_{\substack{A \subseteq E(H_\Gamma):\\ \text{$(V(H_\Gamma),A)$ connected}}}  \prod_{e\in A} w_e \, .
\end{align}
If the graph $H$ is clear from the context we will often write $\phi_\zeta(\Gamma)$ in place of $\phi_{H,\zeta}(\Gamma)$.

 We will use the following lemma which gives a sufficient condition for convergence and  bounds the error in truncating the cluster expansion. Given $\gamma,\Delta\geq 0$, let 
 \begin{align}
 \label{eqRxDef}
     R(\gamma,\Delta)=\{(\lam,\zeta)\in \C \times \C: |1-\zeta|< 1 \textup{ and } e|\lam|(1+|\zeta|\Delta)< \gamma\}\, .
 \end{align}

\begin{lemma}\label{lemClusterTail}
Fix $\gamma\in[0,1)$ and let $H$ be a graph on $n$ vertices with maximum degree $\Delta$. If $(\lam,\zeta)\in R(\gamma,\Delta)$ and $k$ is a positive integer then
\[
 \sum_ {\substack{\Gamma\in\cC(H): \\ |\Gamma|\geq k}} \left|\phi_\zeta(\Gamma) \lam^{|\Gamma|}\right| 
\leq en|\lam| \gamma^{k-1}(1-\gamma)^{-1}\, .
\]
In particular, the cluster expansion~\eqref{eqclusterexp} converges uniformly  for $(\lam,\zeta)\in R(\gamma,\Delta)$ and the function $\log \Xi_H(\lam,\zeta)$ is analytic on $R(\gamma,\Delta)$.
\end{lemma}
The proof of the above lemma is similar to that of \cite[Lemma 4.1]{jenssen2023evolution} and so we defer it to Appendix~\ref{secPinnedCluster}.

We will also need to control the convergence of derivatives of the cluster expansion with respect to $\lam$ and $\zeta$. For this we appeal to the following fact, which is a  simple consequence of  Cauchy's Integral Formula. For $r>0$ and $w\in \C$, we let $\overline B_r(w)\subseteq \C$ denote
the closed disc of radius $r$ with center $w$, and $\partial B_r(w)$ denote its boundary.

\begin{lemma}
\label{lemAnalyticDerivativesBound}
   Let $U\subseteq \C$ be open and let  $f,g:U\to\C$, be two analytic functions.   If $r>0$ and $w\in U$ are such that $\overline B_{r}(w)\subset U$, then 
\begin{align}\label{eqLocalUnifTail}
|f'(w)-g'(w)|\leq \frac{1}{ r}\sup_{z\in \partial B_r(w)}|f(z)-g(z)| \,.
\end{align}
\end{lemma} 
\begin{proof}
Using Cauchy's Integral formula we can write
\begin{align}
    |f'(w)-g'(w)|&= \frac{1}{2 \pi} \left | \oint_{\partial B_r(w)}  \frac{f(z) - g(z)}{ (z-w)^2  }  \, d z \right|  \\
    &\le \frac{2 \pi r}{2 \pi r^2} \sup_{z \in \partial B_r(w)} |f(z) - g(z) | = \frac{1}{ r} \sup_{z \in \partial B_r(w)} |f(z) - g(z) | \,. \qedhere
\end{align}
\end{proof}

In Section~\ref{secRegularity} we will use Lemma~\ref{lemAnalyticDerivativesBound} along with Lemma~\ref{lemClusterTail} and the identities from Lemma~\ref{lemDerivativeIdentities}  to approximate $\E_{H,\lam,\zeta} |S|$ and $\E_{H,\lam,\zeta} |E(S)|$ for certain graphs $H$.

\subsection{Markov chains and concentration of measure}
\label{subsecMCMC}
We will use  Markov chains in our analysis of the distributions $\mu_{\lam,\zeta}$ and $\nu_{H,\lam,\zeta}$. For a finite state space $\Omega$, a Markov chain defined by a kernel $P : \Omega \times \Omega \rightarrow [0,1]$ is a sequence of random variables $(Y_i)_{i \in \mathbb{N}}$ such that for each $i \geq 1$ and $z \in \Omega$, we have
\[
\mathbb{P}(Y_i = z|Y_{i-1},\ldots,Y_0) = \mathbb{P}(Y_i = z |Y_{i-1}) = P(Y_{i-1},z).
\]
We often use $P$ to refer to a Markov chain as well as its kernel. For $y \in \Omega$, we use $Py$ to denote the random variable distributed as
\[
\mathbb{P}(Py = z) = P(y,z)
\]
for $z\in\Omega$.
Thus a Markov chain $P$ on $\Omega$ is determined by random variables $\{Py\}_{y \in \Omega}$. We extend this notation to the case where $Y$ is a random variable: the distribution of $PY$ is given by $\mathbb{P}(PY = z) = \E P(Y,z)$. We denote $P^1Y = PY$ and $P^tY = P\cdot P^{t-1}Y$ for $t > 1$. A distribution $\nu$ is said to be stationary for the Markov chain $P$ if for $Y$ distributed according to $\nu$, we have $PY \stackrel{d}= Y$.  

Suppose that the state space $\Omega$ is also equipped with a distance  $\dist(\cdot,\cdot)$. Given two $\Omega$-valued random variables $Y,Z$, the \emph{Wasserstein $1$-distance} between $Y, Z$ is
\[
W_1(Y, Z)= \inf_\pi\E_\pi(\dist(Y,Z))
\]
where the infimum is taken over all couplings $\pi$ of $Y, Z$.

For $\rho > 0$, we say that a Markov chain $P$ is $\rho$-contractive with respect to $\dist(\cdot,\cdot)$ on $S \subseteq \Omega$ if 
\[
\frac{W_1(Py,Pz)}{\dist(y,z)} \leq (1 - \rho) ~\text{for every }~y,z \in S,~y \neq z.
\]
When $P$ is $\rho$-contractive on the whole state space $\Omega$, we simply say that it is $\rho$-contractive. Contractive Markov chains arise in the study of mixing times~\cite{bubley1997path,dyer2003randomly}. It is known that if $P$ is $\rho$-contractive on $\Omega$, then the stationary distribution $\nu$ is unique, and for $t > 0$ and any $y \in \Omega$ we have
\begin{equation}\label{eqn:mixingtime}
W_1\left(\nu,P^ty\right) \leq \text{diam}(\Omega)\cdot (1 - \rho)^t.
\end{equation}

Contractive Markov chains have recently been studied by Ollivier~\cite{ollivier2007curvature} (in the language of~\cite{ollivier2007curvature}, `$\rho$-contractive' is equivalent to `coarse Ricci curvature at least $\rho$') and, independently, by Luczak~\cite{luczak2008concentration} in the context of proving concentration of measure of well behaved functions on $\Omega$. This line of work was further developed in~\cite{paulin2015concentration,barbour2019longterm}, and provides  concentration inequalities suited to our purposes. 

Our results crucially rely on the following concentration result for contraction Markov chains. For a Markov chain $(Y_i)_{i \in \mathbb{N}}$ on $\Omega$ and $y\in \Omega$ we let
\[
\mathbb{P}_y(\cdot) = \mathbb{P}(\cdot~|Y_0 = y)~\text{and }\mathbb{E}_y[\cdot] = \mathbb{E}[\cdot~|Y_0 = y].
\]
More generally for a distribution $\mu$ on $\Omega$, we let
\[
\mathbb{P}_{\mu}(\cdot) = \mathbb{E}[\mathbb{P}_{Y_0}(\cdot)]~\text{and }\mathbb{E}_{\mu}[\cdot] = \mathbb{E}\left[\mathbb{E}_{Y_0}[\cdot]\right]
\]
where $Y_0$ is distributed according to $\mu$.

For $y\in\Omega$ we denote $N(y) = \{z : P(y,z) > 0\}$ and for $S\subseteq \Omega$ we let $S^{+}=S\cup\left(\cup_{z\in S}N(z) \right)$.

We record the following special case of \cite[Theorem 2.3]{barbour2019longterm}.

\begin{theorem}\label{thm:barbour2019longterm}
Let $P=(Y_i)_{i \in \mathbb{N}}$ be a Markov chain on $\Omega$ equipped with distance $\dist( \cdot,\cdot)$, and let $f : \Omega \rightarrow \mathbb{R}$ be a $L$-Lipschitz function with respect to $\dist$. Suppose that $\Omega_{\text{typ}}\subseteq \Omega_{\text{con}} \subseteq \Omega$ are such that:
\begin{enumerate}
\item $\dist(y,z) \leq 1$ for all $y,z\in\Omega$ such that  $P(y,z) > 0$, \label{MC1}
\item $\P(Py\neq y) \leq D$ for $y \in \Omega_{\text{typ}}$,\label{MC2} 
\item $\Omega_{\text{typ}}^+  \subseteq \Omega_{\text{con}}$,\label{MC3}
\item $P$ is $\rho$-contractive on $\Omega_{\text{con}}$.\label{MC4} 
\end{enumerate}
Let $k$ be a positive integer and let
\[\mathcal{A}_k = \{Y_i \in \Omega_{\text{typ}}~\forall\,  i<k\}, \text{ and } e_k = \max_{y\in \Omega_{\text{typ}}^+}\mathbb{P}_y(Y_i\not \in \Omega_{\text{con}} \textup{ for some }i<k)\, .
\]

Then for all $y\in \Omega_{\text{typ}}$ and $m \geq 0$,
\[
\mathbb{P}_y(\{|f(Y_k) - \mathbb{E}_y[f(Y_k)]| > m\}\land \mathcal{A}_k) \leq 2\exp\left(-\frac{m^2}{4L^2(D\rho^{-1} + 12k^3e_k^2) + 4mL(1 + 6ke_k)}\right)\, .
\]
\end{theorem}

In the setting of Theorem~\ref{thm:barbour2019longterm}, we show a generalization of~\eqref{eqn:mixingtime} using a standard ``burn-in'' argument.

\begin{lemma}\label{lem:mixingtime}
In the setting of Theorem~\ref{thm:barbour2019longterm},
let $y_0\in\Omega_{\text{typ}}$ and let  $P=(Y_i)_{i \in \mathbb{N}}$ where we set $Y_0=y_0$.
Let $\nu$ denote the stationary distribution of $P$. For $k \geq 1$, we have
\[
W_1(\nu,Y_k) \leq \textup{diam}(\Omega)\left((1 - \rho)^{k} + 2k e_k + \nu(\Omega\setminus \Omega_{\text{typ}})\right).
\]
As a consequence, we have
\[
\left|\nu(f) - \mathbb{E}\left[f(Y_k)\right]\right| \leq L\cdot \textup{diam}(\Omega)\left((1 - \rho)^{k} + 2k e_k + \nu(\Omega\setminus \Omega_{\text{typ}}\right)\, ,
\]
where $\nu(f)$ denotes the expectation of $f$ with respect to $\nu$.
\end{lemma}

\begin{proof}
First we observe that
\begin{align*}
W_1(\nu,Y_k) &  \leq \nu(\Omega\setminus \Omega_{\text{typ}}) \cdot \text{diam}(\Omega) + \max_{z_0 \in \Omega_{\text{typ}}}W_1(P^kz_0,Y_k).
\end{align*}
Thus it suffices to show that for any $z_0 \in \Omega_{\text{typ}}$ we have
\[
W_1\left(P^kz_0,Y_k\right) \leq \text{diam}(\Omega)\left((1 - \rho)^k + 2k e_k\right).
\]
Let us define $Z_i=P^iz_0$ for $i\geq 1$. For $t \in [k]$, define the event $\mathcal{E}_t = \{Y_t,Z_t \in \Omega_{\text{con}}\}$ and $\mathcal{F}_t  = \bigcap_{i \leq t}\mathcal{E}_i$. Then under optimal coupling, we have
\begin{align*}
\mathbb{E}[\dist(Y_t,Z_t)] & \leq \mathbb{E}[\dist(Y_t,Z_t)|\mathcal{E}_{t-1}] + \mathbb{P}(\overline{\mathcal{E}_{t-1}})\cdot \text{diam}(\Omega) \\
& \leq \mathbb{E}[\dist(Y_t,Z_t)|\mathcal{E}_{t-1}]  + \mathbb{P}(\overline{\mathcal{F}_{k-1}})\cdot \text{diam}(\Omega)  \\
& \leq (1-\rho) \mathbb{E}[\dist(Y_{t-1},Z_{t-1})] + 2e_k\cdot \text{diam}(\Omega) .
\end{align*}
The second inequality holds since $\mathcal{F}_{k-1} \subseteq \mathcal{F}_{t-1} \subseteq \mathcal{E}_{t-1}$. The third inequality holds since $P$ is $\rho$-contractive on $\Omega_{\text{con}}$, and since on $y_0 \in \Omega_{\text{typ}}$ we have 
\[
\mathbb{P}(\overline{\mathcal{F}_{k-1}}) \leq \mathbb{P}\left(\exists~t<k~\text{s.t.}~Y_t \not\in \Omega_{\text{con}}\right) + \mathbb{P}\left(\exists~t<k~\text{s.t.}~Z_t \not\in \Omega_{\text{con}}\right) \leq 2e_k.
\]
Thus, using $\dist(y_0,z_0) \leq \text{diam}(\Omega)$, we have
\[
\mathbb{E}[\dist(Y_k,Z_k)] \leq \text{diam}(\Omega)\left((1 - \rho)^k + 2ke_k\right) \,. 
\]
We now prove the second half of the lemma. For $z\in \Omega$, let $\pi(z)$ denote a  coupling of the random variables $P^kz$ and $Y_k=P^ky_0$ such that $W_1(P^kz, Y_k)=\E_{\pi(z)}\dist(P^kz,Y_k)$. We note that since $\nu$ is the stationary distribution of $P$ we have $\nu(f)=\E_{z\sim \nu}\E[f(P^kz)]$. Since $f$ is $L$-Lipschitz we conclude that
\begin{align*}
|\nu(f) - \mathbb{E}[f(Y_k)]| & = |\E_{z\sim\nu}\mathbb{E}[f(P^kz)] - \mathbb{E}[f(Y_k)]| \\
&= \left|\E_{z\sim\nu}\left(\mathbb{E}_{\pi(z)}[f(P^kz) - f(Y_k)]\right)\right|\\
&\leq \max_{z \in \Omega_{\text{typ}}}\mathbb{E}_{\pi(z)}|f(P^kz) - f(Y_k)| + \nu(\Omega\setminus \Omega_{\text{typ}})\cdot L\cdot \text{diam}(\Omega) \\
&\leq L\cdot \max_{z \in \Omega_{\text{typ}}}\mathbb{E}_{\pi(z)} \dist(P^kz, Y_k) + \nu(\Omega\setminus \Omega_{\text{typ}})\cdot L\cdot \text{diam}(\Omega)\\
& \leq L\cdot \text{diam}(\Omega)\left((1 - \rho)^k + 2ke_k + \nu (\Omega\setminus \Omega_{\text{typ}})\right) \,. 
\qedhere
\end{align*}

\end{proof}

We continue using the notation of Theorem~\ref{thm:barbour2019longterm}. Suppose that $\nu$ is the stationary distribution of $P$. We say that $P$ is reversible if for any $x,y \in \Omega$, it holds that $\nu(x)P(x,y) = \nu(y)P(y,x)$ (these are the `detailed balance' equations). Often, we would like to understand (usually upper-bound) the probability of an event for a random $Y \in \Omega$ sampled from the distribution $\nu$. One natural, though indirect, approach  to sample $Y$ from $\nu$ is to first sample $Y_0$ from $\nu$, run $P = \{Y_i\}_{i \geq 0}$ for $k$ steps, and then set $Y = Y_k$. Since $\nu$ is stationary for $P$, this gives the correct distribution for $Y$.

\begin{lemma}\label{lem:Gibbsprobability}
In the setting of Theorem~\ref{thm:barbour2019longterm}, let us further suppose that $P$ is reversible with stationary distribution $\nu$, and let $\tilde{\Omega}_{\text{typ}} \subseteq \Omega_{\text{typ}}$. Then for any event $\mathcal{Q} \subseteq \Omega$, we have
\begin{equation}
\nu\left( \mathcal{Q}\wedge \tilde{\Omega}_{\text{typ}}\right) \leq \frac{\mathbb{P}_{\nu}\left(\{Y_k \in \mathcal{Q}\} \land \mathcal{A}_k\right)}{\min_{y \in \tilde{\Omega}_{\text{typ}}}\mathbb{P}_y(\mathcal{A}_k)}.
\end{equation}
\end{lemma}

\begin{proof}
 For $i\geq 0$, let $\mathcal{Q}_i$ denote the event $\{Y_i \in \mathcal{Q}\}$. 
Since $\tilde{\Omega}_{\text{typ}} \subseteq \Omega_{\text{typ}}$ and $P$ is reversible, we have
\begin{equation}
\mathbb{P}_{\nu}(\mathcal{Q}_k \land \mathcal{A}_k)  = \mathbb{P}_{\nu}(\mathcal{Q}_0 \land \mathcal{A}_k) 
 \geq \mathbb{P}_{\nu}(\mathcal{Q}_0 \land (Y_0 \in \tilde{\Omega}_{\text{typ}}))\min_{y \in \tilde{\Omega}_{\text{typ}}}\mathbb{P}_y (\mathcal{A}_k) \, .  \qedhere
\end{equation}
\end{proof}

We now define a pair of Markov chains that are central to our arguments. Recall that a Markov chain $P$ on $\Omega$ is determined by the random variables $\{Py\}_{y \in \Omega}$. Given $\lambda \ge 0$ and $\zeta \in [0,1]$ we define the Markov chain $P_{\lam,\zeta}$ as follows.

\begin{definition}[The Markov chain $P_{\lambda,\zeta}$]\label{defn:Plamzeta}
The state space $\Omega_n$ is the set of all labelled graphs on vertex set $[n]$. For any $y \in \Omega$, we set
\[P_{\lambda,\zeta} y = P^e_{\lambda,\zeta} y\] 
where $e \in \binom{[n]}{2}$ is chosen uniformly at random, and
\[
P_{\lambda,\zeta}^e y = \begin{cases}
y \cup \{e\} &\text{w.p. $\frac{\lambda (1 -\zeta)^{d(e)}}{1 + \lambda (1 - \zeta)^{d(e)}}$, and}\\
y \setminus \{e\} &\text{w.p. $\frac{1}{1 + \lambda (1 - \zeta)^{d(e)}}$,}
\end{cases}
\]
where $d(e)=d_y(e)$ denotes the number of triangles in the graph $y\cup \{e\}$ that contain $e$.
\end{definition}
Thus in one step, $P_{\lambda,\zeta}$ chooses a uniformly random $e \in \binom{[n]}{2}$ and resamples the potential edge $e$ according to $\mu_{\lambda,\zeta}$ conditioned on the graph $y\backslash\{e\}$. One may check that $\mu_{\lambda,\zeta}$ is the unique stationary distribution of $P_{\lambda,\zeta}$. In fact, $P_{\lambda,\zeta}$ belongs to a broad class of Markov chains called {\em Glauber dynamics}, which may be tailored to any multivariate distribution. Given a graph $H$, $\lam \ge0$ and $\zeta\in[0,1]$ we also use the following instantiation of the Glauber dynamics for $\nu_{\lambda,\zeta}^H$.

\begin{definition}[The Markov chain $P_{H,\lambda,\zeta}$]\label{defn:PHlamzeta}
The state space is $\Omega_H = 2^{V(H)}$. For any $y \in \Omega$, we set
\[P_{H,\lambda,\zeta} y = P^u_{H,\lambda,\zeta} y\] 
where $u \in V(H)$ is chosen uniformly at random, and
\[
P_{H,\lambda,\zeta}^u y = \begin{cases}
y \cup \{u\} &\text{w.p. $\frac{\lambda (1 - \zeta)^{d(v)}}{1 + \lambda (1 - \zeta)^{d(v)}}$, and}\\
y \setminus \{u\} &\text{w.p. $\frac{1}{1 + \lambda (1 - \zeta)^{d(v)}}$}
\end{cases}
\]
where $d(v)=d_y(v)$ denotes the degree of $v$ in the subgraph of $H$ induced by $y\cup \{v\}$.
\end{definition}
Thus in one step, $P_{H,\lambda, \zeta}$ chooses a uniformly random $u \in V(H)$ and resamples membership of $u$ conditioned on $y\backslash \{u\}$. One may again check that $\nu_{H,\lambda,\zeta}$ is the unique stationary distribution of $P_{H,\lambda, \zeta}$. For both $P_{\lambda,\zeta}$ and $P_{H,\lambda,\zeta}$, we will work with pairs of parameters $(\lambda,\zeta)$ that ensure the contraction criterion with respect to the Hamming distance (denoted hereafter by $\dist( \cdot,\cdot)$) which we establish with the help of the following. 
\begin{lemma}\label{lem:PCneighbors}
Let $P$ be either $P_{\lambda,\zeta}$ or $P_{\lambda, \zeta, H}$, and let $S \subseteq \Omega$ be closed under taking vertex subsets or edge subsets, respectively. Furthermore, suppose there is a $\rho > 0$ such that every $y,z \in S$ with $\dist(y,z) = 1$ admits a coupling of $(Py,Pz)$ such that
\begin{equation}\label{eqn:1stepcoupling}
\mathbb{E}[\dist(Py,Pz)] \leq 1 - \rho.
\end{equation}
Then $P$ is $\rho$-contractive on $S$.
\end{lemma}
\begin{proof}
This is essentially the path coupling technique of Bubley and Dyer~\cite{bubley1997path} but applied just to a subset of the state space with the above closure property.
For any $y,z \in \Omega$ such that $\dist(y,z) = s$, there is a sequence $y = y_0,y_1,\ldots,y_s = z$ such that for every $i \in [s]$, either $y_i \subseteq y_0$ or $y_i \subseteq y_s$. In particular, if $y,z \in S$, then each $y_i \in S$. Consider the coupling between $(Py,Pz)$ obtained by sampling each $Py_i$ conditioned on $Py_{i-1}$ for $i \in [s]$ given by the coupling satisfying~\eqref{eqn:1stepcoupling}. We then have
\[
\mathbb{E}[\dist(Py,Pz)] \leq \sum_{i = 1}^s\mathbb{E} [\dist(Py_{i-1},Py_i)] \leq s(1-\rho) \,. \qedhere
\]
\end{proof}

\section{Local Conditioning}
\label{secLocalCondition}

In this section we give the proof of Theorem~\ref{thmLowerTailAsymp}, modulo the proofs of three lemmas that we defer to Sections~\ref{sec:concentration},~\ref{secRegularity}, and~\ref{secLowerTailPartition}.

Throughout this section we will work under the following condition of Theorem~\ref{thmLowerTailAsymp}.

\begin{condition}
\label{ConditionThm1}
    Let $\eta \in [0,1)$  and $c \in (0, \overline c(\eta))$ be fixed.  Define $\zeta$ via~\eqref{eqZetaDef}. Suppose $p =(1+o(1))c/\sqrt{n}$, so that with $\lam = p/(1-p)$ we also have $\lam =(1+o(1))c/\sqrt{n}$.
\end{condition}
This condition gives bounds on the parameters $c, \zeta$ which will allow us to apply the cluster expansion in what follows.

\begin{claim}\label{claim:zeta}
    The equation~\eqref{eqZetaDef} defines $\zeta$ uniquely, and under Condition~\ref{ConditionThm1} we have 
    \begin{equation}
    \label{eqZetaCondition}
        e \zeta c^2  < 1 \,.
    \end{equation}
\end{claim}
\begin{proof}
    For any fixed $c>0$, the LHS of~\eqref{eqZetaDef} is a continuous, decreasing function of $\zeta$ mapping $[0,1]$ to $[0,1]$ and thus there is a unique solution $\zeta=\zeta(c,\eta)$.   

   We will show that $\zeta c^2$ is a strictly increasing function of $c$. We then prove that $(i)$ if $\eta<\eta^\ast = \left(\frac{W(2/e)}{2/e}\right)^{3/2}$ and $c=\overline c(\eta)$, then $\zeta c^2=1/e$ and $(ii)$  $\ell:=\lim_{c\to\infty}\zeta c^2$ exists and if $\eta\geq\eta^\ast$ then $\ell\leq 1/e$. The inequality~\eqref{eqZetaCondition} then follows. 
   
   To show that $\zeta c^2$ is increasing, we differentiate~\eqref{eqZetaDef} with respect to $c$ and obtain
\[
\frac{\partial \zeta}{\partial c}=-\frac{6 \, W(2 \zeta c^2 ) \, (1 - \zeta) \, \zeta}{2 c \zeta+ c \, W(2 \zeta c^2 ) \, (3 - \zeta) }
\]
and so 
\[
\frac{\partial}{\partial c}(\zeta c^2)= 
c^2 \frac{\partial \zeta}{\partial c} + 2\zeta c=\frac{4 c \zeta^2 (1 + W(2 \zeta c^2 ))}{2 \zeta + (3 - \zeta) W(2 \zeta c^2 )}>0\, .
\]
Suppose now that $\eta<\eta^\ast$ and let $c=\overline c(\eta)=(e(1-\eta/\eta^*))^{-1/2}$. With this choice of $c$ \eqref{eqZetaDef} yields $\zeta=1-\eta/\eta^\ast$ and so $\zeta c^2=1/e$.

Next  observe that as $x\to\infty$, $W(x)/x\to 0$. Since $\eta>0$, it follows from~\eqref{eqZetaDef} that as $c\to\infty$ we must have that $\zeta c^2$ stays bounded (and so in particular $\zeta\to 0$). Since $\zeta c^2$ is increasing in $c$ we conclude that $\ell=\lim_{c\to\infty}\zeta c^2$ exists (and $\ell<\infty$).  Taking the limit $c\to \infty$ in~\eqref{eqZetaDef} we conclude that 
\[
\left( \frac{ W(2 \ell)}{2 \ell} \right)^{3/2}= \eta\, .
\]
Since $W(x)/x$ is decreasing in $x$ and $\eta\geq \eta^\ast=\left(\frac{W(2/e)}{2/e}\right)^{3/2}$, we conclude that $\ell\leq 1/e$ as desired. 
\end{proof}

The proof of Theorem~\ref{thmLowerTailAsymp} has two steps.  First, we estimate the expected and typical number of edges and triangles in a sample from $\mu_{\lam, \zeta}$ using `local conditioning'.  This step will rely on the concentration inequalities from Markov chains described in Section~\ref{subsecMCMC}.  

Second, we reduce the estimation of the lower-tail probability to the estimation of the expected edge density in the Gibbs measure $\mu_{\lam,\zeta}$, and then apply the results of the first step.  In this step we  see that $\zeta$ is chosen so that the expected number of triangles in a sample from $\mu_{\lam,\zeta}$ is asymptotically $\eta \E_p X$, the target number of triangles in the lower-tail event.  

\subsection{Triangle and edge density estimation via local conditioning}

The main technical results towards the proof of Theorem~\ref{thmLowerTailAsymp} are the following formulas for first-order asymptotics of the expected and typical number of edges and triangles in a sample from $\mu_{\lam,\zeta}$.
\begin{lemma}\label{lemDensityEstimate}
Under Condition~\ref{ConditionThm1}, 
\begin{align}
\label{eqExpectedEdgesFormula}
\E_{\lam,\zeta}|G| &= (1+o(1)) \frac{1}{2} \sqrt{\frac{W(2\zeta c^2)}{2\zeta}} n^{3/2} 
\intertext{and}
\label{eqExpectedTrianglesFormula}
\E_{\lam,\zeta} X(G) &= (1+o(1)) \frac{1-\zeta}{6} \left(  \frac{W(2\zeta c^2)}{2\zeta}\right)^{3/2}    n^{3/2} \,.
\end{align}
Moreover, $|G|=(1+o(1))\E_{\lam,\zeta}|G|$ and $X(G)=(1+o(1))\E_{\lam,\zeta} X(G)$ with high probability with respect to $\mu_{\lam,\zeta}$.
\end{lemma}

From the statement of Lemma~\ref{lemDensityEstimate} we can see how the choice of $\zeta$ in~\eqref{eqZetaDef} arises: under Condition~\ref{ConditionThm1}, $\E_{\lam,\zeta} X(G) = (1+o(1)) \eta \E_p X(G)$.

To prove Lemma~\ref{lemDensityEstimate}, we will use the local conditioning strategy described in Section~\ref{secTechniques}, now generalized to $\mu_{\lam,\zeta}$: we will pick a vertex $v$ among the $n$ vertices, and condition the measure $\mu_{\lam,\zeta}$ on all edges not incident to $v$.  Using concentration inequalities and the cluster expansion, we will be able to write two expressions for $\E_{\lam,\zeta} |G|$, giving an equation that determines this value to first order.

Given a graph $G$ and vertex $v\in V(G)$, we denote by $G_v$  the graph $G-\{v\}$. Recall that $d_G(v)$ denotes the degree of $v$ in $G$, and $t_G(v)$ denotes the number of triangles of $G$ incident to $v$.
Our estimates of $\E_{\lam,\zeta}|G|$  and $\E_{\lam,\zeta}X(G)$ rely on the following fact which allows us to analyze vertex degrees and triangle counts in a sample from $\mu_{\lam,\zeta}$. Recall that for a graph $H$, $\nu_{H,\lam,\zeta}$ denotes the anti-ferromagnetic Ising measure on $H$, defined in~\eqref{eqNulamBetaDef}. 
\begin{lemma}\label{lemLocalCond}
Let $\lam\ge 0$, $\zeta \in [0, 1]$, and let $G\sim \mu_{\lam,\zeta}$. Let $v\in V(G)$. Conditioned on the event $\{G_v=H\}$, $d_G(v)$ is distributed as $|S|$ where $S\sim \nu_{H,\zeta,\lam}$, and $t_G(v)$ is distributed as $|E(S)|$, the number of edges of $H$ induced by the random set $S$.  
\end{lemma}
\begin{proof}
Note that $V(H)= V(G)\backslash\{v\}$. Let $E(v)=\{\{v,x\}: x\in V(H)\}$, the set of all potential edges incident to $v$. Let $\psi : V(H) \to E(v)$ denote the bijection $x\mapsto \{v,x\}$. Then for any $S\subseteq V(H)$,
\[
\P(G \cap E(v)= \psi(S) \mid G_v=H)\propto \lam^{|S|}(1-\zeta)^{|E(S)|} \,. \qedhere
\]
\end{proof}

The above lemma allows us to access information about the degree of a vertex $v$ (and the number of triangles it is incident to) in $G\sim \mu_{\lam,\zeta}$ by studying $\nu_{H,\zeta,\lam}$ on typical instances $H$ of the graph $G_v$. The advantage here is that while the model $\mu_{\lam,\zeta}$ has three-way interactions (penalized substructures are sets of three edges forming  a triangle), the Ising model $\nu_{H,\lam,\zeta}$ has only pairwise interactions, and we can analyze this model precisely using the cluster expansion. 

Given a graph $H$ and $\lam \ge 0$, $\zeta\in [0,1]$, let
\begin{equation}
\label{eqAlphaHdef}
    \alpha_{H,\lam,\zeta}=\frac{1}{|V(H)|}\E_{H,\lam,\zeta}(|S|)
\end{equation}
be the expected fraction of vertices of $H$ in the set $S$ drawn from the measure $\nu_{H,\lam,\zeta}$.  Let
\begin{equation}
\label{eqRhoHdef}
    \rho_{H,\lam,\zeta} = \frac{1}{|V(H)|}\E_{H,\lam,\zeta}(|E(S)|)
\end{equation}
be the expected number of edges induced by $S$ normalized by the number of vertices of $H$.

We record the following immediate corollary of Lemma~\ref{lemLocalCond}. 
\begin{cor}\label{corLocalCond}
Let $\lam \ge 0$, $\zeta\in [0,1]$, and let $G\sim \mu_{\lam,\zeta}$. Then
\begin{align}
   \E_{\lam, \zeta} |G| &=\frac{n-1}{2}\sum_{v\in V(G)} \E_{\lam,\zeta}(\alpha_{G_v,\lam,\zeta}) \, ,
    \intertext{and}
    \E_{\lam,\zeta} X(G) &=\frac{n-1}{3} \sum_{v\in V(G)} \E_{\lam,\zeta}(\rho_{G_v,\lam,\zeta}) \,.
\end{align}
\end{cor}

To analyze $\alpha_{G_v,\lam,\zeta}$ and $\rho_{G_v,\lam,\zeta}$ we will use the cluster expansion. To apply this, we require that the maximum degree of $G_v$ is not too large. We will control the maximum degree of $G$ (hence also $G_v$) using stochastic domination (Lemma~\ref{lemStochDom}).  This allows us to conclude that if $G\sim \mu_{\lam,\zeta}$, then $\Delta(G)\leq (1+o(1))n\lam$ whp. Therefore, if $\lam, \zeta$ are such that $\lam \zeta \le 1/(\gamma n\lam)$ for some $\gamma>e$ (\textit{cf.}~Lemma~\ref{lemClusterTail}) we are in a position to apply the cluster expansion to analyze the Ising model $\mu_{H,\lam,\zeta}$ for any subgraph $H\subseteq G$.

Define the functions
\begin{align}\label{eqalphadef}
    \alpha(\Delta,\lam,\zeta) &= \frac{W(\zeta \lam \Delta)}{\zeta \Delta}  
    \intertext{and}
    \label{eqrhodef}
    \rho(\Delta,\lam,\zeta) &=  \frac{1}{2} \Delta (1-\zeta) \, \alpha(\Delta,\lam,\zeta)^2  \,.  
\end{align}

For a graph $H$, recall that 
$\delta(H)$ denotes its minimum degree and $\Delta(H)$ its maximum  degree. 
Let $C_\ell$ denote the cycle on $\ell$ vertices, and let $C_\ell(H)$ denote the number of copies of $C_{\ell}$ in $H$. The following lemma shows that if a graph $H$ `looks like' a regular tree (in the sense that it is approximately regular and contains few short cycles) then $\alpha_{H,\lam,\zeta}$ and $\rho_{G,\lam,\zeta}$ essentially only depend on $\Delta, \lam, \zeta$ (provided that these parameters are in the cluster expansion regime).

\begin{lemma}\label{lemOccTree}
Fix $M> 0$, $\gamma \in [0,1)$, and $\vartheta>0$. Fix $\zeta \in (0,1]$ and let $H$ be a graph on $n$ vertices such that $\Delta = \Delta(H)=\omega(1)$ and 
\begin{equation}\label{eq:HAlmostReg}
 \delta(H) = (1+o(1)) \Delta \,.   
\end{equation}
Let $k=\omega(1)$ and suppose that  for all $\ell\leq k$
\[
C_\ell(H)\leq nM^{\ell}  \Delta^{\ell-2}\, .
\]
Then if 
\begin{equation}
  \vartheta \le  e\lam \zeta \Delta \le \gamma \, ,
\end{equation}
we have the following:
\begin{align}\label{eqTreeAlph}
 \alpha_{H,\lam,\zeta}=(1+o(1)) \alpha(\Delta,\lam,\zeta) \,, 
 \intertext{and}
 \label{eqTreeRho}
 \rho_{H,\lam,\zeta} = (1+o(1)) \rho(\Delta,\lam,\zeta) \, .
\end{align}
\end{lemma}

\begin{remark}
Note that the RHS of~\eqref{eqTreeAlph} and~\eqref{eqTreeRho} depend only on $\Delta, \lam,\zeta$. Once suitable definitions are given, this suggests (via the connection of Gibbs measures measures on locally tree-like graphs to Gibbs measures on infinite trees, e.g.~\cite{dembo2013factor,SS14}) that $\alpha(\Delta,\lam,\zeta)$ should equal $(1+o(1))\alpha_{\mathbb T_{\Delta},\lam,\zeta}$, where $\alpha_{\mathbb T_\Delta,\lam,\zeta}$ is the probability that the root of the infinite $\Delta$-regular tree is occupied under $\nu_{\mathbb T_\Delta,\lam,\zeta}$ (when $\lam < \lam_c(\Delta,\zeta)\sim e/(\Delta \zeta ))$ this probability is well defined). In fact, this is how we first arrived at the function $\alpha(\Delta,\lam,\zeta)$. 
\end{remark}

We will prove Lemma~\ref{lemOccTree} in Section~\ref{secRegularity}. We will also need the following concentration inequalities for the degrees, number of edges, and the number of triangles in $G$, respectively, which we will prove in Section~\ref{sec:concentration}. 

Recall that we say a sequence of events $A_n$ holds with very high probability (wvhp) with respect to a probability measure $\P$ if $\mathbb P(A_n) = 1-n^{-\omega(1)}$.

\begin{lemma}
\label{lem:fullconcentration}
Under Condition~\ref{ConditionThm1}, the following events hold wvhp with respect to $G \sim \mu_{\lam,\zeta}$.
\begin{enumerate}[itemsep=10pt]
\item\label{degconcentration} 
$\Delta(G) - \delta(G) \leq 4\sqrt{n\lam}\log n$,
\item\label{edgeconcentration} $||G|- \mathbb{E}_{\lam,\zeta}|G||\leq  n\sqrt{\lambda}\log n$, and
\item\label{triangleconcentration} $ \left | X(G) - \E_{\lam,\zeta} X(G) \right | \leq  n^2\lambda^{3/2}\log^3 n$.
\end{enumerate}
\end{lemma}

Before turning to the proof of Lemma~\ref{lemDensityEstimate}, we record the following corollary to Lemmas~\ref{lemOccTree} and~\ref{lem:fullconcentration}.

\begin{cor}\label{cor:GHProps}
Under Condition~\ref{ConditionThm1}, the following events hold wvhp with respect to $G \sim \mu_{\lam,\zeta}$. Setting $\overline \Delta:=\E_{\lambda,\zeta} \Delta(G)$,
\begin{enumerate}
\item  $\delta(G) = (1+o(1))\Delta(G) =(1+o(1))\overline\Delta$,
\item $\alpha_{G_v,\lam,\zeta}=(1+o(1))\alpha(\overline\Delta,\lam,\zeta)$ for all $v\in V(G)$,
\item $\rho_{G_v,\lam,\zeta}=(1+o(1))\rho(\overline\Delta,\lam,\zeta)$ for all $v\in V(G)$.
\end{enumerate}
In particular, $\E_{\lambda,\zeta}(\alpha_{G_v,\lam,\zeta})=(1+o(1))\alpha(\overline\Delta,\lam,\zeta)$ and $\E_{\lambda,\zeta}(\rho_{G_v,\lam,\zeta})=(1+o(1))\rho(\overline\Delta,\lam,\zeta)$ for all $v\in V(G)$.
\end{cor}
\begin{proof}
All expectations and probabilities in the following proof will be with respect to $\mu_{\lam,\zeta}$.  Recall that we let $\overline d(G)$ denote the average degree of $G$. 
By Part (2) of Lemma~\ref{lem:fullconcentration} and Lemma~\ref{lemDegLB} we have that $\E \overline d(G)=\Omega(n^{1/2})$ and so by Part (2) of Lemma~\ref{lem:fullconcentration} we have $\overline d(G)=(1+o(1))\E \overline d(G)$ wvhp. The first item now follows from (1) of Lemma~\ref{lem:fullconcentration}.

For the second and third items we fix $v\in V(G)$ and show that $H=G_v$ satisfies the conditions of Lemma~\ref{lemOccTree} wvhp. Henceforth, any assertion we make about $H$ is to be understood to hold with probability wvhp.

First note that since $H$ and $G$ differ in a single vertex we have $\delta(H)=(1+o(1))\Delta(H)$ by the item (1). Moreover, by item (1) and Lemmas~\ref{lemStochDom} and~\ref{lemDegLB}, we have
\begin{align}\label{eq:DeltaHBd}
\Omega(np)=\Delta(H)\leq (1+o(1))np\, .
\end{align}
Next let $k=\omega(1)$ and let $\ell\leq k$. By Lemma~\ref{lemStochDom}  and a concentration inequality for subgraph counts in $G(n,p)$ due to Vu~\cite{vu2001large},
\begin{equation}
  \P_p \left(C_\ell(H)\geq 2n^\ell p^\ell \right)\leq \exp\left(-c_\ell (np)^{\ell/(\ell-1)} \right)  
\end{equation}
for some $c_\ell>0$. We choose $k$ growing sufficiently slowly so that $c_\ell\geq 1/\log(n)$ for all $\ell\leq k$ and $n$ sufficiently large. Applying a union bound over $\ell\leq k(\leq n)$ we conclude that $C_\ell(H)\leq 2n^\ell p^\ell$ for all $\ell\leq k$. 
We conclude from the lower bound in~\eqref{eq:DeltaHBd} that there exists a constant $M>0$ such that $C_\ell(H)\leq  M^\ell n\Delta(H)^{\ell-2}$ for all $\ell\leq k$. Moreover, by~\eqref{eq:DeltaHBd} 
\[
\Omega(1)\leq \lam \zeta \Delta(H) \leq (1+o(1))c^2\zeta\, .
\]
By Claim~\ref{claim:zeta} we conclude that $H$ satisfies the conditions of Lemma~\ref{lemOccTree} and so 
\[
\alpha_{H,\lam,\zeta}=(1+o(1))\alpha(\Delta(H),\lam,\zeta)=(1+o(1))\alpha(\overline\Delta,\lam,\zeta)\, ,
\]
and similarly 
\[
\rho_{H,\lam,\zeta} = (1+o(1)) \rho(\overline\Delta,\lam,\zeta) 
\]
as desired. 

The final claim follows by noting that $\alpha_{H,\lam,\zeta}\leq 1$ and $\rho_{H,\lam,\zeta}\leq n$ deterministically.
\end{proof}

With this corollary in hand we can  prove  Lemma~\ref{lemDensityEstimate}. 
\begin{proof}[Proof of Lemma~\ref{lemDensityEstimate}]

Let $\overline\Delta=\E_{\lam,\zeta} \Delta(G)$ as in Corollary~\ref{cor:GHProps}. Recall that $\overline\Delta=\Theta(\sqrt{n})$ so that we can parameterize $\overline\Delta=b\sqrt{n}$ for $b=\Theta(1)$. 
We conclude from Corollaries~\ref{corLocalCond} and~\ref{cor:GHProps} that 
\begin{equation}
    b\sqrt{n} = (1+o(1))n  \cdot\alpha(b \sqrt{n},\lam,\zeta)   \, .
\end{equation}
Applying the definition of the function  $\alpha$ in~\eqref{eqalphadef},  we have
\[
b = (1+o(1))\frac{1}{b\zeta}W(b\zeta c) \implies b^2\zeta = (1+o(1))W(b\zeta c).
\]
This implies
\[
\zeta b^2 e^{\zeta b^2}=(1+o(1))b\zeta c \implies 2\zeta c^2 = (1+o(1))2\zeta b^2 e^{2\zeta b^2} \implies 2\zeta b^2 = (1+o(1))W(2\zeta c^2),
\]
or $b = (1+o(1))\sqrt{\frac{W(2\zeta c^2)}{2 \zeta }  }, $
which yields the formula~\eqref{eqExpectedEdgesFormula}.

Turning to triangles, again by Corollaries~\ref{corLocalCond} and~\ref{cor:GHProps}, we have 
\begin{align}
    \E_{\lam,\zeta} X(G) &= (1+o(1)) \frac{n^2}{3} \rho( b \sqrt{n}, \lam, \zeta) \\
    &= (1+o(1)) \frac{n^2}{3} \frac{1}{2} b \sqrt{n} \cdot (1-\zeta) \left (\frac{ W(\zeta \lam b \sqrt{n})} {\zeta b \sqrt{n} }  \right) ^2  \\
    &= (1+o(1))  \frac{1-\zeta}{6}  \sqrt{\frac{W(2\zeta c^2)}{2 \zeta }  }  \left (\frac{ W(\zeta c b )} {\zeta b  }  \right) ^2   n^{3/2}   \\
     &= (1+o(1))  \frac{1-\zeta}{6}  \sqrt{\frac{W(2\zeta c^2)}{2 \zeta }  }  \left (\frac{ \zeta  b^2 } {\zeta b  }  \right) ^2   n^{3/2}   \\
    &= (1+o(1)) \frac{1-\zeta}{6} \left(  \frac{W(2\zeta c^2)}{2\zeta}\right)^{3/2}    n^{3/2} \,,
\end{align}
giving~\eqref{eqExpectedTrianglesFormula}.

The concentration statements in the lemma follow immediately from Lemma~\ref{lem:fullconcentration}. 
\end{proof}

\subsection{Reduction to triangle and edge density estimation}

We now show how the estimates in Lemma~\ref{lemDensityEstimate} can be used to approximate the lower-tail probability. 
 
First we relate the lower-tail probability to the partition function $Z(\lam,\zeta)$ and  relate the lower-tail conditional measure to the Gibbs measure $\mu_{\lam,\zeta}$.

For the special case $\eta=0$, we have an identity.
\begin{lemma}
    \label{lem:LTtoZhardcore}
    Suppose $p\in[0,1]$ and set $\lam = p/(1-p)$. Then
    \begin{align}
        \P_p(X=0) = (1-p)^{\binom{n}{2}} Z(\lam, 1)\,,
    \end{align}
    and the measures $\mu_{\lam, 1}$ and $G(n,p)$ conditioned on the event $\{X = 0\}$ are identical. 
\end{lemma}
This follows immediately from~\eqref{eqZlamIntrodef}, noting that $Z(\lam) = Z(\lam, 1)$ and $\mu_\lam = \mu_{\lam, 1}$.

For the case $\eta\in (0,1)$ we have an approximate identity.  
\begin{lemma}\label{lem:LTtoZ}
Under Condition~\ref{ConditionThm1},
\begin{align}\label{eq:PFnTransfer}
 \frac{1}{n^{3/2}}\log \P_p(X\leq \eta \E X)= -\frac{c}{2} - \log (1 -\zeta) \eta \frac{c^3}{6} +  \frac{1}{n^{3/2}}\log Z(\lam,\zeta) +o(1)\,.
\end{align}
Moreover, if $\eps>0$ is fixed and $\mathcal{E}_n$ is any event such that
\begin{equation}
    \mu_{\lam,\zeta}(\mathcal{E}_n) \le e^{- \eps n^{3/2}} \,,
\end{equation}
then
\begin{equation}
    \P_{p} ( \mathcal{E}_n | X \le \eta \E X )\le e^{-\eps n^{3/2}/2} \, .
\end{equation}
\end{lemma}
Importantly, this lemma relies on the fact that under Condition~\ref{ConditionThm1} (via Lemma~\ref{lemDensityEstimate}), the typical number of triangles in a sample from $\mu_{\lam,\zeta}$ is close to $\eta \E_p X$, the target in the lower-tail event.  We prove  Lemma~\ref{lem:LTtoZ} in Section~\ref{secLowerTailPartition}.

Next, we estimate $Z(\lam,\zeta)$ by integrating the expected number of edges.
\begin{lemma}\label{lemZlambetaEst}
Under Condition~\ref{ConditionThm1},
\[
\lim_{n\to\infty} \frac{1}{n^{3/2}}\log Z(\lam,\zeta)=   \frac{W(2 \zeta c^2)^{3/2} +3W(2\zeta c^2)^{1/2}}{6\sqrt{2\zeta}}\, .
\]
\end{lemma}
\begin{proof}
Using Lemma~\ref{lemDerivativeIdentities}, we have 
\begin{align}\label{eqExppartialInt}
\log Z(\lam,\zeta) = \int_{0}^\lam \frac{\E_{\theta,\zeta}|G|}{\theta} d\theta\, .
\end{align}
We then apply Lemma~\ref{lemDensityEstimate},
\begin{align}
    \log Z(\lam,\zeta) &=  \int_{0}^{c(1+o(1))} \frac{\E_{x/\sqrt{n},\zeta}(|G|)}{x} dx \\
& =(1+o(1))\frac{n^{3/2}}{2} \int_{0}^{c(1+o(1))} \sqrt{\frac{W(2\zeta x^2)}{2\zeta x^2}} dx\\
&= (1+o(1))n^{3/2} \cdot\left[  \frac{W(2 \zeta c^2)^{3/2} +3W(2\zeta c^2)^{1/2}}{6\sqrt{2\zeta}}\right]\, . \qedhere
\end{align}
\end{proof}

We can now prove Theorem~\ref{thmLowerTailAsymp}.
\begin{proof}[Proof of Theorem~\ref{thmLowerTailAsymp}]
  We fix $c>0$, $\eta\in[0,1)$ and $\zeta = [0,1)$ satisfying Condition~\ref{ConditionThm1}.   Lemma~\ref{lem:LTtoZ}  (or Lemma~\ref{lem:LTtoZhardcore} in the case $\eta=0$) gives
  \begin{align}
       \frac{1}{n^{3/2}}\log \P_p(X\leq \eta \E X)= -\frac{c}{2} - \log (1 -\zeta) \eta \frac{c^3}{6} +  \frac{1}{n^{3/2}}\log Z(\lam,\zeta) +o(1)
  \end{align}
  then applying Lemma~\ref{lemZlambetaEst} to estimate $\log Z(\lam,\zeta)$  we obtain
  \begin{align}
       \frac{1}{n^{3/2}}\log \P_p(X\leq \eta \E X) &=-\frac{c}{2} - \log (1 -\zeta) \eta \frac{c^3}{6} +  \frac{W(2 \zeta c^2)^{3/2} +3W(2\zeta c^2)^{1/2}}{6\sqrt{2\zeta}} +o(1)\, . \qedhere
  \end{align}
\end{proof}

It remains to prove Lemmas~\ref{lem:fullconcentration}, \ref{lemOccTree} and \ref{lem:LTtoZ} which we do in Sections~\ref{sec:concentration}, \ref{secRegularity} and \ref{secLowerTailPartition} respectively.

\section{Concentration from contraction}\label{sec:concentration}
In this section, we prove Lemma~\ref{lem:fullconcentration} under the following quantitatively weaker condition.
\begin{condition}\label{condition2}
 $\zeta \in (0,1]$  and $c>0$ are fixed such that $\xi:=1-2\zeta c^2>0$ and $\lam =(1+o(1))c/\sqrt{n}$.
\end{condition}
In particular by Claim~\ref{claim:zeta}, Condition~\ref{ConditionThm1} implies Condition~\ref{condition2}.
The proof of Lemma~\ref{lem:fullconcentration} relies on concentration inequalities for the measures $\mu_{\lam,\zeta}$, $\nu_{H,\lam,\zeta}$ which follow from  the fact that the Markov chains $P_{\lam, \zeta}$ and $P_{H,\lam,\zeta}$ (introduced in Section~\ref{subsecMCMC}) contract for suitable choices of $H,\lam,\zeta$. We begin by establishing these contraction properties and then go on to prove the relevant concentration inequalities. 

\subsection{Contraction}
Recall the definition of $P_{\lam,\zeta}$ (Definition~\ref{defn:Plamzeta}) as a Markov chain on $\Omega_n$, the set of all labeled graphs on vertex set $[n]$. Let 
\[
\Omega_{\text{con}} := \left\{y\in\Omega_n : \Delta(y) \leq \lambda n + \sqrt{\lam n} \log n\right\}\, .
\]
\begin{lemma}\label{lemma:contraction}
Let $\lam,\zeta$ satisfy Condition~\ref{condition2} and let $P=P_{\lam,\zeta}$.
For $y,z \in \Omega_{\text{con}}$, we have
\begin{equation}
\frac{W_1(Py,Pz)}{\dist(y,z)} \leq 1-\frac{\xi}{n^2}.
\end{equation}
\end{lemma}
\begin{proof}
For  $x \in \Omega_n$ and $e\in \binom{[n]}{2}$ we let 
\begin{align}\label{eq:betaMarginal}
p_x(e):=\frac{\lam (1-\zeta)^{ d_{x}(e)}}{1+\lam (1 - \zeta)^{ d_{x}(e)}}\, ,
\end{align}
where $d_x(e)$ denotes the number of triangles in the graph $x\cup \{e\}$ that contain $e$.

By Lemma~\ref{lem:PCneighbors}, it suffices to consider the case where $\dist(y, z)=1$, i.e., the graphs $y, z$ differ in one edge. 
Suppose WLOG that $\{i,j\}$ is an edge of $y$, but not $z$. We couple $Py, Pz$ so that they attempt the same updates. More precisely, draw $u\sim U[0,1]$ and $e\in \binom{[n]}{2}$ uniformly at random. If $u\leq p_y(e)$ then we set $Py=y\cup \{e\}$, otherwise set $Py=y\backslash\{e\}$, and similarly for $z$. 

Under this coupling, we have $\dist(Py,Pz)=\dist(y,z)$ unless $(i)$ the update was performed on pair $\{i,j\}$ in which case the distance decreases by $1$, or $(ii)$ an edge (necessarily adjacent to $i$ or $j$) was included in $Py$ but not in $Pz$ in which case the distance increases by $1$. Thus under this coupling
\begin{align}
\E \left[\dist (Py,Pz)\right] 
\leq
1-\frac{1}{\binom{n}{2}}
&+\frac{1}{\binom{n}{2}}\sum_{\ell \in N_{z}(i)}|p_z(\ell j)-p_y(\ell j)|\\
&+\frac{1}{\binom{n}{2}}\sum_{k\in N_{z}(j)}|p_z(ki)-p_y(ki)|\, ,
\end{align}
where  $N_z(i)$ denotes the neighborhood of the vertex $i$ in the graph $z$.
Note that if  $\ell \in N_{z}(i)$, then $d_{y}(\ell j)=d_{z}(\ell j)+1$. It follows from~\eqref{eq:betaMarginal}, and the inequality $|y/(1+y)-z/(1+z)|<|y-z|$ for $y,z>0$, that
\begin{align}
|p_z(\ell j)-p_y(\ell j)|\leq  \lam (1 - \zeta)^{d_{z}(\ell j)}-\lam (1 - \zeta)^{ d_{y}(\ell j)}
&=\lam (1 - \zeta)^{d_{z}(\ell j)}\zeta \leq \lambda \zeta \, .
\end{align}
Similarly
\(
|p_z(k i)-p_y(k i|Y)|\leq \lam \zeta \, ,
\)
for all $k \in N_{z}(j)$. Recalling the definition of Wasserstein distance and the assumption $z\in \Omega_{\text{con}}$, we conclude that 
\begin{align}
W_1(Py,Pz)\leq\E \left[\dist (Py, Pz)\right] & \leq 1-\frac{1}{\binom{n}{2}}+\frac{1}{\binom{n}{2}}(d_{z}(i)+d_{z}(j))\lam \zeta \\
& \leq 1 + \frac{-1 + 2(\lambda n + \sqrt{\lambda n}\log n)\lambda \zeta}{\binom{n}{2}} \leq 1-\frac{\xi}{2\binom{n}{2}}\, ,  
\end{align}
where for the final inequality we used Condition~\ref{condition2}.
\end{proof}

We now consider the Markov chain $P_{H,\lam,\zeta}$ from Definition~\ref{defn:PHlamzeta}. Recall that given a graph $H$, $P_{H,\lam,\zeta}$ is a Markov chain on state space $\Omega_H=2^{V(H)}$.

\begin{lemma}\label{lem:contractionH}
Let $\lam\ge0, \zeta\in[0,1]$ and let $H$ be a graph with maximum degree $\Delta$. Let
 $P=P_{H,\lam,\zeta}$.
Then for $y,z\in\Omega_H$ we have
\begin{equation}
\frac{W_1(Py,Pz)}{\dist(y,z)} \leq 1-\frac{1-\Delta \lam \zeta}{|V(H)|}\, .
\end{equation}
\end{lemma}

\begin{proof}The proof is similar to that of Lemma~\ref{lemma:contraction}.
For  $x \in \Omega_H$ and $v\in V(H)$ we let 
\begin{align}\label{eq:betaMarginal}
p_x(v):=\frac{\lam (1-\zeta)^{ d_{x}(v)}}{1+\lam (1 - \zeta)^{ d_{x}(v)}}\, ,
\end{align}
where $d_x(v)$ denotes the degree of $v$ in the subgraph of $H$ induced by $x\cup \{v\}$.

Again by Lemma~\ref{lem:PCneighbors}, it suffices to consider the case where $\dist(y, z)=1$ i.e., $y$ and $z$ differ in one vertex. Suppose WLOG that $u$ is contained in $y$ but not $z$. As in Lemma~\ref{lemma:contraction} we couple $Py, Pz$ so that they attempt the same updates. 

We have $\dist(Py,Pz)=\dist(y,z)$ unless $(i)$ the update was performed on $u$ in which case the distance decreases by $1$, or $(ii)$ a vertex (necessarily adjacent to $u$ in $H$) was included in $Py$ but not in $Pz$ in which case the distance increases by $1$. Thus, under this coupling
\begin{align*}
\E \left[\dist (Py, Pz)\right] & \leq 1-\frac{1}{|V(H)|}+\frac{1}{|V(H)|}\sum_{w\in N_H(u)}|p_y(w)-p_z(w)|\, . 
\end{align*}
Note that for $w\in N_H(u)$, $d_y(w)=d_z(u)+1$ and so as in the proof of Lemma~\ref{lemma:contraction} we have $|p_y(w)-p_z(w)|\leq \lam\zeta$. 
Thus
\[
W_1(P_y, P_z)\leq \E \left[\dist (Py, Pz)\right] 
\leq 
 1-\frac{1}{|V(H)|}+\frac{\Delta \lam\zeta}{|V(H)|}\, . \qedhere
\]
\end{proof}

\begin{remark}
Although we don’t need this fact, we note that Lemma~\ref{lemma:contraction} and Corollary~\ref{lem:mixingtime} imply that under Condition~\ref{condition2}, $P_{\lam,\zeta}$ has mixing time $O(n^2 \log n)$. Similarly Lemma~\ref{lem:contractionH} implies that if $\Delta(H)\lam\zeta$ is bounded away from 1 (from above) then $P_{H,\lam,\zeta}$ has mixing time $O(n \log n)$.
\end{remark}

\subsection{Concentration}
We now turn to the proof of Lemma~\ref{lem:fullconcentration} (under the weaker Condition~\ref{condition2}).  This requires the following concentration inequalities for $\nu_{H,\lam,\zeta}$  and $ \mu_{\lam,\zeta}$ that follow from the contraction properties of the Markov chains $P_{H,\lam,\zeta}$  and $P_{\lam,\zeta}$ proved in the previous section. 

Recall that we say a sequence of events $A_n$ holds with very high probability (wvhp) with respect to a probability measure $\P$ if $\mathbb P(A_n) = 1-n^{-\omega(1)}$.

\begin{lemma}\label{lem:concmuH}
Fix $\theta\in(0,1]$.
Let $H$ be a graph on $(1+o(1))n$ vertices with maximum degree at most $\Delta$. Let $\lam>0$, $\zeta\in(0,1]$ satisfy $1-\lam\zeta\Delta\geq \theta$.
If $S\sim \nu_{H,\lam,\zeta}$, then wvhp 
\[
\left ||S| - \mathbb{E}|S| \right| \leq \sqrt{\lambda n}\log n\, .
\]
\end{lemma}

Recall that we let $\Omega_n$ denote the set of all labeled graphs on vertex set $[n]$.

\begin{lemma}\label{lem:Lipschitzconcentration}
Let $f : \Omega_n \rightarrow \mathbb{R}$ be any $L$-Lipschitz function where $L=n^{O(1)}$. 
Then under Condition~\ref{condition2} we have for $m>1$
\[
\mathbb{P}_{\lambda,\zeta}\left(\left(|f(G) - \mathbb{E}(f)| \geq m\right)\land \left\{\Delta(G) \leq \lambda n + \frac{\sqrt{\lambda n}\log n}{100}\right\}\right) \leq 4\exp\left(-\frac{m^2}{48L^2\lam n^2/\xi + 16mL}\right).
\]
\end{lemma}

We first prove Lemma~\ref{lem:fullconcentration} assuming Lemmas~\ref{lem:concmuH} and~\ref{lem:Lipschitzconcentration} and then turn to the proofs of these two results. 

\begin{proof}[Proof of Lemma~\ref{lem:fullconcentration}]
Let $G$ denote a sample from $\mu_{\lam,\zeta}$. All probabilities will be with respect to $\mu_{\lam,\zeta}$ unless specified otherwise.
We first prove item \eqref{degconcentration}. Let
\begin{align}\label{eqdegdef}
\Omega_{\text{deg}}=\{G\in \Omega_n : \Delta(G) \leq \lambda n + \sqrt{\lambda n}\log n/100\}\, ,
\end{align}
 and note that
\begin{align}\label{eqn:degconcstart}
\mathbb{P}_{\lambda, \zeta}(\Delta(G) - \delta(G) \geq 4\sqrt{n\lam}\log n) &\leq 
\mathbb{P}_{\lambda,\zeta}(\{\Delta(G) - \delta(G) \geq 4\sqrt{n\lam}\log n\}\wedge \Omega_{\text{deg}})+\P_{\lambda,\zeta}(\overline{\Omega_{\text{deg}}})\, .
\end{align}
By Corollary~\ref{cor:SDdegbound}, $\P_{\lambda,\zeta}(\overline{\Omega_{\text{deg}}})\leq n^{-\omega(1)}$. Next observe that if $G\in \Omega_{\text{deg}}$, then the complement graph of $G$ has diameter at most $2$. Therefore, if $G\in \Omega_{\text{deg}}$ and $\Delta(G) - \delta(G) \geq 4\sqrt{n\lam}\log n$ then there exists $xy\notin G$ such that $|d_G(x)-d_G(y)|\geq 2\sqrt{n\lam}\log n$\, .
Letting $A_{xy} : = \{|d_G(x) - d_G(y)| \geq 2\sqrt{\lambda n} \log n \}$ for $x,y\in[n]$ it therefore suffices to show that $\mathbb{P}_{\lambda,\zeta}(\cup_{xy\notin G}A_{xy}\wedge \Omega_{\text{deg}})=n^{-\omega(1)}$. For $x,y\in [n]$ distinct let $B_{xy}=\{xy\notin G\}$. By a union bound we have
\begin{align}
\mathbb{P}_{\lambda,\zeta}(\cup_{xy\notin G}A_{xy}\wedge \Omega_{\text{deg}})  
= \mathbb{P}_{\lambda,\zeta}\left(\cup_{xy\in\binom{[n]}{2}}A_{xy}\wedge B_{xy}\wedge \Omega_{\text{deg}}\right) 
&\leq \sum_{xy \in \binom{[n]}{2}}\mathbb{P}_{\lambda,\zeta}(A_{xy} \wedge B_{xy}\wedge \Omega_{\text{deg}})\nonumber \\
& \leq  \binom{n}{2}\mathbb{P}_{\lambda,\zeta}(A_{uv} \wedge \Omega_{\text{deg}}|B_{uv})\, ,\label{eqDegNeq}
\end{align}
where $uv\in\binom{[n]}{2}$ is an arbitrary pair. It therefore suffices to show that $\mathbb{P}_{\lambda,\zeta}(A_{uv} \wedge \Omega_{\text{deg}}|B_{uv})=n^{-\omega(1)}$.
Let $G_{uv}$ denote the graph $G$ with vertices $u,v$ deleted. The proof of Lemma~\ref{lemLocalCond} shows that conditioned on the event $B_{uv}\wedge \{G_{uv}=H\}$, the pair of degrees $(d_G(u), d_G(v))$ has the same distribution as $(|S_1|, |S_2|)$ where $S_1, S_2$ are two independent samples from $\nu_{H,\lambda,\zeta}$. We therefore have
\begin{align}\label{eq:GuvH}\\
\mathbb{P}_{\lambda,\zeta}(A_{uv} \land \Omega_{\text{deg}}|B_{uv}, G_{uv}=H)\leq  \mathbb{P}_{H,\lambda,\zeta}(||S_1| - |S_2|| \geq 2\sqrt{\lambda n}\log n)\cdot \1_{\{\Delta(H)\leq \lambda n + \sqrt{\lambda n}\log n/100\}}\, .
\end{align}
We note that if $\Delta(H) \leq \lambda n + \sqrt{\lambda n}\log n/100$ then by Condition~\ref{condition2},
\[
\lambda \zeta \Delta(H) \leq (1+o(1))\zeta c^2< 1/2\, .
\]
We may therefore apply Lemma~\ref{lem:concmuH}, to conclude that  
\[
\mathbb{P}_{H,\lambda,\zeta}(||S_1| - |S_2|| \geq 2\sqrt{\lambda n}\log n) \leq n^{-\omega(1)}\, ,
\]
and so 
\(
\mathbb{P}(A_{uv} \land \Omega_{\text{deg}}|B_{uv}) \leq n^{-\omega(1)}
\) by~\eqref{eq:GuvH}. This completes the proof of item \eqref{degconcentration}.

Item (\ref{edgeconcentration}) follows immediately from Lemma~\ref{lem:Lipschitzconcentration} by choosing $f:\Omega_n \rightarrow \mathbb{R}$ given by $f(G) = |G|$ (a $1$-Lipschitz function) and recalling that $G\in \Omega_{\text{deg}}$ wvhp. 

We now turn to item~(\ref{triangleconcentration}). 
We again aim to apply Lemma~\ref{lem:Lipschitzconcentration}, however a naive application fails since the function $X(G)$ is not sufficiently Lipschitz on $\Omega=\Omega_n$.
To get around this we define a function $\widehat X$ that approximates $X$ and has smaller Lipschitz constant. Set $L:=\lambda^2 n\log^2n$ and 
let
\[
\Omega_{\text{Lip}} = \left\{G \in \Omega_{\text{deg}} : d_G(\{u,v\}) \leq L ~\forall~\text{distinct}~u,v\in [n]\right\} \,,
\]
which imposes a codegree condition in addition to the degree condition of $\Omega_{\text{deg}}$. 
We will show that the restriction of $X(\cdot)$ to $\Omega_{\text{Lip}}$ is $L$-Lipschitz i.e. for any $G,H \in \Omega_{\text{Lip}}$ such that $\dist(G,H) = s$, we have that $|X(G) - X(H)| \leq L s$. To see this, consider a shortest-path sequence $G = G_0,G_1,\ldots,G_s = H$ where each $G_i$ is either a subgraph of $G$ or a subgraph of $H$ and $\dist(G_i, G_{i+1})=1$  for all $i$. Since $\Omega_{\text{Lip}}$ is closed under taking subgraphs we have $G_i \in \Omega_{\text{Lip}}$ for all $i$, and the codegree condition ensures that $|X(G_i) - X(G_{i - 1}| \leq L$. We then have
\[
|X(G) - X(H)| \leq \sum_{i = 1}^{s}|X(G_i) - X(G_{i-1})| \leq L\cdot s\, .
\]
We now extend $X$ restricted to $\Omega_{\text{Lip}}$ to a function $\widehat X$ that is $L$-Lipschitz on the whole space $\Omega$. This is a standard exercise (any real-valued Lipschitz function on a subset of a metric space can be extended to a Lipschitz function on the whole space with the same constant) and  is known as the McShane--Whitney theorem~\cite{mcshane1934extension}. For completeness, we present the full argument. Define the function $\widehat{X}:\Omega \rightarrow \mathbb{R}$ where
\[
\widehat{X}(G) = \min_{G' \in \Omega_{\text{Lip}}}\left(X(G) + L\cdot \dist(G,G')\right).
\]
Since $X$ is $L$-Lipschitz on $\Omega_{\text{Lip}}$, one may observe that for $G \in \Omega_{\text{Lip}}$ we have $\widehat{X}(G) = X(G)$. Now fix a graph $G \in \Omega$ and let $G' \in \Omega_{\text{Lip}}$ be such that $\widehat{X}(G) = X(G') + L\cdot \dist(G,G')$. For any $H \in \Omega$ and $H' \in \Omega_{\text{Lip}}$, we have $\widehat{X}(H) \leq X(H') +L\cdot \dist(H,H')$. So in particular,
\begin{align*}
\widehat{X}(H) & \leq X(G') + L\cdot \dist(H,G') \\
& \leq X(G') + L\cdot (\dist(H,G) + \dist(G,G'))\\
& = \widehat{X}(G) + L\cdot \dist(H,G),
\end{align*}
where the second inequality uses the triangle inequality. This shows that $\widehat{X}(\cdot)$ is $L$-Lipschitz on $\Omega$. In particular, $\max_{G \in \Omega}\widehat{X}(G) \leq Ln^2\leq n^3$ and $\max_{G \in \Omega}X(G)\leq n^3$. 
It follows that
\begin{equation}\label{eqn:approxexpectation}
|\mathbb{E}_{\lambda,\zeta}[\widehat{X}(G)] - \mathbb{E}_{\lambda,\zeta}[X(G)]| \leq 2\mathbb{P}_{\lambda,\zeta}(\overline{\Omega_{\text{Lip}}})\cdot n^3 \leq 1,
\end{equation}
where we used the upper bound $\mathbb{P}_{\lambda,\zeta}(\overline{\Omega_{\text{Lip}}}) \leq n^{-\omega(1)}$, which follows from Lemma~\ref{lemStochDom}. Letting $\mathcal{Q} = \left\{\left|\widehat{X}(G) - \mathbb{E}_{\lambda,\zeta}[\widehat{X}(G)]\right| \geq n^2\lambda^{3/2}\log^2 n -1\right\}$ and recalling~\eqref{eqdegdef}, the definition of $\Omega_{\text{deg}}$, Lemma~\ref{lem:Lipschitzconcentration} gives us that 
\[
\mathbb{P}_{\lambda,\zeta}\left(\mathcal{Q}\land \Omega_{\text{deg}}\right) \leq n^{-\omega(1)}.
\]
Recall that if $G\in \Omega_{\text{Lip}}$ then $\widehat{X}(G) = X(G)$. Thus using~\eqref{eqn:approxexpectation}, we have
\begin{align*}
\mathbb{P}_{\lambda,\zeta}(|X(G) - \mathbb{E}_{\lambda,\zeta}[X(G)]| \geq n^2\lambda^{3/2}\log^2 n) & \leq \mathbb{P}_{\lambda,\zeta}(\mathcal{Q} \land\Omega_{\text{deg}}\land \Omega_{\text{Lip}}) + \mathbb{P}_{\lambda,\zeta}(\overline{\Omega_{\text{deg}}}) + \mathbb{P}_{\lambda,\zeta}(\overline{\Omega_{\text{Lip}}}) \\
& \leq \mathbb{P}_{\lambda,\zeta}\left(\mathcal{Q}\land \Omega_{\text{deg}}\right) + \mathbb{P}_{\lambda,\zeta}(\overline{\Omega_{\text{deg}}}) + \mathbb{P}_{\lambda,\zeta}(\overline{\Omega_{\text{Lip}}}) \\
&\leq n^{-\omega(1)} \,. \qedhere
\end{align*}
\end{proof}

It remains to prove Lemmas~\ref{lem:concmuH} and~\ref{lem:Lipschitzconcentration}, which we do this in the following two subsections. 

\subsection{Proof of Lemma~\ref{lem:concmuH}}
Recall from Definition~\ref{defn:PHlamzeta} the Markov chain $P=P_{H,\lam,\zeta}$ on state space $\Omega=2^{V(H)}$ which has stationary distribution $\nu=\nu_{H,\lam,\zeta}$. Let $(Y_i)_{i\in \N}$ be a run of chain $P$. Recall from Section~\ref{subsecMCMC} that for $y\in \Omega$ we let $\mathbb{P}_y(\cdot) = \mathbb{P}(\cdot~|Y_0 = y)$ and $\mathbb{E}_y[\cdot] = \mathbb{E}[\cdot~|Y_0 = y]$. Recall also that we let 
$\mathbb{P}_{\nu}(\cdot) = \mathbb{E}[\mathbb{P}_{Y_0}(\cdot)]$ and $\mathbb{E}_{\nu}[\cdot] = \mathbb{E}\left[\mathbb{E}_{Y_0}[\cdot]\right]$ where $Y_0$ is distributed according to $\nu$.

We aim to apply Theorem~\ref{thm:barbour2019longterm} with $f:\Omega\to \R$ given by $f(S)=|S|$. We note that $P$ updates a single vertex at a time and so $f$ is $1$-Lipschitz and item~\eqref{MC1} of Theorem~\ref{thm:barbour2019longterm} is satisfied. 

We set $\Omega_{\text{con}}=\Omega$ and $\Omega_{\text{typ}} = \{x \in \Omega : |x| \leq \lambda n + \sqrt{\lambda n}\log n\}$, and so item~\eqref{MC3} is trivially satisfied. For item~\eqref{MC2} we note that to transition from a state $y$ to a state $z\neq y$ we must either $(i)$ select a vertex $v\in y$ and then delete $v$ from $y$ or $(ii)$ select a vertex $v\notin y$ and then add $v$ to $y$. The first occurs with probability at most $|y|/n$ (the probability of selecting $v\in y$) and the second occurs with probability at most $\lam$ (an upper bound for the probability of removing $v$). Therefore, for $y\in \Omega_{\text{typ}}$
\[
\P(Py \neq y) \leq \frac{|y|}{n} + \lam\leq 3\lam\, .
\]
For item~\eqref{MC4} we note that by Lemma~\ref{lem:contractionH} and the assumption that $1-\lam\zeta\Delta\geq \theta$ and $|V(H)|=(1+o(1))n$, we have that $P$ is $(\theta/(2n))$-contractive on $\Omega$.
Recall that for $k\geq 1$ we let $\mathcal{A}_k = \{Y_i \in \Omega_{\text{typ}}~\forall\,  i<k\}$. Moreover we trivially have $e_k=0$ for all $k\geq 1$ where $e_k$ is as defined in Theorem~\ref{thm:barbour2019longterm}.
We may therefore apply Theorem~\ref{thm:barbour2019longterm} and conclude that for $y\in\Omega_{\text{typ}}$, $m>0$ and $k$ a positive integer
\begin{align}\label{eq:YkConc}
\mathbb{P}_y(\{||Y_k| - \mathbb{E}_y|Y_k|| > m\}\land \mathcal{A}_k) \leq 2\exp\left(-\frac{m^2}{24\lam n/\theta  + 4m}\right)\, .
\end{align}
Set $k=n^2$ and let $\kappa:=\E_{H,\lam,\zeta}|S|$ the expected size of a sample $S$ from $\nu=\nu_{H,\lam,\zeta}$.  By Lemma~\ref{lem:mixingtime} we have for $y\in\Omega_{\text{typ}}$,
\begin{align}
\left|\kappa - \mathbb{E}_{y}|Y_k|\right| \leq  \textup{diam}(\Omega)\left((1 - \theta/(2n))^{k} + \nu(\Omega\setminus \Omega_{\text{typ}}\right)<1\, ,
\end{align}
where for the last inequality we note that $\textup{diam}(\Omega)\leq |V(H)|=(1+o(1))n$ and apply Corollary~\ref{cor:SDdegbound} to bound $\nu(\Omega\setminus \Omega_{\text{typ}})=n^{-\omega(1)}$.
Returning to~\eqref{eq:YkConc} and setting $m=\sqrt{\lambda n}\log n-1$, it follows that for $y\in\Omega_{\text{typ}}$,
\[
\mathbb{P}_y(\{||Y_k| - \kappa| > m+1\}\land \mathcal{A}_k)\leq \mathbb{P}_y(\{||Y_k| - \mathbb{E}_y|Y_k|| > m\}\land \mathcal{A}_k)=n^{-\omega(1)}\, .
\]
Taking an expectation over $y$ distributed according to $\nu$ we then have
\[
\mathbb{P}_\nu(\{||Y_k| - \kappa| > m+1\}\land \mathcal{A}_k)\leq n^{-\omega(1)}\, .
\]
 Finally we note that by Corollary~\ref{cor:SDdegbound} applied to $Y_i$ (distributed according to $\nu$ under $\mathbb{P}_\nu$) and a union bound $\mathbb{P}_\nu(\overline{\mathcal{A}_k})\leq k\cdot \nu(\Omega\setminus \Omega_{\text{typ}})=n^{-\omega(1)}$. 
It follows that 
\[
\P_{H,\lam,\zeta}(\left ||S| - \mathbb{E}|S| \right| \geq \sqrt{\lambda n}\log n)=\mathbb{P}_\nu(||Y_k| - \mu| > m+1)=n^{-\omega(1)}
\]
as desired. \qed

\subsection{Proof of Lemma~\ref{lem:Lipschitzconcentration}}
The proof follows similar lines to that of 
Lemma~\ref{lem:concmuH}.
Recall from Definition~\ref{defn:Plamzeta} the Markov chain $P=P_{\lam,\zeta}$ on state space $\Omega=\Omega_n$, the set of all graphs on vertex set $[n]$. Recall also that $P$ has stationary distribution $\mu=\mu_{\lam,\zeta}$. Let $(Y_i)_{i\in \N}$ be a run of chain $P$. 

We again aim to apply Theorem~\ref{thm:barbour2019longterm}. We note that $P$ updates a single edge at a time and so item~\eqref{MC1} of Theorem~\ref{thm:barbour2019longterm} is satisfied.  
We set 
\[
\Omega_{\text{typ}}=\{y : \Delta(y)\leq \lam n + \sqrt{\lam n}\log n/10\} \text{ and }\Omega_{\text{con}}=\{y : \Delta(y)\leq \lam n + \sqrt{\lam n}\log n\}
\]
and note that item~\eqref{MC3} is trivially satisfied. For item~\eqref{MC2} we note that to transition from a state $y$ to a state $z\neq y$ we must either $(i)$ select an edge $e\in y$ and then delete $e$ from $y$ or $(ii)$ select an edge $e\notin y$ and then add $e$ to $y$. The first occurs with probability at most $|y|/\binom{n}{2}$ (the probability of selecting $e\in y$) and the second occurs with probability at most $\lam$ (an upper bound for the probability of removing $e$). Therefore, for $y\in \Omega_{\text{typ}}$
\[
\P(Py \neq y) \leq \frac{|y|}{\binom{n}{2}} + \lam\leq 3\lam\, .
\]
For item~\eqref{MC4} we note that by Lemma~\ref{lemma:contraction}, $P$ is $(\xi/n^2)$-contractive on $\Omega_{\text{con}}$.

Recall that for $k\geq 1$ we let $\mathcal{A}_k = \{Y_i \in \Omega_{\text{typ}}~\forall\,  i<k\}$ and we let 
\[e_k = \max_{y\in \Omega_{\text{typ}}^+}\mathbb{P}_{y}(Y_i\not \in \Omega_{\text{con}} \textup{ for some }i<k)\, .\]
Theorem~\ref{thm:barbour2019longterm} then shows that for  $y\in\Omega_{\text{typ}}$, $m>0$ and $k$ a positive integer
\begin{align}\label{eq:GkConc}
\mathbb{P}_y(\{|f(Y_k) - \mathbb{E}_y[f(Y_k)]| > m\}\land \mathcal{A}_k) \leq 2\exp\left(-\frac{m^2}{4L^2(3\lam n^2/\xi + 12k^3e_k^2) + 4mL(1 + 6ke_k)}\right)\, .
\end{align}

Our next goal is to upper bound $e_k$. We do so in the following claim which is stated in slightly more generality than immediately needed since it will come in handy again later. 

\begin{claim}\label{claim:exittime}
Let $0 \leq c_1 < c_2$ be constants such that $c_2 > 6c_1$ and let $k=O(n^3)$. If $y\in \Omega$ is such that $\Delta(y) \leq \lambda n + c_1\sqrt{\lambda n}\log n$. Then
\[
\mathbb{P}_y(\Delta(Y_i) \geq \lambda n + c_2\sqrt{\lambda n}\log n \textup{ for some }i<k) \leq kn^{-\omega(1)}\, . 
\]
\end{claim}

\begin{proof}
We condition on the event $Y_0=y$ so that $Y_t=P^ty$. Let $v \in [n]$ and note that $d_{y}(v) \leq \lambda n + c_1\sqrt{\lambda n}\log n$. Let $\{u_1,\ldots,u_{d_{y}(v)}\}$ be the neighbors of $v$ in $y$. Choose a step $t \leq k$. Let 
\[
U_t = \{i\in [d_{y}(v)]~|~(v,u_i)~\text{was not chosen for update in the first $t$ steps}\},
\]
and 
\[
V_t = \{u \in [n]~|~(v,u)~\text{was chosen for update in the first $t$ steps}\}.
\]
By Lemma~\ref{lemStochDom}, we first observe that 
\[
d_{Y_t}(v) \leq |U_t| + \text{Bin}\left(|V_t|,\frac{\lambda}{1 + \lambda}\right).
\]
Let $\mathcal{E}_t(v)$ denote the event $\left\{d_{Y_t}(v) \geq \lambda n + c_2\sqrt{\lambda n} \log n\right\}$. If $\mathcal{E}_t(v)$ occurs, then one of the following must have happened.
\begin{enumerate}
\item $|U_t| > e^{-t/\binom{n}{2}}\lambda n + (c_2/3)\sqrt{\lambda n}\log n$,
\item $|V_t| > (1-e^{-t/\binom{n}{2}})n + (c_2/3)\sqrt{n/\lambda}\log n $, or
\item $|V_t| \leq (1-e^{-t/\binom{n}{2}})n + (c_2/3)\sqrt{n/\lambda}\log n$ and $\text{Bin}\left(|V_t|,\frac{\lambda}{1 + \lambda}\right) > \lambda |V_t| + (c_2/3)\sqrt{\lambda n}\log n$.
\end{enumerate}
We bound the probabilities of each of these events separately. Consider the following setting: there are $m$ balls and $N$ bins out of which $N_0 < N$ are colored red. Each ball is placed uniformly and independently in one of the $N$ bins. Let $J(m,N,N_0)$ denote the number of red empty bins. We have the following that may be derived from Azuma's inequality (similar to e.g.~\cite{dubhashi2009concentration}, section $7.2$). For any $z \geq 0$ we have
\begin{equation}\label{eqn:ballsinbins}
\mathbb{P}\left(|J(m,N,N_0) - \mathbb{E}[J(m,N,N_0)]| > z\right) \leq \exp\left(-\frac{z^2}{10N_0}\right).
\end{equation}
We have that $|U_t|$ and $n -1- |V_t|$ are distributed as $J\left(t, \binom{n}{2},d_y(v)\right)$ and $J\left(t,\binom{n}{2},n-1\right)$, respectively. Moreover, since $t = O(n^3)$, we have have 
\[
\mathbb{E}[|U_t|] = \left(1 - \frac{1}{\binom{n}{2}}\right)^td_y(v) \leq e^{-t/\binom{n}{2}}\cdot d_{y}(v)
\]
and
\[
\mathbb{E}[|V_t|] =(n-1)\left[1-\left(1 - \frac{1}{\binom{n}{2}}\right)^t\right] \leq n(1-e^{-t/\binom{n}{2}}) + O(1).
\]
Thus~\eqref{eqn:ballsinbins} and the Chernoff bound give us
\begin{align*}
\mathbb{P}(|U_t| \geq e^{-t/\binom{n}{2}}\cdot \lambda n + (c_2/3)\sqrt{\lambda n}\log n) & \leq \exp\left(-\frac{(c_2/3 - c_1)^2\lambda n \log^2 n}{10 d_y(v) }\right)
\leq n^{-\omega(1)}, \\
\mathbb{P}(|V_t| \geq (1-e^{-t/\binom{n}{2}})n + (c_2/3)\sqrt{n/\lambda}\log n ) & \leq \exp\left(-\frac{(c_2/4)^2(n/\lambda)\log^2 n}{10 n}\right)
\leq n^{-\omega(1)}\, ,
\end{align*}
and
\begin{align*}
\mathbb{P}\left(\text{Bin}\left(|V_t|,\frac{\lambda}{1 + \lambda}\right) > \lambda |V_t| + (c_2/3)\sqrt{\lambda n}\log n\right) & \leq \exp\left(-\frac{(c_2/3)^2(\lambda n)\log^2 n}{3\lambda n}\right) \leq n^{-\omega(1)}.
\end{align*}
Therefore by union bound,
\[
\mathbb{P}_y(\Delta(Y_i) \geq \lambda n + c_2\sqrt{\lambda n}\log n \textup{ for some }i<k) \leq \sum_v\sum_{t \in [k-1]}\mathbb{P}(\mathcal{E}_t(v)) \leq kn^{-\omega(1)} \,. \qedhere 
\]
\end{proof}
Setting $k=n^3$ it follows that $e_k=n^{-\omega(1)}$ and so returning to~\eqref{eq:GkConc} we conclude that for $y\in\Omega_{\text{typ}}$ and $m>0$,
\begin{align}\label{eq:GkConcAgain}
\mathbb{P}_y(\{|f(Y_k) - \mathbb{E}_y[f(Y_k)]| > m\}\land \mathcal{A}_k) \leq 2\exp\left(-\frac{m^2}{24L^2\lam n^2/\xi + 8mL}\right)\, .
\end{align}

Let $\kappa:=\E_{\lam,\zeta}[f(G)]$ the expectation of $f(G)$ where $G$ is sampled from $\mu=\mu_{\lam,\zeta}$.  By Lemma~\ref{lem:mixingtime} we have for $y\in\Omega_{\text{typ}}$,
\begin{align}
\left|\kappa - \mathbb{E}_{y}[f(Y_k)]\right| \leq  L\cdot \textup{diam}(\Omega)\left((1 - \xi/n^2)^{k} + 2k e_k + \nu(\Omega\setminus \Omega_{\text{typ}}\right)<1,
\end{align}
where for the last inequality we note that $L=n^{O(1)}$ by assumption, $\textup{diam}(\Omega)\leq n^2$, $e_k=n^{-\omega(1)}$,  and we apply Corollary~\ref{cor:SDdegbound} to bound $\nu(\Omega\setminus \Omega_{\text{typ}})=n^{-\omega(1)}$. 
It follows that for $y\in\Omega_{\text{typ}}$ and $m>1$,
\[
\mathbb{P}_y(\{|f(Y_k) - \kappa| > m\}\land \mathcal{A}_k)\leq \mathbb{P}_y(\{|f(Y_k) - \mathbb{E}_y[f(Y_k)]| > m-1\}\land \mathcal{A}_k)\, .
\]
Applying the bound at ~\eqref{eq:GkConcAgain} and
taking an expectation over $y$ distributed according to $\mu$ we have
\begin{align}\label{eq:penConc}
\mathbb{P}_\mu(\{|f(Y_k) - \kappa| > m\}\land \mathcal{A}_k)\leq 2\exp\left(-\frac{m^2}{48L^2\lam n^2/\xi + 16mL}\right)\, .
\end{align}
Finally we apply Lemma~\ref{lem:Gibbsprobability} with 
\[
\tilde\Omega_{\text{typ}}=\{y : \Delta(y)\leq \lam n + \sqrt{\lam n}\log n/100\}
\]
and 
$\mathcal{Q}=\{|f(y)-\kappa|>m\}$ to conclude that
\[
\mathbb{P}_{\lambda,\zeta}\left\{\left(|f(G) - \mathbb{E}(f)| \geq m\right\}\land  \tilde{\Omega}_{\text{typ}}\right)=\mu\left( \mathcal{Q}\wedge \tilde{\Omega}_{\text{typ}}\right) \leq \frac{\mathbb{P}_\mu(\{|f(Y_k) - \kappa| > m\}\land \mathcal{A}_k)}{\min_{y \in \tilde{\Omega}_{\text{typ}}}\mathbb{P}_y(\mathcal{A}_k)}.
\]

Claim~\ref{claim:exittime} (applied with $c_1=1/100, c_2=1/10$) shows that $\min_{y \in \tilde{\Omega}_{\text{typ}}}\mathbb{P}_y(\mathcal{A}_k)=1-o(1)$. The result now follows from~\eqref{eq:penConc}.

\section{Applying the cluster expansion}
\label{secRegularity}

In this section we prove Lemma~\ref{lemOccTree}.   To do this we will show that if a graph $H$ is approximately regular with few short cycles, then the leading order terms in the cluster expansion of $\log \Xi_{H}(\lam,\zeta)$ (and thus also of its derivatives $\alpha_{H,\lam,\zeta}$ and $\rho_{H,\lam, \zeta}$) come from subgraph counts $T(H)$ where $T$ is a tree. The key point here is that for approximately regular graphs the number of copies of a small tree $T$ it contains is determined to first order by its (maximum) degree. This is the content of the next lemma.
 \begin{lemma}\label{lemtreecount}
 Let $H$ be a graph on $n$ vertices such that $\Delta =\Delta(H)=\omega(1)$ and $\delta= \delta(H) = (1+o(1))\Delta$.
 There exists $k=\omega(1)$ such that the following holds. If $\ell \leq k$ and $v\in V(G)$, then the number of trees on $\ell$ vertices in $H$ that contain $v$ is 
 \[
(1+o(1))\frac{(\Delta \ell)^{\ell-1}}{\ell!}\, .
 \]
 \end{lemma}
\begin{proof}
 Let $T$ be a labeled tree on the vertex set $[\ell]$ and let $x\in [\ell]$. Let $h_{T,x}$ denote the number of injective homomorphisms $f: T\to H$ such that $f(x)=v$. A simple induction shows that 
\[
(\delta-\ell)^{\ell-1}\leq h_{T,x}\leq \Delta^{\ell-1}\, .
\]
It follows that
if $k\to\infty$ sufficiently slowly then by the assumption on degrees 
\[
h_{T,x}= (1+o(1))\Delta^{\ell-1}\, .
\]
By Cayley's formula, there are $\ell^{\ell-2}$ possible trees $T$.
Summing over all $x\in [\ell]$ and all possible $T$ then dividing by $\ell!$ to account for the labeling completes the proof. 
\end{proof}

We will also need the following lemma.
\begin{lemma}[\cite{galvin2004phase}, Lemma 2.1]\label{lemConCount}
In a graph $H$ of maximum degree at most $\Delta$, the number of connected, induced subgraphs of order $t$ containing a fixed vertex $v$ is at most $(e\Delta)^{t-1}$.
\end{lemma}

Now define the function
\begin{equation}
    f(\Delta, \lam,\zeta):= \frac{1}{2 \zeta \Delta}W(\zeta \Delta \lambda) (2+W(\zeta \Delta \lambda)) \,,
\end{equation}
where we always consider $\Delta>0$ but allow $\lam, \zeta$ to take on complex values.
The following  facts about the function $f$ follow from simple calculations and the definitions of the functions $\alpha, \rho$ from~\eqref{eqalphadef},\eqref{eqrhodef}.
\begin{lemma}
\label{lemFderivs}
    The partial derivatives of $f(\Delta,\lam,\zeta)$ satisfy
    \begin{align}
       \lam  \frac{\partial}{\partial \lam} f(\Delta, \lam,\zeta) &= \alpha(\Delta,\lam,\zeta) \\
       -(1-\zeta) \frac{\partial}{\partial \zeta} f(\Delta, \lam,\zeta) &= \rho(\Delta,\lam,\zeta) \, . 
    \end{align}
    Moreover, if $|e\lam \zeta \Delta|< 1$ then
    \begin{equation}
f(\Delta,\lam,\zeta) = \sum_{\ell=1}^{\infty}\frac{(-\zeta)^{\ell-1}\ell^{\ell-2}}{\ell!} \Delta^{\ell-1}\lam^\ell \,.
    \end{equation}
\end{lemma}
\begin{proof}
The derivative of the Lambert-W function is $W'(x) = \frac{W(x)}{x(1+W(x))} $.  Applying this along with the chain rule gives
\begin{align}
  \lam   \frac{\partial}{\partial \lam} f(\Delta,\lam,\zeta) &= \frac{ W(\zeta \Delta \lam) }{\zeta \Delta(1+ W(\zeta \Delta \lam))   } + \frac{ W(\zeta \Delta \lam)^2} {\zeta \Delta (1+W(\zeta \Delta \lam))}  \\
    &=\frac{W(\zeta \Delta \lam)}{\zeta \Delta } = \alpha(\Delta, \lam,\zeta) \,.
\end{align}
The calculation for $ -(1-\zeta) \frac{\partial}{\partial \zeta} f(\Delta, \lam,\zeta) $ is similar.

Now recall that the Taylor series for the Lambert-W function around $x=0$ is
\begin{equation}
   W(x) = \sum_{\ell=1}^{\infty}\frac{(-1)^{\ell-1}\ell^{\ell-1}}{\ell!}x^{\ell} \,,
\end{equation}
and it is convergent for $|x|<e^{-1}$.
This gives, for $|\lam \zeta \Delta|<e^{-1}$,
\begin{align}
     \frac{\partial}{\partial \lam} f(\Delta,\lam,\zeta) &=\sum_{\ell=1}^{\infty}\frac{(-1)^{\ell-1}\ell^{\ell-1}}{\ell!}(\lam \Delta \zeta)^{\ell-1} \,, 
\end{align}
and integrating term by term gives
\begin{align}
    f(\Delta,\lam,\zeta) &=\sum_{\ell=1}^{\infty}\frac{(-\zeta)^{\ell-1}\ell^{\ell-2}}{\ell!} \Delta ^{\ell-1} \lam^\ell \,,
\end{align}
also valid for $|\lam \Delta \zeta|<e^{-1}$.
\end{proof}

The next lemma shows that $f(\Delta,\lam,\zeta)$ approximates the normalized log partition function $ \frac{1}{|V(H)|}\log \Xi_H(\lambda,\zeta)$ when $H$ satisfies the conditions of Lemma~\ref{lemOccTree} and 
$\lambda,\zeta$ are in a regime where cluster expansion can be applied. 

Recall from~\eqref{eqRxDef} that for $\gamma,\Delta>0$ we let
\[
     R(\gamma,\Delta)=\{(\lam,\zeta)\in \C \times \C: |1-\zeta|< 1 \textup{ and } e|\lam|(1+|\zeta|\Delta)< \gamma\}\, .
\]
\begin{lemma}\label{lem:Free}
Let $H$ be a graph satisfying the conditions of Lemma~\ref{lemOccTree}. Fix $\gamma\in[0,1)$ and let $(\lam,\zeta)\in R(\gamma,\Delta)$ where $\zeta$ is fixed. Then 
\[
\frac{1}{n}\log \Xi_H(\lambda,\zeta)= f(\Delta,\lam,\zeta)+o(|\lam|)\, .
\]
\end{lemma}
\begin{proof}
Let $k=\omega(1)$ be as in Lemma~\ref{lemOccTree} and note that we may assume the $k$ grows sufficiently slowly.
Applying  Lemma~\ref{lemClusterTail} we have
\begin{align}
    \label{eqClusterExpandLogXi}
   \frac{1}{n}\log \Xi_H(\lambda,\zeta) = \frac{1}{n} \sum_{\Gamma\in \cC(H)} \phi_{\zeta}(\Gamma) \lam^{|\Gamma|}  &= \frac{1}{n}\sum_{\Gamma\in \cC(G) : |\Gamma|<k}  \phi_{\zeta}(\Gamma) \lam^{|\Gamma|} + 
    O\left ( |\lam| \gamma^{k-1} \right) \\ 
    \label{eqSumTrunc}
    &= \frac{1}{n}\sum_{\Gamma\in \cC(G) : |\Gamma|<k} \phi_{\zeta}(\Gamma) \lam^{|\Gamma|} +o(|\lam|)  \,.
\end{align}
We now want to compare the sum in~\eqref{eqSumTrunc} to $f(\Delta, \lam,\zeta)$.

  Let $\ell\leq k$.
  Call a graph \emph{cyclic} if it contains at least one cycle. By Lemma~\ref{lemConCount} and the assumption on cycle counts in $H$ we conclude that 
  \begin{align}\label{eq:Hcycle}
  \text{$H$ has at most $n (M')^{\ell} \Delta^{\ell-2}$
  connected cyclic induced subgraphs on $\ell$ vertices}
  \end{align}
  for some $M'>0$ (first choose the cycle and then choose remaining subgraphs rooted at the cycle). It then follows from Lemma~\ref{lemtreecount}, the assumption $\Delta= \omega(1)$, and assuming $k$ grows sufficiently slowly compared to $\Delta$, that there are at least 
  \[
   (1+o(1))n\frac{\Delta^{\ell-1}\ell^{\ell-2}}{\ell!}- n(M')^{\ell}\Delta^{\ell-2} = (1+o(1))n\frac{\Delta^{\ell-1}\ell^{\ell-2}}{\ell!}
  \]
  induced subgraphs of $H$ on $\ell$ vertices that are trees. Clearly there are also at most $(1+o(1))n{\Delta^{\ell-1}\ell^{\ell-2}}/{\ell!}$ such subgraphs.

    We first estimate the contribution to the  sum in~\eqref{eqSumTrunc} from those clusters $\Gamma=(v_1,\ldots,v_\ell)$ such that $v_1,\ldots, v_\ell$ are all distinct and the induced subgraph $H[\{v_1,\ldots,v_\ell\}]$ is a tree; it will turn out that this is the dominant contribution. Let $\cC^\ast_\ell$ denote the collection of all such clusters. By the above estimate on induced tree counts, and noting that each tree corresponds to $\ell!$ clusters, we have
    \[
|\cC^\ast_\ell|=(1+o(1))n \Delta^{\ell-1}\ell^{\ell-2}.
    \]
 Additionally, if $\Gamma \in \cC^\ast_\ell$ then the incompatibility graph $H_\Gamma=H[\{v_1,\ldots,v_\ell\}]$ is a tree and therefore its only spanning subgraph is itself. Since $v_1,\ldots, v_\ell$ are distinct it follows that $\phi_{\zeta}(\Gamma)={(-\zeta)^{\ell-1}}/{\ell!}$ for all $\Gamma\in\cC^\ast_\ell $. Letting $\cC^{\ast}=\bigcup_\ell \cC^\ast_\ell$, we conclude that 
\begin{align}
    \frac{1}{n}\sum_{\Gamma\in \cC^\ast : |\Gamma|<k}  \phi_{\zeta}(\Gamma) \lam^{|\Gamma|}
 & = \sum_{\ell=1}^{k-1}\frac{(-\zeta)^{\ell-1}\ell^{\ell-2}}{\ell!}\Delta^{\ell-1}\lam^\ell +o(1)\cdot \sum_{\ell=1}^{k-1}\frac{|\zeta|^{\ell-1}\ell^{\ell-2}}{\ell!}\Delta^{\ell-1}|\lam|^\ell\\
    &=\sum_{\ell=1}^{k-1}\frac{(-\zeta)^{\ell-1}\ell^{\ell-2}}{\ell!}\Delta^{\ell-1}\lam^\ell +o(|\lam|)\, ,
\end{align}
where for the second equality we used that $\ell^\ell/\ell!\leq e^\ell$ and $|e\lam \zeta\Delta|\leq \gamma<1$.
On the other hand if $\Gamma=(v_1,\ldots,v_\ell)\in  \cC\backslash\cC^\ast$ then either $H[\{v_1,\ldots,v_\ell\}]$ is cyclic or $|\{v_1,\ldots,v_\ell\}|<\ell$. By ~Lemma~\ref{lemConCount} and \eqref{eq:Hcycle} we conclude that the number of such clusters is at most $O_\ell(1)n\Delta^{\ell-2}$ and so
\begin{align}
    \frac{1}{n}\sum_{\Gamma\in \cC\backslash\cC^\ast : |\Gamma|<k}  |\phi_{\zeta}(\Gamma)| |\lam|^{|\Gamma|}&=
    O_k(1) \sum_{\ell=2}^{k-1}\Delta^{\ell-2}|\lam|^\ell = O_k(1) |\lam|^2 =o(|\lam|)
\end{align}
where for the final equality we again assume that $k$ grows sufficiently slowly noting that $|\lam|=O(\Delta^{-1})$. 
Putting everything together and letting $k$ tend to infinity sufficiently slowly we conclude that
\begin{align}\label{eqLogXiExp}
    \frac{1}{n}\log \Xi_H(\lambda,\zeta)= \sum_{\ell=1}^{k-1}\frac{(-\zeta)^{\ell-1}\ell^{\ell-2}}{\ell!} \Delta^{\ell-1}\lam^\ell + o(|\lam|)\, . 
\end{align}
From Lemma~\ref{lemFderivs}, we have
\begin{align}
    f(\Delta,\lam,\zeta) = \sum_{\ell=1}^{\infty}\frac{(-\zeta)^{\ell-1}\ell^{\ell-2}}{\ell!} \Delta^{\ell-1}\lam^\ell  = \sum_{\ell=1}^{k}\frac{(-\zeta)^{\ell-1}\ell^{\ell-2}}{\ell!} \Delta^{\ell-1}\lam^\ell + o(|\lam|)\,,
\end{align}
and so the result follows. 
\end{proof}

    We now prove Lemma~\ref{lemOccTree}.
\begin{proof}[Proof of Lemma~\ref{lemOccTree}]
We are given that  $\Delta = \Delta(H) \to \infty$, that $\delta(H) = (1+o(1)) \Delta$, and that for some fixed $\theta>0$, $\gamma\in [0,1)$, $\zeta\in (0,1]$ we have $\theta\leq  e\lam\zeta \Delta \leq \gamma$.

We will now consider complex pairs $(\bl, \bz)$, using the bold font to denote complex parameters, and reserving $\lam,\zeta$ for the real parameters in the statement of the lemma.

Fix $\gamma<\gamma_1<1$ and let
 \begin{align}
     R=R(\Delta,\gamma_1)=\{(\bl,\bz): |1-\bz|< 1 \textup{ and } e|\bl|(1+|\bz|\Delta)< \gamma_1\}\, .
 \end{align}
Now let
\begin{align}
  r_\lam &= \frac{1}{2}\left[\frac{\gamma_1}{e(1+\zeta\Delta)}-\lam\right]  \quad \text{ and } \quad
  r_{\zeta} = \frac{1}{2}\min\left\{\zeta, \frac{\gamma_1}{e\lam \Delta}-\frac{1}{\Delta}-\zeta\right\} \,.
\end{align}
These will be radii for disks in the complex plane around the parameters $\lam$ and $\zeta$.  Since $\gamma,\zeta$ are both fixed constants and $\lam = \Theta(\Delta^{-1})$ by the assumptions of the lemma, $r_{\lam}, r_{\zeta}>0$ and $r_{\lam} =\Theta(\lam)$ and $r_{\zeta} =\Theta(1)$. Recall that  $\overline B_{r}(z)$ denotes the closed disk of radius $r$ centered at $z$. By the choices above, 
\[
B:=(\{\lam\}  \times \overline B_{r_{\zeta}}(\zeta)) \cup (\overline B_{r_\lam}(\lam) \times\{\zeta\}) \subseteq R\, 
\]
for $n$ sufficiently large. Let 
\begin{equation}
    g_H(\bl,\bz)= \frac{1}{|V(H)|}\log \Xi_H(\bl,\bz) \,.
\end{equation}
We will show that the following hold:
\begin{enumerate}
    \item $g_H(\bl,\bz)$ is an analytic function of both $\bl$ and $\bz$ on $R$.
    \item 
    \begin{equation}
    \label{eqfgclose}
        \sup_{(\bl,\bz)\in B}|g_H(\bl,\bz) - f(\Delta,\bl,\bz)|  = o(\lam) \,.
    \end{equation}
\end{enumerate}
The first follows immediately  from Lemma~\ref{lemClusterTail}. The second follows from Lemma~\ref{lem:Free} noting that if $(\bl,\bz)\in B$ then $|\bl|=O(\lam)$. 

Given these two facts, we write, using Lemmas~\ref{lemDerivativeIdentities} and~\ref{lemFderivs} for the first line and Lemma~\ref{lemAnalyticDerivativesBound} and~\eqref{eqfgclose} for the second,
\begin{align}
    |\alpha_{H,\lam,\zeta} - \alpha(\Delta,\lam, \zeta) | &= \left| \lam \frac{\partial}{\partial \lam} g_H(\lam,\zeta) -  \lam \frac{\partial}{\partial \lam} f(\Delta,\lam,\zeta)  \right | \\
    &\le  \frac{\lam}{r_\lam} \cdot o(\lam) =o(\lam) \,.
\end{align}
Since $\alpha(\Delta,\lam, \zeta) =  \Theta(\Delta^{-1}) = \Theta(\lam)$, we have $\alpha_{H,\lam,\zeta}=(1+o(1))\alpha(\Delta,\lam, \zeta)$.

Similarly,
\begin{align}
    |\rho_{H,\lam,\zeta} - \rho(\Delta,\lam, \zeta) | &= \left| -(1-\zeta) \frac{\partial}{\partial \zeta} g_H(\lam,\zeta) + (1-\zeta)  \frac{\partial}{\partial \zeta} f(\Delta,\lam,\zeta)  \right | \\
    &\le \frac{1-\zeta}{r_{\zeta}} \cdot o(\lam) = o(\lam) \,.
\end{align}
Since $\rho(\Delta,\lam, \zeta) = \Theta(\Delta^{-1})= \Theta(\lam)$, we have $\rho_{H,\lam,\zeta}=(1+o(1))\rho(\Delta,\lam, \zeta)$. 
This completes the proof.
\end{proof}

\section{Lower tails via  partition functions}
\label{secLowerTailPartition}

\newcommand{\open}{{\mathcal{X}}}

In this section we prove Lemma~\ref{lem:LTtoZ}, relating the lower-tail probability $\P_p( X \le \eta \E X)$ to the partition function $Z(\lam,\zeta)$.   We then prove our results for lower tails in $G(n,m)$ in a similar way.

\subsection{Point probability estimates}
\label{secPointProbs}

We first will need some weak lower bounds on the probability that a sample from $\mu_{\lam,\zeta}$ has exactly $M$ edges and at most $T$ triangles, when these statements typically hold to first order.

\begin{lemma}
    \label{lemTransferGnm}
    Fix $c >0$, $\zeta \in[0, 1]$. Let $\eps>0$. Suppose $\lam = (1+o(1))c/\sqrt{n}$ and that whp for $G \sim \mu_{\lam,\zeta}$,
    \begin{align}
        |G| &= (1+o(1)) M, &
        X(G) = (1+o(1)) T,\\
    \end{align}
    where $M = \Theta(n^{3/2})$ and either $\zeta = 1$ and $T=0$ or else $T = \Theta(n^{3/2})$. Set $T_0 = T - \varepsilon n^{3/2}$. Then
    \begin{equation}
        \P_{\lam,\zeta} \left(\left(|G| = M\right) \wedge\left( T_0 \leq X(G) \leq T\right) \right) = \exp( - o(n^{3/2})  ) \,.
    \end{equation}
\end{lemma}

\begin{remark}
    The assumed bounds on $M$ and $T$ are superfluous. Indeed, it is not too difficult to show that for every fixed $c$ and $\zeta$ the number of edges in $G$ is typically on the order of $n^{3/2}$. Similarly, the number of triangles $X(G)$ is also on the order of $n^{3/2}$ unless $\zeta=1$ in which case  $X(G)=0$ with probability $1$. This paper already contains a proof of these facts whenever $c < \bar c(\eta)$ (where $\eta,\zeta,c$ satisfy \eqref{eqZetaDef}). Since we do not rely on this lemma in any other setting we omit the proof for other choices of $c,\zeta$.
\end{remark}

To prove Lemma \ref{lemTransferGnm} we first show that there are many graphs with the desired number of triangles and approximately the desired number of edges.

\begin{claim}\label{clmLoweringTriangleCount}
    Suppose that $G$ is an $n$-vertex graph with $X(G) = (1+o(1))T$ where $X(G)>T$ and $|G| = (1+o(1))M$. Set $k = \lfloor M\cdot (X(G)-T + n^{1.1})/X(G) \rfloor$. There are at least $(1-o(1))\binom{|G|}{k}$ graphs $H \subseteq G$ satisfying $T - o(n^{3/2}) \leq X(H) \leq T$.
\end{claim}

\begin{proof}
    Let $H \subseteq G$ be a uniformly random subgraph of $G$ with $|G|-k$ edges (equivalently, $H$ is obtained by deleting a uniformly random set of $k$ edges). The parameter $k$ was chosen so that whp slightly more than $X(G)-T$ triangles are deleted. Indeed, let $Y = X(G) - X(H)$ be the number of deleted triangles. We prove this with a second-moment calculation. We have
    \[
    \E [Y] = (1+o(1)) X(G) \frac{3k}{|G|} = (1+o(1)) 3(X(G)-T+n^{1.1}).
    \]
    Since every triangle in $G$ shares an edge with fewer than $3n$ other triangles, we have
    \[
    \var[Y] \leq (1+o(1))X(G) 3n \frac{k}{|G|} = (1+o(1)) 3n (X(G)-T+n^{1.1}) \leq 2n\E [Y].
    \]
    Chebychev's inequality now implies
    \[
    \P \left(|Y-\E [Y]| \geq \frac{1}{2}\E [Y]\right) \leq \frac{4\var [Y]}{(\E [Y])^2} \leq \frac{8n\E [Y]}{(\E[Y])^2} \leq \frac{8}{n^{0.1}}.
    \]
    Hence, whp, $5(X(G)-T+n^{1.1}) \geq Y \geq X(G)-T$ and $T \geq X(H) \geq (1-o(1))T$.
    
    As a consequence there are $(1-o(1)) \binom{|G|}{k}$ graphs $H \subseteq G$ satisfying $T \geq X(H) \geq T-o(n^{3/2})$, as claimed. 
\end{proof}

\begin{claim}\label{clmDownwardsConstruction}
    Suppose that $G$ is an $n$-vertex graph with $X(G) = (1+o(1))T$ triangles and $M' = (1+o(1))M$ edges. There are at least $(1-o(1))\binom{M'}{M}$ graphs $H \subseteq G$ with $M$ edges that also satisfy
    $X(H) \geq X(G) - o(n^{3/2})$.
\end{claim}

\begin{proof}
    Let $H$ be a uniformly random subgraph of $G$ with $M$ edges. Since there are $\binom{M'}{M}$ choices for $H$ it suffices to show that whp $X(G)-X(H) = o(n^{3/2})$.
    Indeed, the expected number of triangles in $G$ that are not in $H$ is at most $3X(G) (M'-M)/M' = o(n^{3/2})$. Markov's inequality gives us that whp $X(G)-X(H) = o(n^{3/2})$, as claimed. 
\end{proof}

Our next goal is to prove an analogue of Claim \ref{clmDownwardsConstruction} in the case $M' \leq M$. Here we must assume slightly more about the graph $G$. Namely, we assume that its degree is not too large and that a substantial fraction of potential edges would not form a triangle if added to the graph.

\begin{defn}
    For a graph $G$, we write $\open(G)$ for the number of \textit{open} potential edges in $G$, i.e., those vertex pairs $\{x,y\}$ that are not edges in $G$ and would not form a triangle if the edge $xy$ were added to $G$.
\end{defn}

We begin by observing that with high probability, a constant fraction of potential edges in a graph sampled from $\mu_{\lambda,\zeta}$ are open.
\begin{claim}\label{clmManyOpenEdges}
    Fix $c>0$ and $\zeta \in [0,1]$, and suppose that $\lam = (1+o(1))c/\sqrt{n}$ and $\alpha = \frac{1}{3}e^{-c^2}$. Then we have
    \[
    \mathbb{P}_{\lambda,\zeta}\left(\left(\open(G) \geq \alpha n^2\right)\land \left(\Delta(G) \leq 2\lambda n\right)\right) = 1 - \exp\left(-\Omega(\sqrt{n})\right).
    \]
\end{claim}
\begin{proof}
    Set $p = \lam/(1+\lam)$. Lemma~\ref{lemStochDom} asserts that $\mu_{\lam,\zeta}$ is stochastically dominated by $G(n,p)$, so it suffices to prove that 
    \[
    \P_p\left(\left(\open(G) \geq \alpha n^2\right)\land \left(\Delta(G) \leq 2\lambda n\right)\right) = 1 - \exp\left(-\Omega(\sqrt{n})\right).
    \]
    To begin, note that
    \[
    \E_p [\open(G)] = \binom{n}{2}(1-p^2)^{n-2} = (1 + o(1))\frac{n^2}{2}e^{-c^2}.
    \]
    Observe that $\open(G)$ is a $2n$-Lipschitz function of the edge set of $G$. Therefore, Freedman's inequality (see e.g. \cite{freedman1975tail} Proposition $2.1$)  gives us
    \[
    \P_{\lambda,\zeta}\left( \open(G)< \alpha n^2 \right) \leq \P_p \left( \open(G)< \alpha n^2 \right) \leq \exp \left(- \frac{(n^2 e^{-c^2}(1/2-1/3))^2}{12n^2pn^2} \right) = \exp(-\Omega(\sqrt{n})).
    \]
Corollary~\ref{cor:SDdegbound}  gives us that $\P_{\lambda,\zeta}\left(\Delta(G) > 2\lambda n\right) \leq \exp(-\Omega(\sqrt{n}))$. A union bound gives our desired claim.
\end{proof}

\begin{claim}\label{clmUpwardsConstruction}
    Let $\alpha,D>0$ be fixed. Suppose that $G$ is an $n$-vertex graph with $M' = (1-o(1))M$ edges. Suppose further that $\open(G) \geq \alpha n^2$ and that $\Delta(G) \leq D \sqrt{n}$. Then there are at least $(1-o(1))(\alpha n^2\!/2)^{M-M'}\!/(M-M')!$ graphs $H \supseteq G$ satisfying $V(H)=V(G)$, $|H| = M$, and $X(H) \leq X(G)$.
\end{claim}

\begin{proof}
    Consider the following process that constructs an $M$-edge graph $H_{M} \supseteq G$. Set $H_{M'} = G$. For each $i=M',M'+1,\ldots,M-1$, given $H_i$ construct $H_{i+1}$ as follows. If $\open(H_i) \geq \alpha n^2\!/2$ then let $e_{i+1}$ be a uniformly random open potential edge in $H_i$. Otherwise let $e_{i+1}$ be a uniformly random edge not in $E(H_i)$. Set $H_{i+1} = H_i + e_{i+1}$. We note that since $M=O(n^{3/2})$ there are at least $\alpha n^2\!/2$ choices for the edge $e_{i+1}$ at each step and so the number of graphs $H_M$ constructed in this way is at least $(\alpha n^2\!/2)^{M-M'}\!/(M-M')!$.

    We claim that whp $X(H_{M}) \leq X(G)$. We first observe that $H_{M}$ contains a triangle that is not in $G$ only if for some $M' \leq i < M$ we have $\open(H_i) < \alpha n^2\!/2$. We further observe that for every such $i$, $\open(H_{i+1}) \geq \open(H_i) - (1+2\Delta(H_i))$. Proceeding inductively we conclude that $\open(H_{i+1}) \geq \open(H_{M'}) - (M-M')(1 + 2\Delta(H_{M}))$. By assumption $\open(H_{M'}) = \open(H) \geq \alpha n^2$ and so $\open(H_{i+1}) < \alpha n^2\!/2$ only if $\Delta(H_{ M}) = \Omega(n^2\!/(M-M')) = \omega(\sqrt{n})$. Hence, it suffices to prove that whp $\Delta(H_{M})=O(\sqrt{n})$. Since by assumption $\Delta(G) = O(\sqrt{n})$, this will follow from the (stronger) claim that whp $\Delta(H_{M} \setminus G) = o(\sqrt{n})$.

    For every vertex $v$, let $Y_v$ be the number of edges $e_i$ such that $v \in e_i$. We note that $d_{H_{M} \setminus G}(v) = Y_v$. It suffices to prove that $Y_v = o(\sqrt{n})$ with probability at least $1-o(1/n)$. We begin by calculating $\E[Y_v]$. For $M' < i \leq M$ let $Y_{v,i}$ indicate the event $\{v \in e_i\}$. Then $Y_v = \sum_{i=M'+1}^{M}Y_{v,i}$. Since each $e_i$ is chosen uniformly at random from a set of $\Omega(n^2)$ edges, of which fewer than $n$ are incident to $v$, it follows that $\P(Y_{v,i}=1) = O(1/n)$. Therefore $Y_v$ is stochastically dominated by $\bin(M-M',O(1/n))$. A Chernoff  bound now implies that $\P(Y_v > 2\E[Y_v]+\log^2(n)) = o(1/n)$. Since $2\E[Y_v]+\log^2(n) = o(\sqrt{n})$ this implies the claim.
\end{proof}

We now prove Lemma \ref{lemTransferGnm}.
\begin{proof}
    By assumption whp $|G| = (1+o(1))M$ and $X(G) = (1+o(1))T$. Hence, by the pigeonhole principle, there exist $M' = (1+o(1))M$ and $T' = (1+o(1))T$ such that $\P[|G|=M' \land X(G) = T'] \geq n^{-5}$. Let $\mathcal S$ be the collection of $n$-vertex, $M'$-edge graphs $H$ satisfying $X(H) = T'$, $\open(H) \geq \alpha n^2$, and $\Delta(H) \leq 2c \sqrt{n}$. Claim \ref{clmManyOpenEdges} implies that $\P[G \in \mathcal S] \geq n^{-6}$.

    We now show  a lower bound on the number of graphs with their triangle count in $[T_0,T]$. We consider two cases. First, if $T' \leq T$ then set $\mathcal S' = \mathcal S$ and $M''=M'$. Note that since $T'=(1+o(1))T$ in this case all graphs in $\mathcal S'$ have a triangle count in $[T-o(n^{3/2}),T]$ and $M'' = (1+o(1))M$ edges.

    For the second case we consider $T' > T$. Set $k = \lfloor M(T'-T+n^{1.1})/T'\rfloor$ and $M''=M'-k$. By Claim \ref{clmLoweringTriangleCount} each graph in $\mathcal S$ contains at least $(1-o(1)) \binom{M'}{k}$ graphs with $M''$ edges and triangle count in $[T-o(n^{3/2}),T]$. Let $\mathcal S'$ be the collection of these graphs. Since each graph in $\mathcal S'$ is contained in fewer than $\binom{n^2\!/2}{k}$ graphs in $\mathcal S$ there holds
    \[
    |\mathcal S'| \geq (1-o(1)) |\mathcal S| \binom{M'}{k} / \binom{n^2\!/2}{k} = |\mathcal S| \lam^k \exp (-o(n^{3/2})).
    \]
    
    In either case we have now constructed $\mathcal S'$ as a collection of graphs with $M''=(1+o(1))M$ edges, triangle count in $[T-o(n^{3/2}),T]$, and maximum degree at most $2c\sqrt{n}$. Additionally in either case there holds
    \[
    |\mathcal S'| \geq |\mathcal S| \lam^{M'-M''}\exp(-o(n^{3/2})). 
    \]
    Now, let $\mathcal S''$ be the set of $n$-vertex $M$-edge graphs with triangle count in $[T_0,T]$. To obtain a lower bound on $|\mathcal S''|$, once again we consider two cases. If $M'' \geq M$ then, by Claim \ref{clmDownwardsConstruction} each graph in $\mathcal S'$ contains at least $(1-o(1))\binom{M''}{M''-M}$ graphs in $\mathcal S''$. Since each graph in $\mathcal {S''}$ is contained in at most $\binom{n^2\!/2}{M''-M}$ graphs in $\mathcal S'$, we conclude that
    \[
    |\mathcal S''| \geq (1-o(1)) |\mathcal S'| \binom{M''}{M''-M} / \binom{n^2\!/2}{M''-M} = |\mathcal S'| \lam^{M''-M} \exp(-o(n^{3/2})).
    \]
    Similarly, if $M>M''$ then Claim \ref{clmUpwardsConstruction} implies that
    \[
    |\mathcal S''| \geq (1-o(1))|\mathcal S'| \frac{(\alpha n^2\!/2)^{M-M''}}{(M-M'')!\binom{M}{M-M''}} = |\mathcal S'| \lam^{M''-M} \exp(-o(n^{3/2})).
    \]
    By a pigeonhole argument there exists some $T'' \in [T-o(n^{3/2}),T]$ such that there are at least $|\mathcal S''|/n^2$ graphs in $\mathcal S''$ with exactly $T''$ triangles. Let $\mathcal S''' \subseteq \mathcal S''$ be the collection of these graphs. By the calculations above there holds
    \[
    |\mathcal S'''| \geq \frac{1}{n^2}|\mathcal S''| \geq \lam^{M'-M}|\mathcal S|\exp(-o(n^{3/2})).
    \]
    To complete the proof we observe that
    \begin{align*}
        \P(|G| = M \land X(G) \in [T_0,T]) & \geq \P[G \in \mathcal S''']
        = \frac{\lam^M (1-\zeta)^{T''}}{Z(\lam,\zeta)} |\mathcal S'''|\\
        & \geq \frac{\lam^M (1-\zeta)^{T''}}{Z(\lam,\zeta)} \lam^{M'-M} |\mathcal S| \exp(-o(n^{3/2}))\\
        & \geq \frac{\lam^{M'} (1-\zeta)^{T''}}{Z(\lam,\zeta)} |\mathcal S| \exp(-o(n^{3/2}))\\
        & \geq \frac{\lam^{M'} (1-\zeta)^{T'}}{Z(\lam,\zeta)} |\mathcal S| \exp(-o(n^{3/2}))\\
        & = \P (G \in \mathcal S) \exp(-o(n^{3/2})) \geq \exp(-o(n^{3/2})),
    \end{align*}
    which implies the claim.
\end{proof}

\subsection{Lower tails for $G(n,p)$ via the partition function}
\label{subsecGnpLTpartition}

Here we prove Lemma~\ref{lem:LTtoZ}. 
\begin{proof}[Proof of Lemma~\ref{lem:LTtoZ}]
We fix $\eta \in (0,1)$ and let $\zeta \in (0, 1)$ be defined via~\eqref{eqZetaDef}. 

 We first prove the upper bound.  Let $T = \eta \E_p X$. Then
 \begin{align}
   \P_p \left(X \le \eta \E X\right) &= (1-p)^{\binom{n}{2}} \sum_{G : X(G) \leq T} \lam ^{|G|}  \\
      &\le  (1-p)^{\binom{n}{2}} (1-\zeta)^{- T} \sum_{G : X(G)\leq T} \lam ^{|G|} (1-\zeta)^{ X(G)}\\
      &\le
       (1-p)^{\binom{n}{2}} (1-\zeta)^{- T} Z(\lam,\zeta),
\end{align}
and so
\begin{align}
   \frac{1}{n^{3/2}} \log  \P_p \left(X \le \eta \E X\right)  &\le  - \frac{c}{2} - \log(1-\zeta)   \eta \frac{c ^3}{6}  + \frac{1}{n^{3/2}} \log Z(\lam, \zeta) + o(1) \,.
\end{align}
 Now we prove the lower bound. As before let $T = \eta \E_p X$. Fix $\eps>0$  and let $T_0 = T - \eps n^{3/2}$. 
 \begin{align}
      \P_p \left(X \le \eta \E X\right) &\ge  (1-p)^{\binom{n}{2}} \sum_{G : X(G) \in [T_0,T]} \lam ^{|G|}  \\
      &\ge  (1-p)^{\binom{n}{2}} (1-\zeta)^{- T_0} \sum_{G : X(G) \in [T_0,T]} \lam ^{|G|} (1-\zeta)^{ X(G)} \\
      &= (1-p)^{\binom{n}{2}} (1-\zeta)^{- T_0} Z(\lam,\zeta) \mu_{\lam,\zeta}(X(G) \in [T_0,T]) \\
      &= (1-p)^{\binom{n}{2}} (1-\zeta)^{- T} Z(\lam,\zeta) e^{- \eps n^{3/2} - o(n^{3/2})} \,,
 \end{align}
where the last line follows from Lemma~\ref{lemTransferGnm},  whose conditions are satisfied by  Lemma~\ref{lemDensityEstimate} under Condition~\ref{ConditionThm1}.  This gives
\begin{align}
   \frac{1}{n^{3/2}} \log  \P_p \left(X \le \eta \E X\right)  &\ge  - \frac{c}{2} - \log(1-\zeta)   \eta \frac{c ^3}{6}  + \frac{1}{n^{3/2}} \log Z(\lam, \zeta) -\eps  + o(1) \,,
\end{align}
and since $\eps$ was arbitrarily small, this completes the lower bound.

For the second assertion of the lemma, note that 
\begin{align}
 \P_{p} ( \mathcal{E} | X(G) \le T )
 &=
 \frac{(1-p)^{\binom{n}{2}}}{\P_{p}(X(G) \le T)}\sum_{G\in \mathcal{E}: X(G)\leq T}\left(\frac{p}{1-p} \right)^{|G|}\\
 &\leq 
  \frac{(1-p)^{\binom{n}{2}}}{\P_{p}(X(G) \le T)}(1-\zeta)^{-T}\sum_{G\in \mathcal{E}: X(G)\leq T}\lam ^{|G|}(1-\zeta)^{X(G)}\\
  &\leq
   \frac{(1-p)^{\binom{n}{2}}}{\P_{p}(X(G) \le T)}(1-\zeta)^{-T} Z(\lam,\zeta)\mu_{\lam,\zeta}(\mathcal{E})\\
   &\leq 
   e^{o(n^{3/2})}\cdot e^{-\eps n^{3/2}}\, ,
\end{align}
where for the final inequality we used~\eqref{eq:PFnTransfer} and the assumed upper bound on $\mu_{\lam,\zeta}(\mathcal{E})$. The result follows.
\end{proof}

\subsection{Lower tails for $G(n,m)$ via the partition function}
\label{subsecGnmLowerTailpartition}

To prove Theorem~\ref{thmLowerTailGnm} on the lower-tail rate function for $G(n,m)$ we will proceed similarly to Section~\ref{subsecGnpLTpartition}, only we now must choose $\lam$ and $\zeta$ simultaneously  to achieve the desired number of edges and triangles in expectation.

\begin{proof}[Proof of Theorem~\ref{thmLowerTailGnm}]
We are given $b$ and $\eta$ satisfying $b^2(1-\eta)<W(2/e)$.  Let $c = b e^{(1-\eta)b^2}$ and $\zeta= 1-\eta$. We will assume $m =(1+o(1) \frac{b}{2} n^{3/2}$ and $\lam = (1+o(1))c/\sqrt{n}$. Note that $c, \eta$ satisfy $c < \overline c(\eta)$ (Condition~\ref{ConditionThm1}) and so we will be able to apply Lemma~\ref{lemZlambetaEst}.

We note a couple of identities we will use below. From the choices of parameters, we have 
\begin{equation}
\label{eqGnmExpId}
    \eta \E_m X = (1+o(1)) \eta \frac{b^3}{6} n^{3/2} \,.
\end{equation}
We also have
\begin{equation}
\label{eqBinomCoeffLamid}
    \frac{1}{n^{3/2}} \log\left( \frac{\lam^{-m}}{ \binom{\binom{n}{2}}{m}}\right) =  -\frac{b}{2} \log(ec/b) + o(1)  =  - (1-\eta)\frac{b^3}{2} -\frac{b}{2} +o(1)\, .
\end{equation}
Next, we have 
\begin{equation}
    \label{eqLogZinGnmAsym}
    \frac{1}{2} \cdot\left[  \frac{W(2 \zeta c^2)^{3/2} +3W(2\zeta c^2)^{1/2}}{3\sqrt{2\zeta}}\right] = \frac{b^3(1-\eta)}{3} + \frac{b}{2}  \,.
\end{equation}

We first prove the upper bound for Theorem~\ref{thmLowerTailGnm}.  Let $T = \eta \E _m X$.  
\begin{align}
    \P_m ( X \le T) &\le   \frac{\lam^{-m} (1-\zeta)^{- T}}{ \binom{\binom{n}{2}}{m}} Z(\lam,\zeta)  \mu_{\lam,\zeta} ( |G| =m \wedge X(G) \le T)  \\
    &\le  \frac{\lam^{-m} (1-\zeta)^{- T}}{ \binom{\binom{n}{2}}{m}} Z(\lam,\zeta)  \, ,
\end{align}
and so, applying Lemma~\ref{lemZlambetaEst} and the identities~\eqref{eqGnmExpId},~\eqref{eqBinomCoeffLamid}, and~\eqref{eqLogZinGnmAsym}, 
\begin{align}
    \frac{1}{n^{3/2}} \log \P_m ( X \le T) &\le -(1-\eta) \frac{b^3}{2} -\frac{b}{2}    -\log (1-\zeta)   \eta \frac{b^3}{6}  +\frac{b^3(1-\eta)}{3} + \frac{b}{2}   +o(1) \\
    &=  - \frac{b^3}{6} (1- \eta+\eta\log \eta)     +o(1) \,.
\end{align}

For the lower bound we proceed in a similar way.  Let $\eps>0$ be fixed and set $T_0 = T - \eps n^{3/2}$.  Then
\begin{align}
     \P_m ( X \le T) &\ge   \frac{\lam^{-m} (1-\zeta)^{- T_0}}{ \binom{\binom{n}{2}}{m}} Z(\lam,\zeta)  \mu_{\lam,\zeta} \left( |G| =m \wedge X(G) \in[T_0, T] \right) \,, 
\end{align}
and so, applying Lemma~\ref{lemTransferGnm} for a lower bound on $\mu_{\lam,\zeta} ( |G| =m \wedge X(G) \in[T_0, T])$, and applying the same identities as above, we have 
\begin{align}
     \frac{1}{n^{3/2}} \log \P_m ( X \le T) &\ge - \frac{b^3}{6} (1- \eta+\eta\log \eta)  - \eps + o(1) \,.
\end{align}
Since $\eps$ was arbitrary, we have proved the lower bound.  
\end{proof}

\section{Convergence in cut norm}
\label{secCutNorm}

In this section we  prove Theorem~\ref{thmLowerTailStructure}.  To do so, we first  prove a result  for the measure $\mu_{\lam, \zeta}$ then transfer to the lower-tail conditional measure by applying Lemma~\ref{lem:LTtoZ} (or Lemma~\ref{lem:LTtoZhardcore} in the case $\eta=0$). 

We are given $\eta \in [0,1)$ and  $c < \overline c(\eta) $. We define $\zeta$ via~\eqref{eqZetaDef}.  We assume $p = (1+o(1)) c/\sqrt{n}$, set $\lam = p/(1-p)$, and set $q= \sqrt{\frac{W( 2 \zeta c^2)}{2 \zeta } }\cdot n^{-1/2} $ as in~\eqref{eqQform}. Note that this implies $q = \Theta(p) = \Theta(\lam)$.

For a graph $G$, let $A_G$ be its adjacency matrix, and let $M = \mathbb{E}_{\lambda,\zeta}[A_G]$. Note that by Lemma~\ref{lemDensityEstimate} and symmetry, each off-diagonal entry of $M$ is $(1+o(1))q$ and each diagonal entry is $0$. It follows that $\|M-qJ\|_{\boxempty}\leq nq+o(qn^2)=o(qn^2)$. By the triangle inequality, it therefore suffices to prove Theorem~\ref{thmLowerTailStructure} with $M$ in place of $qJ$.

Fix $\eps>0$ and let
\[
\mathcal{Q}=\mathcal{Q}_n : = \left\{\|A_G - M\|_{\boxempty}  \geq \eps qn^2\right\} .
\]
For brevity, let us denote the lower-tail event by $\mathcal{T} = \left\{X(G) \leq \eta\mathbb{E}_p X(G)\right\}$. To prove the theorem it suffices, by the Borel--Cantelli Lemma, to show that
\begin{align}\label{eq:BC}
\sum_{n=1}^\infty \P_p(\mathcal{Q}_n \mid \mathcal{T})<\infty \, .
\end{align}
Let
\begin{align*}
\mathcal{B} & : = \left\{\Delta(G) \leq \lambda n + \sqrt{\lambda n}\log n/100\right\}\,.
\end{align*}
We will show for $\lambda$ and $\zeta$ satisfying Condition~\ref{ConditionThm1}, 
\begin{equation}\label{eqn:cutnormETS}
\mathbb{P}_{\lambda, \zeta}\left(\mathcal{Q} \wedge \mathcal{B}\right) = \exp\left(-\Omega_{\eps}(n^{3/2})\right).
\end{equation}
 Indeed, since Lemma~\ref{lem:LTtoZ} would then give us
\begin{align*}
\mathbb{P}_p\left(\mathcal{Q} \wedge \mathcal{B}\mid \mathcal{T}\right) \leq \exp\left(-\Omega_{\eps}(n^{3/2})\right) \,,
\end{align*}
and Corollary~\ref{cor:SDdegbound} gives us
\[
\mathbb{P}_p(\mathcal{Q}\mid\mathcal{T}) \leq \mathbb{P}_p\left(\mathcal{Q} \wedge \mathcal{B}|\mathcal{T}\right) + \mathbb{P}_p(\overline{\mathcal{B}}\mid\mathcal T) = \exp\left(-\Omega_{\eps}(\log^2 n)\right)\,,
\]
this would imply~\eqref{eq:BC} and thereby complete the proof. We now proceed to proving~\eqref{eqn:cutnormETS}.

For any fixed $x,y \in \{0,1\}^n$, $x^T(A_G - M)y$ is $2$-Lipschitz in $G$ with respect to the Hamming metric, with $\mathbb{E}_{\lambda,\zeta}[x^T(A_G - M)y] = 0$, so Lemma~\ref{lem:Lipschitzconcentration} gives us that 
\begin{equation}\label{eqn:x,yinP}
\mathbb{P}_{\lambda,\zeta}\left(\left(|x^T(A_G - M)y| \geq m\right)\land \mathcal {B}\right) \leq \exp\left(-\Omega\left(\frac{m^2}{n^2\lam + m}\right)\right).
\end{equation}
Setting $m = \eps qn^2$ gives us that this probability is at most $\exp\left(-\Omega_\eps( n^{3/2})\right)$. Taking a union bound we have
\begin{align*}
\mathbb{P}_{\lambda,\zeta}\left(\mathcal{Q}\land \mathcal{B}\right) & \leq \sum_{x,y\in \{0,1\}^n} \mathbb{P}_{\lambda,\zeta}\left(\left(|x^T(A_G - M)y| \geq m\right)\land \mathcal {B}\right)
\leq \exp\left(-\Omega_\eps( n^{3/2})\right)\, .
\end{align*}
as desired.

\section{Phase transitions}
\label{secPhaseTransition}

In this section we prove Corollaries~\ref{corPhaseTransitionGnp} and~\ref{corGnmPhase}.  In both cases, the proof structure is the same: we have  formulas for the rate functions $\varphi_\eta(c)$ and $\widehat \varphi_\eta(b)$ given in Theorems~\ref{thmLowerTailAsymp} and~\ref{thmLowerTailGnm}, valid for small enough $c$ and $b$ respectively.  These formulas are analytic functions of their arguments on the entire positive real axis.  We then prove lower bounds on $\varphi_\eta(c)$ and $\widehat \varphi_\eta(b)$ that contradict the above formulas for large enough $c, b$.  By uniqueness of analytic continuation, there must be a non-analytic point of $\varphi_\eta$ and $\widehat \varphi_\eta$.

The lower bound for $\widehat \varphi_\eta(b)$ is very simple:  the lower-tail probability is at least the probability $G(n,m)$ is bipartite. This suffices to prove a phase transition occurs for every $\eta \in [0,1)$. 
\begin{proof}[Proof of Corollary~\ref{corGnmPhase}]
   By Theorem~\ref{thmLowerTailGnm}, we have $\widehat\varphi_\eta(b) = -\frac{b^3}{6}( 1-\eta +\eta \log \eta)$ for small enough $b$.  On the other hand, a simple lower bound on the lower-tail probability is the probability $G(n,m)$ is bipartite which is at least $2^{- (1+o(1))m}$;  this gives the lower bound $\widehat \varphi_\eta(b) \ge - \frac{\log 2}{2} b $.  Since the first formula is a cubic in $b$ and the second is linear, the functions must cross at some $b \in (0, \infty)$ and thus $\widehat\varphi_\eta(b)$  must have a non-analytic point.  
\end{proof}

We need a slightly more sophisticated lower bound on $\varphi_\eta(c)$ to prove the existence of a phase transition in the $G(n,p)$ model for as large a range of $\eta$ as possible.  We will use the `mean-field' bound (see e.g.~\cite{kozma2023lower,jain2019mean}):  lower bounds on the lower-tail event given by random graphs in which edges are independent but may have different probabilities.  Roughly, in our setting the mean-field bound states that $\log \P_p( X \le \eta \E X)$ is at least the negative of the relative entropy of a random graph $\mathbf G$ with respect to $G(n,p)$, where $\mathbf G$ has independent edges and the expected number of triangles of $\mathbf G$ is at most the target $\eta \E_p X$. 

To prove Corollary~\ref{corPhaseTransitionGnp} we will state a convenient form of the general mean-field lower bound from~\cite{kozma2023lower}, then use a result from~\cite{zhao2017lower} to show that a simple choice of random graph $\mathbf G$ (a two-part block model) gives a bound that contradicts the formula from Theorem~\ref{thmLowerTailAsymp} for large values of $c$.

Define the mean-field bound $\Phi_{n,p}(\eta)$ for lower tails for triangles as follows.  We  consider vectors $\mathbf q$ of probabilities indexed by the $\binom{n}{2}$ edges of the complete graph on $n$ vertices; further let $i_p(q) = q \log \frac{q}{p} + (1-q) \log \frac{1-q}{1-p}$ denote the Bernoulli relative entropy between $q$ and $p$. Then 
\begin{align}
    \label{eqMeanField}
  \Phi_{n,p}(\eta) := \min \left  \{ \sum_{e \in \binom{n}{2}} i_p(\mathbf q_e) : \mathbf q\in [0,1]^{\binom{n}{2}}, \E_{\mathbf q} X \le \eta \E_p X    \right \}\,, 
\end{align}
where $\E_{\mathbf q}$ denotes expectation relative to the random graph  in which each potential edge $e\in\binom{[n]}{2}$ is included independently with probability $\mathbf q_e$. Kozma and Samotij~\cite{kozma2023lower} state the following straightforward mean-field lower bound (their main contribution is a matching upper bound in the $p=\omega(n^{-1/2})$ regime). 
\begin{lemma}[\cite{kozma2023lower}]
\label{lemMeanFieldBound}
 Fix $\eta \in [0,1)$. For every $\eps>0$, there exists $C>0$ so that if $p \ge C/\sqrt{n}$,
 \begin{align}
     \log \P_p(X \le \eta \E X) \ge  -(1+\eps) \Phi_{n,p}(\eta -\eps) \,.
 \end{align}
\end{lemma}
Next, a simple upper bound on $\Phi_{n,p}(\eta)$ (and hence a lower bound on the lower-tail probability) comes from taking $\mathbf q_e = \eta^{1/3} p$ for all $e$, yielding 
\begin{align}
    \Phi_{n,p}(\eta) \le \binom{n}{2} \left( \frac{p}{3} \eta^{1/3} \log \eta +(1- \eta^{1/3} p) \log \frac{1- \eta^{1/3} p}{1-p}  \right)  =: \Phi_{n,p}^{\mathrm{RS}}(\eta) \, ,
\end{align} where the notation RS stands for `replica symmetry', following the terminology of~\cite{lubetzky2017variational,zhao2017lower}.  Letting $p = (1+o(1)) c/\sqrt{n}$, and defining $h(x) := x \log x - x +1$,
we have 
\begin{align}
    \frac{1}{n^{3/2}}\Phi_{n,p}^{\mathrm{RS}}(\eta) = \frac{c}{2} h \left ( \eta^{1/3} \right)    +o(1) \,. 
\end{align}

Corollary~\ref{corPhaseTransitionGnp} will now follow from two facts. The first states that the formula in  Theorem~\ref{thmLowerTailAsymp} approaches $-\frac{c}{2} h \left ( \eta^{1/3} \right) $ as $c \to \infty$.  

\begin{lemma}
\label{lemClimtBound}
   Let $G(\eta,c)$ denote the formula on the RHS of~\eqref{eqMainThmAsymptotics}.  Then
   \begin{align}
       \lim_{c \to \infty} \frac{G(\eta,c)}{c} = -\frac{1}{2} h \left ( \eta^{1/3} \right)   \, .
   \end{align}
\end{lemma}

The second fact is that for $\eta$ small enough, in the $c \to \infty$ limit with $p= c/\sqrt{n}$, the mean-field bound $ \Phi_{n,p}(\eta) $ is strictly less than  the replica symmetric bound. This is a result of Zhao~\cite[Theorem 2.6]{zhao2017lower}.
\begin{lemma}[\cite{zhao2017lower}]
\label{lemZhaoBound}
   For $0 \le \eta< \eta_\ell $, 
   \begin{equation}
       \lim_{c \to \infty} \frac{\Phi_{n,c/\sqrt{n}}(\eta)}{c n^{3/2}} < -\frac{1}{2} h \left ( \eta^{1/3} \right)  \,, 
   \end{equation}
   where $\eta_\ell \approx .0091$ is given explicitly in~\cite{zhao2017lower} as the maximum of a single variable function on the unit interval.
\end{lemma}
Zhao proves Lemma~\ref{lemZhaoBound} is  by taking $\mathbf q$ supported on two values; that is, the random graph $\mathbf G$ is a block model with two parts and a different edge probability for crossing and interior edges,  and further conjectures that this two-part block model is indeed optimal for this range of $\eta$.

With these lemmas we now prove Corollary~\ref{corPhaseTransitionGnp}.  
\begin{proof}
    We need to show that for $\eta \in [0, \eta_\ell)$, there is $c$ large enough so that $\varphi_\eta(c) > G(\eta,c)$, which establishes the existence of a non-analyticity of $\varphi_\eta(c)$ via the uniqueness of analytic continuation.  By Lemma~\ref{lemMeanFieldBound}, Lemma~\ref{lemZhaoBound}, and the continuity of $h(x)$, we have that for some $\eps>0$ and all $c$ large enough,
    \begin{equation}
        \frac{\varphi_\eta(c)}{c} > \eps - \frac{1}{2}h \left ( \eta^{1/3} \right)  \,. 
    \end{equation}
    On the other hand, Lemma~\ref{lemClimtBound} shows that for every $\eps>0$ and $c$ large enough,
    \begin{equation}
        \frac{G(\eta,c)}{c} < \eps - \frac{1}{2} h (\eta^{1/3}) \, , 
    \end{equation}
    and this completes the proof. 
\end{proof}

We end by proving Lemma~\ref{lemClimtBound}.
\begin{proof}
In the proof of Claim~\ref{claim:zeta} we showed that the limit $\ell:=\lim_{c\to\infty}\zeta c^2$ exists and satisfies 
\[
\left( \frac{ W(2 \ell)}{2 \ell} \right)^{3/2}= \eta \iff \ell=-\frac{\log (\eta^{1/3})}{ \eta^{2/3}}\, .
\]
We conclude that
\begin{align}
 \lim_{c \to \infty} \frac{G(\eta,c)}{c}&=\lim_{c \to \infty}\frac{1}{2} \left[  \frac{W(2 \zeta c^2)^{3/2} +3W(2\zeta c^2)^{1/2}}{3\sqrt{2\zeta c^2}}-\frac{\log(1-\zeta) \eta c^2}{3}-1\right]\\
 &=
 \frac{1}{2} \left[  \frac{W(2\ell)^{3/2} +3W(2\ell)^{1/2}}{3\sqrt{2\ell}}+\frac{\eta\ell}{3}-1\right]= -\frac{1}{2} h \left ( \eta^{1/3} \right) \,. \qedhere
\end{align}
\end{proof}

\section*{Acknowledgments}

The authors thank Wojciech Samotij for many enlightening conversations. MJ is supported by a UK Research and Innovation Future Leaders Fellowship MR/W007320/2.  WP supported in part by NSF grant DMS-2348743. AP is supported by an NSERC Discovery grant. MS is supported by NSF grant DMS-2349024.

\appendix
\section{Proof of Lemma~\ref{lemClusterTail}}
\label{secPinnedCluster}

We use the tree--graph bound of Penrose~\cite{penrose1967convergence}; see~\cite[Section 4]{faris2010combinatorics} for a  discussion.  

\begin{lemma}[\cite{penrose1967convergence}] \label{lempenrosetree}
Consider a graph $H=(V,E)$ equipped with complex edge weights $\{w_e\}_{e\in E}$ satisfying $|1+w_e|\leq 1$ for all $e$. Then 
\[
\left|  \sum_{\substack{A \subseteq E:\\ \textup{$(V,A)$ connected}}} \prod_{e\in A} w_e \right| \leq  \sum_{\substack{A \subseteq E:\\ \textup{$(V,A)$ tree}}} \prod_{e\in A} |w_e| \, .
\]
\end{lemma}

With this we prove Lemma~\ref{lemClusterTail}. Given a graph $H$, let $\cT(H)$ denote the set of labeled spanning trees of $H$.  Moreover we let $\cT_k$ denote the set of all labeled trees on the vertex set $\{1,\ldots, k\}$. 
\begin{proof}[Proof of Lemma~\ref{lemClusterTail}]
Fix $k\geq 1$. We will show that
\begin{align}\label{fixedclusterphibd}
 \sum_ {\substack{\Gamma:  |\Gamma|= k}} |\phi_\zeta(\Gamma)|  \leq  n e^k k^{-2} (1+|\zeta|\Delta)^{k-1}\, . 
\end{align}
Given a cluster $\Gamma$ of size $k$, we identify the vertex set of $H_\Gamma$ with $\{1,\ldots, k\}$.
By Lemma~\ref{lempenrosetree}, we have
\begin{align}
\sum_ {\substack{\Gamma:  |\Gamma|= k}} |\phi_\zeta(\Gamma)|
 &=
  \frac{1}{k!} 
  \sum_ {\substack{\Gamma:  |\Gamma|= k}}
\left|  \sum_{\substack{A \subseteq E(H_\Gamma):\\(V(H_\Gamma), A) \textup{ connected}}}  \prod_{e\in A} w_e
  \right|\\
   &\leq
  \frac{1}{k!} 
  \sum_ {\substack{\Gamma:  |\Gamma|= k}}
\sum_{T\in \cT_k}\mathbf 1_{T\in \cT(H_\Gamma)} \prod_{e\in T}|w_e| \\
 &=
 \frac{1}{k!} 
  \sum_{T\in \cT_k}
  \sum_ {\substack{\Gamma:  |\Gamma|= k}}\mathbf 1_{T\in \cT(H_\Gamma)} \prod_{e\in T}|w_e| \, .\label{eqpenroseapp}
 \end{align}
 Fix $T\in \cT_k$ and $T'\subseteq T$.
 We will construct a cluster $\Gamma$ such that $T\in \cT(H_\Gamma)$ (so in particular $|\Gamma|=k$) and $w_e=-\zeta$ if $e\in T'$ and $w_e=-1$ otherwise. We do so iteratively as follows. First we select a vertex $v\in V(G)$ and place it in the first coordinate of $\Gamma$ (there are $n$ choices for $v$). Now suppose we have filled coordinates $i_1=1, \ldots, i_j$ of $\Gamma$ with vertices $v_{i_1}=v, \ldots, v_{i_j}\in V(G)$ respectively. There exists $i_{j+1}\in [k]\backslash \{i_1, \ldots, i_j\}$ such that $i_{j+1}$ is adjacent to one of $\{i_1, \ldots, i_j\}$ in the graph $T$. Without loss of generality assume $e=\{i_1,i_{j+1}\}\in T$. 
 
 If $e\in T'$, we must select $v_{i_{j+1}}\in V(G)$ such that $v_{i_{j+1}}$ is adjacent to $v_{i_1}$ in $H$ and place it in the $i_{j+1}$th coordinate of $\Gamma$. There are at most $\Delta$ choices for such a $v_{i_{j+1}}$ in $V(H)$. If $e\notin T'$ then we must set $v_{i_{j+1}}=v_{i_1}$. Continuing iteratively we see that
 \[
 \sum_ {\substack{\Gamma:  |\Gamma|= k}}\mathbf 1_{T\in \cT(H_\Gamma)} \prod_{e\in T}|w_e|\leq \sum_{j=0}^{k-1}\binom{k-1}{j}|\zeta|^j\Delta^j=(1+|\zeta|\Delta)^{k-1}\, ,
 \]
 where the binomial coefficient $\binom{k-1}{j}$ accounts for the number of choices $T'\subseteq T$ such that $|T'|=j$.

 By Cayley's formula $|\cT_k|=k^{k-2}$ and so we conclude from~\eqref{eqpenroseapp}
 \[
 \sum_ {\substack{\Gamma: |\Gamma|= k}} |\phi_\zeta(\Gamma)|
 \leq 
  n\frac{1}{k!}  k^{k-2} (1+|\zeta|\Delta)^{k-1}
  \leq n e^k k^{-2} (1+|\zeta|\Delta)^{k-1}\, ,
 \]
 where for the final inequality we used that $k!\geq(k/e)^k$ for all $k\geq 1$.
This establishes~\eqref{fixedclusterphibd}.
We conclude that
 \begin{align*}
\sum_ {\substack{\Gamma\in\cC(G): \\ |\Gamma|\geq k}} \left|\phi_\zeta(\Gamma) \lam^{|\Gamma|}\right| 
\leq 
en|\lam|\sum_{j\geq k}  \left|e(1+|\zeta|\Delta)\lam \right|^{j-1} \leq en|\lam| \gamma^{k-1}(1-\gamma)^{-1}\, .
\end{align*}
 where for the final inequality we used that $e|\lam|(1+|\zeta|\Delta))< \gamma$ where $\gamma<1$. 
\end{proof}

\end{document}